\documentclass[11pt]{article}
\usepackage[a4paper,margin=24mm]{geometry}
\usepackage[T1]{fontenc}
\usepackage{lmodern}
\usepackage{amsmath,amssymb,amsthm,mathtools,array,longtable,booktabs,microtype,fancyvrb}
\usepackage[colorlinks=true,linkcolor=blue,citecolor=blue,urlcolor=blue]{hyperref}
\newtheorem{theorem}{Theorem}[section]
\newtheorem{lemma}[theorem]{Lemma}
\newtheorem{corollary}[theorem]{Corollary}
\theoremstyle{remark}
\newcommand{\Z}{\mathbb Z}

\newcommand{\calL}{\mathcal L_n}
\newcommand{\Lam}{\Lambda_n}
\newcommand{\M}{\mathcal M}

\allowdisplaybreaks[3]
\title{The Golomb--Welch conjecture for arbitrary centre sets}
\author{Bocong Chen$^1$,  Jing Huang$^2$ and Gaojun Luo$^3$\footnote{E-mail addresses: {\it
bocongchen@foxmail.com (B. Chen),
jhuangmath@foxmail.com (J. Huang),
gaojun\_luo@nuaa.edu.cn (G. Luo).
}}
}

\date{\small
$1.$ School of Mathematics, South China University of Technology, Guangzhou 510641, China\\
$2.$
School of Mathematics and Information Science,
Guangzhou University, Guangzhou 510006, China\\
$3.$ School of Mathematics, Nanjing University of Aeronautics and Astronautics, Nanjing 210016, China\\
}

\begin{document}
\maketitle
\begin{abstract}
	We prove that for all $n\ge3$ and $r\ge2$, the Lee ball $B_n(r)$ does not
	tile $\Z^n$ by translations, resolving the Golomb--Welch conjecture for
	arbitrary centre sets and completing the existence classification of
	perfect Lee codes in $\Z^n$.
	The proof combines translation-invariant averaging, boundary contact
	identities, and integer moment inequalities with coefficient comparisons
	in stable symmetric functions.
	Consequently, we obtain finite local obstructions, strict Banach density
	gaps for packings and coverings, and the nonexistence of exact measurable
	Lee-ball towers.
	In the lattice specialization, our result classifies the attainability of
	the Lee bound for even-degree abelian Cayley graphs, excludes normal
	topological coverings and Galois graph coverings for Lee-ball completions,
	and gives a strict node bound for lattice cubature as well as the
	nonexistence of subgroup-indexed orthogonal exponential bases on $B_n(r)$.
\end{abstract}



\section{Introduction}\label{sec:introduction}

For $\mathbf x,\mathbf y\in\Z^n$, the Lee distance  
$d_L(\mathbf x,\mathbf y)$ is $\|\mathbf x-\mathbf y\|_1=\sum_{i=1}^n|x_i-y_i|$.
 For integers $a\ge0$ and $b$, set $\binom ab=0$ when $b<0$ or $b>a$.
For $j\in\mathbb Z_{\ge0}$, write
$(x)_0=1$ and $(x)_j=x(x-1)\cdots(x-j+1)$ for $j\ge1$;
these are falling factorials. Whenever its upper argument may be negative,
$\binom{x}{2}$ denotes the polynomial $x(x-1)/2$, rather than a
zero-truncated counting function.
For $r\ge0$, put
\[
 B_n(r)=\{\mathbf x\in\Z^n:\|\mathbf x\|_1\le r\},
 \qquad
 V_r(n)=|B_n(r)|
 =\sum_{j=0}^{\min\{n,r\}}2^j\binom nj\binom rj.
\]
A perfect $r$-error-correcting Lee code is a set
$\mathcal C\subseteq\Z^n$ for which every lattice point lies at
Lee distance at most $r$ from exactly one element of $\mathcal C$.
Equivalently,
\begin{equation}\label{eq:tiling}
 \Z^n=\bigsqcup_{\mathbf c\in\mathcal C}
       (\mathbf c+B_n(r)).
\end{equation}
Such a code attains the sphere-packing bound. Its error balls give a
translational tiling; replacing each lattice point of a ball by its
centred unit cube gives the corresponding tiling by Lee polyominoes.
The tiling is a lattice tiling if, after translation, its centre set is
a subgroup of $\Z^n$. A periodic centre set is a finite union of cosets
of a finite-index subgroup and need not itself be a subgroup.

The Lee metric entered coding theory through the study of nonbinary
transmission errors~\cite{Lee}. Golomb and Welch connected perfect Lee
codes with polyomino packings and constructed lattice examples for
$n=1,2$ at every positive radius, and for $r=1$ in every
dimension~\cite{GW1968,GW}. They conjectured that these exhaust the
positive parameter pairs: no tiling~\eqref{eq:tiling} exists for
$n\ge3$ and $r\ge2$. 
Its geometric form permits arbitrary centres, whereas its lattice
specialization also admits a finite-group description.

Early progress came from geometric obstructions. The asymptotic
argument of Golomb and Welch, with the interpretation clarified by
Horak and Kim, excludes sufficiently large radii in every fixed
dimension $n\ge3$~\cite{GW,HK}. Post~\cite{Post} and
Lepist\"o~\cite{Lepisto} obtained explicit bounds in the periodic or
large finite-alphabet setting. Horak and Kim~\cite{HK} extended these
bounds to arbitrary centres and developed a linear-programming method.
In a complementary direction, Gravier, Mollard, and Payan~\cite{GMP}
settled dimension three, and Horak~\cite{HorakTilings} established
nonexistence in dimensions three through five.
Szab\'o~\cite{Sz} also studied tilings by Lee polyominoes of unequal
radii in dimensions three and four. These results excluded large
regions of the parameter space, while leaving the interaction of
dimension and radius unresolved in general.

The lattice problem led to a second line of attack. A lattice tiling
exists precisely when there is a homomorphism from $\Z^n$ onto a
finite abelian group whose restriction to $B_n(r)$ is bijective.
Horak and Gro\v sek~\cite{HG} used finite-group invariants to exclude
radius-two lattice tilings in dimensions $3$ through $12$.
Kim~\cite{Kim} introduced a symmetric-polynomial obstruction for
certain dimensions with prime Lee-ball cardinality.
Zhang and Zhou~\cite{ZZ} developed group-ring and character methods
for radii two and three. Leung and Zhou~\cite{LZ} then excluded
radius-two lattice tilings in every dimension $n\ge3$; their result
also excludes arbitrary-centre tilings when $2n^2+2n+1$ is prime.
Qureshi~\cite{Qureshi} obtained further polynomial criteria for
linear perfect Lee codes. More recently, Xiao and Zhou~\cite{XZ}
determined the coefficients of the associated $Q$-polynomials,
extended the group-ring formulation to arbitrary radii, and developed
an algorithm combining the resulting nonexistence conditions.

The distinction between general and lattice centres is substantial.
Greenfeld and Tao~\cite{GTStructure} proved structural theorems for
translational tilings, including a quantitative periodic-tiling theorem
in dimension two. Their subsequent counterexample to the periodic
tiling conjecture~\cite{GTCounterexample} shows that a finite tile in
high dimension can tile the integer lattice without admitting any
periodic tiling. These results concern general tiles, not Lee balls,
but rule out imposing periodicity as a general reduction.
Symbolic dynamics provides another framework: a tiling indicator
satisfies a finite-range equation, so the tilings form a subshift of
finite type. Kari~\cite{Kari} studied the algebraic and dynamical
structure of such systems, including their connection with the
Golomb--Welch problem.
We prove the conjecture in its arbitrary-centre form.
\begin{theorem}\label{thm:main}
For all integers $n\ge3$ and $r\ge2$, there is no tiling of
$\Z^n$ by translates of the Lee ball $B_n(r)$.
\end{theorem}
Theorem~\ref{thm:main} requires no linearity, periodicity, or density
assumption on the centre set. Together with the classical
constructions, it shows that a perfect Lee code in $\Z^n$ exists
for positive integers $n,r$ if and only if $n\le2$ or $r=1$.
This classifies the admissible parameters, not the individual codes.
This existence classification parallels the classification of
perfect-code parameters over prime-power alphabets in the Hamming
metric~\cite{Tietavainen,ZL,Heden}.

We state the consequences first for arbitrary centre sets and then
for lattices. Put $Q_m=[-m,m]^n\cap\Z^n$. For
$\mathcal C\subseteq\Z^n$, its upper and lower Banach densities are,
respectively,
\[
\limsup_{m\to\infty}\sup_{\mathbf z\in\Z^n}
\frac{|\mathcal C\cap(\mathbf z+Q_m)|}{|Q_m|},
\qquad
\liminf_{m\to\infty}\inf_{\mathbf z\in\Z^n}
\frac{|\mathcal C\cap(\mathbf z+Q_m)|}{|Q_m|}.
\]
Each statement in the following corollary is equivalent to
nonexistence of a Lee-ball tiling for the given $n,r$.

\begin{corollary}\label{cor:arbitrary-centre-consequences}
	Let $n\ge3$ and $r\ge2$. There exist an integer $m\ge1$
	and a constant $\varepsilon_{n,r}>0$ such that
	\begin{enumerate}
	\renewcommand{\labelenumi}{\textup{(\roman{enumi})}}
		\item No function $f:Q_{m+r}\to\{0,1\}$ satisfies
		\[
		\sum_{\mathbf b\in B_n(r)}f(\mathbf x-\mathbf b)=1
		\qquad(\mathbf x\in Q_m).
		\]
		
		\item Every packing by translates of $B_n(r)$ has centre set
		of upper Banach density at most
		$(1-\varepsilon_{n,r})/V_r(n)$.
		Every covering by these translates has centre set
		of lower Banach density at least
		$(1+\varepsilon_{n,r})/V_r(n)$.
		
		\item For every measure-preserving $\Z^n$-action $\tau$
		on a probability space $(\Omega,\mu)$, every measurable
		$A\subseteq\Omega$ whose translates
		$\{\tau^{\mathbf b}A:\mathbf b\in B_n(r)\}$
		are pairwise disjoint up to null sets satisfies
		$V_r(n)\mu(A)\le1-\varepsilon_{n,r}$.
	\end{enumerate}
\end{corollary}

For part~(i), compactness would turn arbitrarily large local
solutions into a global tiling. The resulting finite obstruction
forces a hole in every translated cube of a packing and an overlap
in every translated cube of a covering, giving part~(ii).
Integrating this local defect along orbits gives part~(iii).
Conversely, a tiling has density $1/V_r(n)$, and an invariant
probability measure on its orbit closure gives an exact measurable
tower. For related compactness and packing-to-tiling arguments,
see~\cite{HK,Kolountzakis}.

We next state three lattice consequences. For finite simple
undirected abelian Cayley graphs of degree $2n$ and diameter at
most $r$, attaining the Lee bound $V_r(n)$ is equivalent to a
lattice Lee tiling~\cite{DF,LZ}.

Let $\mathbf e_1,\ldots,\mathbf e_n$ be the standard basis of $\Z^n$.
For the covering formulation, a \emph{Lee-ball completion} is an
$n$-tuple of permutations $\pi_1,\ldots,\pi_n$ of $B_n(r)$ satisfying
$\pi_i(\mathbf x)=\mathbf x+\mathbf e_i$ whenever both points
belong to the ball. These permutations define a right action of the
free group $F_n=\langle a_1,\ldots,a_n\rangle$ by
$\mathbf x\cdot a_i=\pi_i(\mathbf x)$.
For the cubature formulation, a finite subgroup
$H\le\mathbb R^n/\Z^n$ is the node set of an equal-weight lattice
rule. The character identities below express exactness for
trigonometric polynomials with frequencies in
$B_n(2r)$~\cite{CoolsSloan,Lyness}.

\begin{corollary}\label{cor:lattice-consequences}
	Let $n\ge3$ and $r\ge2$.
	\begin{enumerate}
		\renewcommand{\labelenumi}{\textup{(\roman{enumi})}}
		\item Every finite simple undirected abelian Cayley graph of degree
		$2n$ and diameter at most $r$ has order at most $V_r(n)-1$.
		
		\item Every Lee-ball completion has a nonnormal stabilizer of
		$\boldsymbol0$ in $F_n$, of index $V_r(n)$. Equivalently, its
		associated covering is neither a normal topological covering
		nor a Galois graph covering.
		
		\item Every finite subgroup $H\le\mathbb R^n/\Z^n$ satisfying
		\[
		\sum_{\boldsymbol\theta\in H}
		e^{2\pi i\mathbf k\cdot\boldsymbol\theta}=0
		\qquad
		(\boldsymbol0\ne\mathbf k\in B_n(2r))
		\]
		has $|H|\ge V_r(n)+1$. This node bound also applies to translated
		lattice rules. Equivalently, no finite subgroup of
		$\mathbb R^n/\Z^n$ or its translate indexes an orthogonal
		exponential basis of $\mathbb C^{B_n(r)}$.
	\end{enumerate}
\end{corollary}

Each part is equivalent to lattice nonexistence. For part~(ii),
the Schreier graph is a connected $V_r(n)$-sheeted covering of
the bouquet of $n$ circles. Normal stabilizers correspond to normal
topological coverings and Galois graph
coverings~\cite{Hall,Hatcher,Terras}. Since $r\ge2$, the internal
coordinate squares show that every generator commutator fixes
$\boldsymbol0$. A normal stabilizer would therefore give an
abelian quotient with $B_n(r)$ as a complete representative set,
hence a lattice tiling. Conversely, a lattice tiling gives a
normal completion.

For part~(iii), character orthogonality shows that the annihilator
of $H$ meets $B_n(2r)$ only at $\boldsymbol0$. Since
$B_n(r)-B_n(r)=B_n(2r)$, the ball injects into the corresponding
quotient of order $|H|$. Equality $|H|=V_r(n)$ would give a lattice
tiling. This is the lattice Fourier correspondence~\cite{Fuglede}.
The assertion concerns lattice rules and subgroup spectra;
the unrestricted tiling--spectral equivalence is false~\cite{TaoFuglede}.
Together with the classical constructions, part~(i) classifies
attainability of the Lee bound: for positive integers $n,r$,
equality occurs exactly when $n\le2$ or $r=1$.
The exact maximum order in the remaining cases is not determined.

The proof of Theorem~\ref{thm:main} combines translation-invariant
averaging with local contact identities and integer moment
inequalities. Write $2n=(r+1)q+c$ with $0\le c\le r$. Geometric
estimates and the small-radius arguments reduce the remaining parameters
to the two residue classes $c=0$ and $c=r$.
Stable symmetric functions encode the contact profiles;
differentiation before specialization pairs adjacent shells.
Coefficient comparisons with nonnegative corrections then exclude
both residue classes, with explicit finite identities handling
the exceptional pairs. We use the known exclusions
\cite[Theorem~3]{Sz} and \cite[Theorem~13]{HK}, and the dilation
identity \cite[Corollary~25]{HK}.

The paper is organized as follows.
Sections~2 and~3 develop the averaging tools and geometric estimates.
Section~4 establishes the contact identities, treats radii two,
three, and four, and completes the parameter reduction.
Section~5 excludes the two remaining residue classes.

\section{Preliminaries}\label{sec:preliminaries}
We collect the averaging, counting, and translation tools used in the proof.

Assume \eqref{eq:tiling}, let $f=\mathbf1_{\mathcal C}$ be its indicator,
and write $V_j=V_j(n)$ with the dimension fixed. Translation-invariant averaging
permits local incidence arguments without a periodicity assumption. A local
function at $\mathbf z$ is determined by the finite pattern
$\mathcal C\cap(\mathbf z+K)$ for a fixed finite set $K\subseteq\Z^n$.

\begin{lemma}
\label{foundation:lem:v5-local-mean}
Distinct centres have Lee distance at least $2r+1$. There is a positive,
translation-invariant linear functional $\M$ on bounded local functions
of the translated tiling such that $\M(1)=V_r$ and $\M(f)=1$.
\end{lemma}
\begin{proof}
If $\mathbf{c},\mathbf{c}'\in {\mathcal C}$ had distance $d\le2r$, an integer geodesic from $\mathbf{c}$
to $\mathbf{c}'$ would contain a point at distance at most $r$ from both:
take the point at step $\max(0,d-r)$.  This contradicts disjointness
in \eqref{eq:tiling}.

Since $\Z^n$ is abelian, Day's theorem~\cite[Theorem~2]{Day}
supplies a positive, normalized, translation-invariant mean on its
bounded real-valued functions. Multiplying this mean by $V_r>0$ and
restricting it to local functions gives $\M$ with $\M(1)=V_r$.
Applying $\M$ to \eqref{eq:tiling} gives
\[
 V_r=\M(1)
  =\sum_{\mathbf u\in B_n(r)}\M\bigl(f(\,\cdot-\mathbf u)\bigr)
  =\sum_{\mathbf u\in B_n(r)}\M(f)
  =V_r\M(f).
\]
Division by $V_r>0$ gives $\M(f)=1$.
\end{proof}

The same mean equates incoming and outgoing finite-range incidence counts.
A subscript on $\M$ specifies the variable being averaged.

\begin{lemma}\label{foundation:lem:v5-transport}
Let $A(\mathbf{c},\mathbf{z})$ be a bounded real incidence weight, zero unless
$\mathbf{c}\in {\mathcal C}$ and $\|\mathbf{z}-\mathbf{c}\|_1\le K$ for some fixed $K$.  Suppose that for
each fixed displacement its value is a local function of the tiling.
Then
\begin{equation}\label{foundation:eq:v5-transport}
 \M_\mathbf{z}\sum_{\mathbf{c}\in {\mathcal C}}A(\mathbf{c},\mathbf{z})
   =\M_\mathbf{c}\sum_{\mathbf{z}\in\mathbb Z^n}A(\mathbf{c},\mathbf{z}).
\end{equation}
In particular, if every centre sends total weight at most $A_0$,
then the left side is at most $A_0$.
\end{lemma}
\begin{proof}
For $\mathbf{v}\in B_n(K)$ put $a_\mathbf{v}(\mathbf{c})=A(\mathbf{c},\mathbf{c}+\mathbf{v})$, with value zero off ${\mathcal C}$.
The received weight at $\mathbf{z}$ is $\sum_\mathbf{v} a_\mathbf{v}(\mathbf{z}-\mathbf{v})$ and the sent weight
at $\mathbf{c}$ is $\sum_\mathbf{v} a_\mathbf{v}(\mathbf{c})$.  These are finite sums, so translation
invariance proves \eqref{foundation:eq:v5-transport}.  The bound follows from
$\sum_\mathbf{z} A(\mathbf{c},\mathbf{z})\le A_0f(\mathbf{c})$ and $\M(f)=1$.
\end{proof}

The shell counts and coordinate deficits in the next lemma will
supply the boundary moments.

For $j\ge0$, let $S_j=\{\mathbf x\in\Z^n:\|\mathbf x\|_1=j\}$
and $\sigma_j=\sigma_j(n)=|S_j|$; set $S_j=\varnothing$ and
$\sigma_j=0$ for $j<0$. For a vector $\mathbf x$, write
$s(\mathbf x)=|\{i:x_i\ne0\}|$ for its support size.

\begin{lemma}\label{foundation:lem:v5-shell}
One has $\sigma_0=1$, $\sigma_1=2n$, and
$\sigma_{k-1}/\sigma_k\le k/(2n)$ for $k\ge1$.
For $k\ge1$,
\begin{align}
 &\sum_{\mathbf{x}\in S_k}(k-s(\mathbf{x}))=(k-1)\sigma_{k-1},\label{foundation:eq:v5-first-deficit}\\
& \sum_{\mathbf{x}\in S_k}(k-s(\mathbf{x}))(k-s(\mathbf{x})-1)
   =(k-1)(k-2)\sigma_{k-2}.\label{foundation:eq:v5-second-deficit}
\end{align}
\end{lemma}
\begin{proof}
The shells $S_0=\{\boldsymbol0\}$ and
$S_1=\{\pm\mathbf e_i:1\le i\le n\}$ give $\sigma_0=1$ and
$\sigma_1=2n$. A coordinate contributes
$1+2u+2u^2+\cdots=(1+u)/(1-u)$, so
\[
 F(u):=\sum_{k\ge0}\sigma_k u^k
       =\left(\frac{1+u}{1-u}\right)^n,
 \qquad
 F'(u)=nF(u)\left(\frac1{1+u}+\frac1{1-u}\right)
       =\frac{2nF(u)}{1-u^2}.
\]
Comparing coefficients of $u^{k-1}$ in $(1-u^2)F'=2nF$ yields
\begin{equation}\label{foundation:eq:shell-recurrence}
 k\sigma_k=2n\sigma_{k-1}+(k-2)\sigma_{k-2}\qquad(k\ge2).
\end{equation}
For $k\ge2$, the last term is nonnegative, so
$2n\sigma_{k-1}\le k\sigma_k$; division by $2n\sigma_k>0$
gives the ratio bound. For $k=1$, the ratio is
$\sigma_0/\sigma_1=1/(2n)$.

A vector in $S_k$ with exactly $j$ nonzero coordinates is specified
by its support, its signs, and a composition of $k$ into $j$ positive
parts, so their number is
$2^j\binom {n}{j}\binom{k-1}{j-1}$.  For $k\ge2$, grouping the vectors by their support size gives
\[
\begin{aligned}
	\sum_{\mathbf x\in S_k}(k-s(\mathbf x))
	&=\sum_{j=1}^{\min(n,k)}
	2^j\binom nj(k-j)\binom{k-1}{j-1}\\
	&=(k-1)\sum_{j=1}^{\min(n,k-1)}
	2^j\binom nj\binom{k-2}{j-1}\\
	&=(k-1)\sigma_{k-1}.
\end{aligned}
\]
Here the term with $j=k$ vanishes, and we used
$(k-j)\binom{k-1}{j-1}=(k-1)\binom{k-2}{j-1}$
for $1\le j\le k-1$. This proves
\eqref{foundation:eq:v5-first-deficit}.
Similarly, for $k\ge3$, the terms with $j=k-1,k$ vanish, and
\[
(k-j)(k-j-1)\binom{k-1}{j-1}
=(k-1)(k-2)\binom{k-3}{j-1}
\qquad(1\le j\le k-2).
\]
Consequently,
\[
\begin{aligned}
	&\sum_{\mathbf x\in S_k}
	(k-s(\mathbf x))(k-s(\mathbf x)-1)\\
	&\qquad=\sum_{j=1}^{\min(n,k)}
	2^j\binom nj(k-j)(k-j-1)\binom{k-1}{j-1}\\
	&\qquad=(k-1)(k-2)\sum_{j=1}^{\min(n,k-2)}
	2^j\binom nj\binom{k-3}{j-1}\\
	&\qquad=(k-1)(k-2)\sigma_{k-2},
\end{aligned}
\]
which proves \eqref{foundation:eq:v5-second-deficit}.

For $k=1$, every $\mathbf x\in S_1$ has $s(\mathbf x)=1$,
so both left sides are zero; both right sides are also zero
because $k-1=0$.
For $k=2$, every $\mathbf x\in S_2$ has
$s(\mathbf x)\in\{1,2\}$, so
$(2-s(\mathbf x))(1-s(\mathbf x))=0$.
Thus the second left side is zero, as is its right side
because $k-2=0$.
\end{proof}

For $\mathbf w\in\Z^m$, write $L=\sum_i|w_i|$ and define
\[
 \mathcal K_j(\mathbf w)=
 \{\mathbf u\in\Z^m:u_iw_i\ge0,\ |u_i|\le|w_i|\ (1\le i\le m),\
                    \|\mathbf u\|_1=j\},
 \qquad a_j=|\mathcal K_j(\mathbf w)|.
\]

\begin{lemma}
\label{foundation:lem:v5-rank-comparison}
For $0\le k\le\min(j,L-j)$, $a_j>0$ and $
 \frac{a_j}{a_k}\le\frac{\binom Lj}{\binom Lk}.$
\end{lemma}
\begin{proof}
Taking coordinatewise absolute values gives a rank-preserving bijection
to the product of chains $\prod_i\{0,1,\ldots,|w_i|\}$, so we may
assume that every $w_i\ge0$.
We first record the only consequence of symmetry and unimodality that
is needed. Suppose $(b_0,\ldots,b_M)$ is nonnegative, symmetric, and
nondecreasing up to its middle, and put $b_i=0$ outside $[0,M]$.
If $\xi\le\eta$ and $\xi+\eta\le M$, then
$
 b_\eta\ge b_\xi.
$
For $\xi<0$ this follows from nonnegativity. Otherwise
$0\le\xi\le\eta\le M-\xi$, so
$\xi\le\min(\eta,M-\eta)\le M/2$ and
$b_\eta=b_{\min(\eta,M-\eta)}\ge b_\xi$.

The rank generating function is
$\prod_i(1+z+\cdots+z^{w_i})$. Symmetry follows from
$\mathbf u\mapsto\mathbf w-\mathbf u$, and every rank between zero and
$L$ is nonempty by allocating its units successively among the
coordinates. To prove unimodality, start with the single coefficient
$b_0=1$ and adjoin the factors one at a time. If the old total rank is
$M$ and the next factor has degree $w$, the new coefficients satisfy
\[
 a_k=\sum_{j=0}^w b_{k-j},\qquad
 a_k-a_{k-1}=b_k-b_{k-w-1}.
\]
For $k\le(M+w)/2$, the two indices obey
$(k-w-1)+k\le M-1$, so  $
 b_\eta\ge b_\xi
$
gives $a_k-a_{k-1}\ge0$. This completes the induction.

For the normalized comparison, let $b^{(i)}_j$ be the rank count after
coordinate $i$ is deleted, extended by zero outside its rank interval.
Count edges from rank $k$ to rank $k+1$, weighting the edge that
increments $u_i$ by $u_i+1$. At rank $k+1$ the incoming weights sum to
$\sum_i v_i=k+1$. At rank $k$ one also has
$\sum_i(w_i-u_i)=L-k$. Subtracting these two counts gives
\begin{align*}
 (L-k)a_k-(k+1)a_{k+1}
 &=\sum_i\sum_{a=0}^{w_i}
  \bigl(w_i-a-\mathbf1_{a<w_i}(a+1)\bigr)b^{(i)}_{k-a}\\
 &=\sum_i\sum_{a=0}^{w_i-1}(w_i-2a-1)b^{(i)}_{k-a}.
\end{align*}
Assume $2k+1\le L$. Pair $a$ with $w_i-1-a$ in the last sum.
For $a<(w_i-1)/2$, the paired contribution is $(w_i-2a-1) \bigl(b^{(i)}_{k-a}-b^{(i)}_{k-w_i+1+a}\bigr)\ge0$: the first index is at least the second, their sum is
$2k-w_i+1\le L-w_i$, and $
 b_\eta\ge b_\xi
$
applies to the remaining chain product. A possible middle term has
coefficient zero. Consequently
$a_{k+1}/a_k\le(L-k)/(k+1)$. For $k\le j\le L/2$, multiplying
these positive inequalities from $k$ to $j-1$ gives
\[
 \frac{a_j}{a_k}
 \le\prod_{t=k}^{j-1}\frac{L-t}{t+1}
 =\frac{\binom Lj}{\binom Lk}.
\]
Finally $a_j=a_{L-j}$ and $\binom {L}{j}=\binom {L}{L-j}$ give the
claimed reflected version.
\end{proof}

Coprime dilations provide a second exact tiling identity for the same
centre indicator. Define the translation operator by
$[\mathbf v]F(\mathbf x)=F(\mathbf x-\mathbf v)$ and, for $s\ge0$, put
\[
 \mathsf B_s=\sum_{\|\mathbf v\|_1\le s}[\mathbf v],\qquad
 \mathsf B_s^{[a]}=\sum_{\|\mathbf v\|_1\le s}[a\mathbf v].
\]
The tiling equation is $\mathsf B_rf=1$.

\begin{lemma}[{\cite[Corollary~25]{HK}}]
\label{foundation:v5:primary-dilation}
If $\mathsf B_rf=1$ and $a$ is a positive integer with
$\gcd(a,V_r)=1$, then $\mathsf B_r^{[a]}f=1$.
\end{lemma}

\section{Geometric obstructions and boundary estimates}\label{sec:general}
We derive the boundary identities and use two known exclusions to
prove the triangular bound. Weighted-shadow estimates and diagonal
endpoint exclusions then give the geometric reduction used below.

\subsection{Boundary partitions and moments}
We begin with the exact partition of the signed coordinate directions
at a boundary point and then average its coordinate deficits.
Let $\mathbf{c}_0(\mathbf{z})$ be the unique covering centre of $\mathbf{z}$.  Write $I_{\partial}(\mathbf{z})=1$
when $\|\mathbf{z}-\mathbf{c}_0(\mathbf{z})\|_1=r$ and $I_{\partial}(\mathbf{z})=0$ otherwise.  At a point with $I_{\partial}=1$
define
\[
\mathbf{b}=\mathbf{z}-\mathbf{c}_0(\mathbf{z}),\qquad
 \mathcal F_{\mathbf z}=\{\mathbf{z}-\mathbf{c}:\mathbf{c}\in {\mathcal C},\ \|\mathbf{z}-\mathbf{c}\|_1={(r+1)}\}.
 \]
For a vector $\mathbf x$, define its signed support by
\[
 \mathcal S(\mathbf x)=\{\operatorname{sgn}(x_i)\mathbf e_i:x_i\ne0\},
 \qquad |\mathcal S(\mathbf x)|=s(\mathbf x).
\]
The boundary geometry is governed by the allocation of these directions to nearby centres.
\begin{lemma}\label{foundation:lem:v5-partition}
Every point at distance $r+1$ from a centre is a boundary point.
At each boundary point, the signed supports of $\mathbf b$ and
the vectors in $\mathcal F_{\mathbf z}$ partition the $2n$ signed
coordinate directions.
\end{lemma}
\begin{proof}
Suppose $\|\mathbf z-\mathbf c\|_1=r+1$. Its covering centre
$\mathbf c_0$ is distinct from $\mathbf c$. If
$\|\mathbf z-\mathbf c_0\|_1\le r-1$, then
\[
 \|\mathbf c-\mathbf c_0\|_1
 \le \|\mathbf c-\mathbf z\|_1+\|\mathbf z-\mathbf c_0\|_1
 \le (r+1)+(r-1)=2r,
\]
contrary to centre separation. Hence
$\|\mathbf z-\mathbf c_0\|_1=r$, so $\mathbf z$ is a boundary point.

Fix a boundary point $\mathbf z$.
Distinct vectors
$\mathbf x,\mathbf y\in\{\mathbf b\}\cup\mathcal F_{\mathbf z}$
correspond to distinct centres $\mathbf z-\mathbf x$ and
$\mathbf z-\mathbf y$. If they shared a signed direction
$\mathbf a$, then
\[
\begin{aligned}
 2r+1
 &\le \|(\mathbf z-\mathbf x)-(\mathbf z-\mathbf y)\|_1\le \|\mathbf x-\mathbf a\|_1+\|\mathbf y-\mathbf a\|_1\\
 &=\|\mathbf x\|_1+\|\mathbf y\|_1-2\le 2(r+1)-2=2r,
\end{aligned}
\]
a contradiction. Thus their signed supports are pairwise disjoint.

Let $\mathbf a$ be a signed direction outside
$\mathcal S(\mathbf b)$. Then
\[
 \|\mathbf z-\mathbf a-\mathbf c_0\|_1
 =\|\mathbf b-\mathbf a\|_1
 =\|\mathbf b\|_1+1=r+1.
\]
Consequently $\mathbf z-\mathbf a$ has a covering centre
$\mathbf c\ne\mathbf c_0$. Since $\mathbf z$ is covered only by
$\mathbf c_0$,
\[
 r+1\le\|\mathbf z-\mathbf c\|_1
 \le\|\mathbf z-\mathbf a-\mathbf c\|_1+1
 \le r+1.
\]
Thus $\mathbf x=\mathbf z-\mathbf c$ belongs to
$\mathcal F_{\mathbf z}$ and $\|\mathbf x-\mathbf a\|_1\le r$.
If $\mathbf a\notin\mathcal S(\mathbf x)$, then
$ \|\mathbf x-\mathbf a\|_1=\|\mathbf x\|_1+1=r+2,$
a contradiction. Hence $\mathbf a\in\mathcal S(\mathbf x)$.
Every signed direction therefore belongs to exactly one of the
indicated supports.
\end{proof}

At a boundary point, define the covering deficit and total deficit by
\[
 t=r-s(\mathbf b),\qquad
 Z=t+\sum_{\mathbf x\in\mathcal F_{\mathbf z}}
          (r+1-s(\mathbf x)).
\]
The support partition gives
$s(\mathbf b)+\sum_{\mathbf x\in\mathcal F_{\mathbf z}}s(\mathbf x)=2n,$
and hence
\[
 Z=r+(r+1)|\mathcal F_{\mathbf z}|
       -\left(s(\mathbf b)+\sum_{\mathbf x\in\mathcal F_{\mathbf z}}
                    s(\mathbf x)\right)=r+(r+1)|\mathcal F_{\mathbf z}|-2n.
\]
Moreover, $1\le s(\mathbf b)\le\|\mathbf b\|_1=r$ gives
$0\le t\le r-1$.

Off the boundary, set $\mathcal F_{\mathbf z}=\varnothing$ and
$t=Z=0$. With this convention, the counting identity holds globally as
\[
 (r+1)|\mathcal F_{\mathbf z}|
 =(2n-r)I_{\partial}(\mathbf z)+Z.
\]

Averaging the support partition and using the shell identities of
Section~\ref{sec:preliminaries} gives the following exact moments.

\begin{lemma}\label{foundation:prop:v5-means}
For every fixed weight $a$ on $S_r$ or on $S_{r+1}$, respectively,
\[
 \M(I_{\partial}a(\mathbf{b}))=\sum_{\mathbf{u}\in S_r}a(\mathbf{u}),\qquad
 \M\left(\sum_{\mathbf{x}\in \mathcal F_{\mathbf z}}a(\mathbf{x})\right)=\sum_{\mathbf{u}\in S_{r+1}}a(\mathbf{u}).
\]
Here the exterior expression is zero off the boundary. Then
$ \M(I_{\partial})=\sigma_r, \M(t)=(r-1)\sigma_{r-1},$
$ \M(Z)=r\sigma_r+(r-1)\sigma_{r-1}$ and 
$
 \M(t^2)=(r-1)\sigma_{r-1}+(r-1)(r-2)\sigma_{r-2}.
$
\end{lemma}
\begin{proof}
The integrand of the first mean is
$\sum_{\mathbf{u}\in S_r}a(\mathbf{u})f(\mathbf{z}-\mathbf{u})$, since a centre at distance $r$
from $\mathbf{z}$ is necessarily its unique covering centre. The integrand of the second is
$\sum_{\mathbf{u}\in S_{r+1}}a(\mathbf{u})f(\mathbf{z}-\mathbf{u})$, and Lemma~\ref{foundation:lem:v5-partition}
shows that every such incidence occurs on the boundary.
Translation invariance and $\M(f)=1$ prove both formulas.
Taking $a=1$ and then the corresponding shell deficit gives
$\M(I_{\partial})=\sigma_r$ and $\M(t)=(r-1)\sigma_{r-1}$.
For the remaining moments, Lemma~\ref{foundation:lem:v5-shell} gives
\[
\begin{aligned}
 \M(Z)&=\M(t)+\sum_{\mathbf u\in S_{r+1}}(r+1-s(\mathbf u))
       =(r-1)\sigma_{r-1}+r\sigma_r,\\
 \M(t^2)&=\M(t)+\M(t(t-1))
       =(r-1)\sigma_{r-1}+(r-1)(r-2)\sigma_{r-2}.
\end{aligned}
\]
\end{proof}

Whenever a polynomial in boundary statistics has a nonzero constant
term, its value in a mean below means its value on $I_{\partial}=1$, extended
by zero elsewhere.  In particular a boundary constant $a$ has mean
$a\sigma_r$, not $a{V_r}$.

\begin{lemma}\label{foundation:prop:v5-direction-sums}
For $k\ge0$, put
\[
 T_k=\sum_{\mathbf v\in S_k}(2n-s(\mathbf v)),\qquad
 \Phi_k=\sum_{\mathbf v\in S_k}
       \bigl((2n-s(\mathbf v))^2-2n+2s(\mathbf v)\bigr).
\]
These sums are nonnegative and satisfy
\begin{align*}
 T_k&=(k+1)\sigma_{k+1}-k\sigma_k,\\
 \Phi_k&=(k+1)(k+2)\sigma_{k+2}
       -(k+1)(2k+1)\sigma_{k+1}+k^2\sigma_k.
\end{align*}
In particular $T_0=2n$ and $\Phi_0=2n(2n-1)$.
\end{lemma}
\begin{proof}
For a support size $0\le s\le n$,
\[
 2n-s\ge0,\qquad
 (2n-s)^2-2n+2s=(2n-s)(2n-s-1)+s\ge0,
\]
so both sums are nonnegative. At $k=0$, the single vector is
$\boldsymbol0$, giving $T_0=2n$ and $\Phi_0=2n(2n-1)$.

The two deficit identities in Lemma~\ref{foundation:lem:v5-shell}
hold also at $k=0$, since both sides are zero. The recurrence
\eqref{foundation:eq:shell-recurrence} also holds at $k=0,1$:
its two sides reduce to $0=0$ and $\sigma_1=2n\sigma_0$,
respectively.
The first deficit identity and the recurrence at $k+1$ give
\[
 \begin{aligned}
 T_k&=(2n-k)\sigma_k+
          \sum_{\mathbf v\in S_k}(k-s(\mathbf v))\\
    &=(2n-k)\sigma_k+(k-1)\sigma_{k-1}
      =(k+1)\sigma_{k+1}-k\sigma_k.
 \end{aligned}
\]
For the second sum, expand
\[
 (2n-s)^2-2n+2s
 =(2n-k)^2-2n+2k+(4n-2k-1)(k-s)+(k-s)(k-s-1).
\]
The two deficit identities and the recurrence, successively at
$k$, $k+1$, and $k+2$, therefore give
\begin{align*}
 \Phi_k
 &=[(2n-k)^2-2n+2k]\sigma_k
      +(4n-2k-1)(k-1)\sigma_{k-1}+(k-1)(k-2)\sigma_{k-2}\\
 &=[2n(2n-2k-1)+k(2k+1)]\sigma_k
      +(2n-2k-1)(k-1)\sigma_{k-1}\\
 &=(k+1)(k+2)\sigma_{k+2}
      -(k+1)(2k+1)\sigma_{k+1}+k^2\sigma_k.
\end{align*}
\end{proof}

\subsection{Geometric exclusions}\label{foundation:sec:nonexceptional}
We use the low-dimensional exclusion below and a four-orbit comparison.

\begin{lemma}[{\cite[Theorem~3]{Sz}}]
\label{foundation:v5:small-dim}
For $3\le n\le5$ and $r\ge3$, the ball $B_n(r)$ has no
translational tiling of $\Z^n$.
\end{lemma}

The following lemma computes the four orbit intersections and reduces
the required lower bounds to two parameter families.

Fix $r\ge4$ and $n\ge r+3$. Let
$\mathcal O_-,\mathcal O_0,\mathcal O_2,\mathcal O_3$
be the signed-permutation orbits of
$(1^{r+1})$, $(1^{r+3})$, $(1^{r+1},2)$, and $(1^r,3)$,
respectively, and put $M_j=|\mathcal O_j|$ for
$j\in\{-,0,2,3\}$.
For $\mathbf x\in S_{2r+1}$, write
\[
m=s(\mathbf x),\quad
a=|\{i:|x_i|\ge2\}|,\quad
b=|\{i:|x_i|\ge3\}|,\quad
C_j=|\mathcal O_j\cap(\mathbf x+B_n(r))|.
\]

\begin{lemma}\label{foundation:lem:four-orbit-ratios}
	One has $C_-=\binom m{r+1}$.
	If $m\ge r+1$, put $t=m-r-1$, $v=n-r-1$, and
	$F_j=(C_j/M_j)/(C_-/M_-)$ for $j=0,2,3$.
	Then $0\le t\le\min(r,v)$ and
	\[
	\begin{aligned}
		F_0&=\frac{t\{(r+3)(2v-t-1)+t-1\}}{4v(v-1)},\\
		F_2&=\frac{mt+a(r+1)(2v-t)}{2vm},\\
		F_3&=\frac{a(t+1)+br(2v-t+1)}{m(t+1)}.
	\end{aligned}
	\]
\end{lemma}
\begin{proof}
Signed coordinate permutations preserve the ball and all four orbits,
so they preserve each $C_j$. We may therefore assume $x_i\ge0$.
Choosing the nonzero coordinates, their signs, and the distinguished
coordinate in the last two orbits gives
\[
\begin{array}{c|cccc}
j&-&0&2&3\\\hline
M_j=|\mathcal O_j|&2^{r+1}\binom {n}{r+1}&2^{r+3}\binom {n}{r+3}
 &2^{r+2}n\binom{n-1}{r+1}&2^{r+1}n\binom{n-1}{r}
\end{array}
\]
The identity
\[
\|\mathbf x-\mathbf y\|_1=\|\mathbf x\|_1+\|\mathbf y\|_1
 -2\sum_{i:y_i>0}\min(x_i,y_i)
\]
shows that a negative-orbit vector contributes exactly when all its
$r+1$ entries are positive and inside the support. Thus
$C_-=\binom {m}{r+1}$.
For a positive orbit, its norm is $r+3$, and the allowed defect
$\|\mathbf y\|_1-\sum_{y_i>0}\min(x_i,y_i)$ is at most one.
A good unit entry is positive and inside the support; a bad unit
entry is either negative inside it or has either sign outside it.
Consequently the three counts are
\begin{align*}
C_0&=\binom {m}{r+3}
 +\{m-r-2+2(n-m)\}\binom {m}{r+2},\\
C_2&=(m-a)\binom{m-1}{r+1}
 +a\left[\binom{m-1}{r+1}
 +\{m-r-1+2(n-m)\}\binom{m-1}{r}\right]\\
 &=m\binom{m-1}{r+1}+a(2n-m-r-1)\binom{m-1}{r},\\
C_3&=(a-b)\binom{m-1}{r}
 +b\left[\binom{m-1}{r}
 +\{m-r+2(n-m)\}\binom{m-1}{r-1}\right]\\
 &=a\binom{m-1}{r}+b(2n-m-r)\binom{m-1}{r-1}.
\end{align*}
In the second count, a positive entry two at a coordinate equal to
one spends the defect allowance; at a coordinate at least two it
leaves room for one bad unit. The third count has the same alternatives
for an entry three at coordinates equal to two or at least three.
A nonexistent choice contributes zero.

Suppose $m\ge r+1$. Since $m\le\min(n,2r+1)$, we have
$0\le t\le\min(r,v)$; also $v\ge2$ because $n\ge r+3$.
The required cardinality and binomial quotients are
\[
\frac{M_-}{M_0}=\frac{(r+2)(r+3)}{4v(v-1)},\quad
\frac{M_-}{M_2}=\frac{1}{2v},\quad
\frac{M_-}{M_3}=\frac{1}{r+1},\quad
\frac{\binom{m-1}{r+1}}{\binom m{r+1}}=\frac {t}{m},
\]
\[
\frac{\binom m{r+2}}{\binom m{r+1}}=\frac {t}{r+2},\quad
\frac{\binom m{r+3}}{\binom m{r+1}}
 =\frac{t(t-1)}{(r+2)(r+3)},\quad
\frac{\binom{m-1}r}{\binom m{r+1}}=\frac{r+1}{m},\quad
\frac{\binom{m-1}{r-1}}{\binom m{r+1}}
 =\frac{r(r+1)}{m(t+1)}.
\]
Substituting these quotients into $F_j=(M_-/M_j)(C_j/C_-)$
proves the three ratio formulas.
If $b=0$, then $2r+1=m+a$, so $a=r-t$.
If $b\ge1$, then $a\ge b\ge1$ and
$2r+1\ge m+a+b\ge r+3+t$, so $t\le r-2$.
On the feasible range $t\le v$, the coefficients of $a,b$ in
$F_2,F_3$ are nonnegative. Thus the case $b=0$ is exact at
$(a,b)=(r-t,0)$, while every case $b\ge1$ has $F_2,F_3$ at least
as large as their values at $(a,b)=(1,1)$ with the same $t$.
Since $F_0$ is independent of $a,b$, the same comparison applies to
every positive linear combination of the three ratios.
\end{proof}

The orbit ratios now give a tiling obstruction: assign positive mass
to the three outer orbits and the opposite total mass to $\mathcal O_-$.

\begin{lemma}\label{foundation:lem:v5:four-orbit-duality}
	Let $w_0,w_2,w_3>0$ and set
	$T(t,a,b)=\sum_{j\in\{0,2,3\}}w_j(F_j-1)$.
	If
	\[
	\begin{aligned}
		T(t,r-t,0)&>0 &&(0\le t\le\min(r,v)),\\
		T(t,1,1)&>0 &&(0\le t\le\min(r-2,v))
	\end{aligned}
	\]
	for all integers $t$ in the respective ranges, then
	$B_n(r)$ admits no translational tiling of $\Z^n$.
\end{lemma}
\begin{proof}
Put
$
g=\sum_{j=0,2,3}\frac{w_j}{M_j}\mathbf1_{\mathcal O_j}
 -\frac {w_0+w_2+w_3}{M_-}\mathbf1_{\mathcal O_-}.
$
Then $g=0$ on $B_n(r)$ and $\sum g=0$. Write
$(g*\mathbf1_{B_n(r)})(\mathbf x)=\sum_{\mathbf y\in B_n(r)}g(\mathbf x-\mathbf y)$.
At norm $2r+1$,
$(g*\mathbf1_{B_n(r)})(\mathbf x)
=\frac{C_-}{M_-}T(t,a,b)>0\quad\text{if }C_->0;$
if $C_-=0$, the convolution is nonnegative. For
$\|\mathbf x\|_1>2r+1$ and $\mathbf y\in\mathcal O_-$,
$\|\mathbf x-\mathbf y\|_1\ge\|\mathbf x\|_1-(r+1)>r$.
Thus the negative orbit contributes nothing at larger norms,
and the convolution remains nonnegative.
Suppose an arbitrary subset ${\mathcal C}\subseteq\mathbb Z^n$ gives an exact
translational tiling.  Translate ${\mathcal C}$ so that $0\in {\mathcal C}$.
Distinct centres are separated by at least $2r+1$, since
$B_n(r)-B_n(r)=B_n(2r)$.  In particular every nonzero $\mathbf{c}\in {\mathcal C}$
has $\|\mathbf{c}\|_1\ge2r+1$.
Choose $\mathbf y\in\mathcal O_-$.  Its unique covering centre $\mathbf{c}$
is nonzero because $\|\mathbf y\|_1=r+1$.  It satisfies $2r+1\le\|\mathbf{c}\|_1\le\|\mathbf y\|_1+r=2r+1$, and its negative-orbit intersection count is positive.  Hence the
convolution at this centre is strictly positive.  All other
centre summands are nonnegative, including the zero summand at the
origin.  On the other hand, exact coverage and symmetry of the ball give
\[
0<\sum_{\mathbf{c}\in {\mathcal C}}(g*\mathbf1_{B_n(r)})(\mathbf{c})
=\sum_{\mathbf y}g(\mathbf y)
  \sum_{\mathbf{c}\in {\mathcal C}}\mathbf1_{B_n(r)}(\mathbf{c}-\mathbf y)
=\sum_{\mathbf y}g(\mathbf y)=0,
\]
a contradiction.  The convolutions have finite support, so all displayed
sums are finite even when ${\mathcal C}$ is nonperiodic.
\end{proof}

The same weights below handle the entire four-orbit range; no change of
weights at $n=2r$ is needed.

\begin{lemma}\label{foundation:v5:triangular}
	For $r\ge4$ and $3\le n\le r(r+1)/2+2$,
	$B_n(r)$ admits no translational tiling of $\Z^n$.
\end{lemma}
\begin{proof}
	Lemma~\ref{foundation:v5:small-dim} excludes $3\le n\le5$.
	For $6\le n\le r+2$, one has $r\ge n-2$ and
	\[
	r-\left(\frac{\sqrt2}{2}n-\frac{3\sqrt2}{4}-\frac12\right)
	\ge\frac{(2-\sqrt2)(2n-3)}4>0.
	\]
	Horak and Kim~\cite[Theorem~13]{HK} exclude tilings whenever
	$n\ge6$ and $r\ge\sqrt2 n/2-3\sqrt2/4-1/2$, so they exclude
	this range. For the remaining range put
	\[
	v=n-r-1,\qquad 2\le v\le V:=\frac{r^2-r+2}{2},\qquad
	(w_0,w_2,w_3)=\left(1,\frac2r,\frac{5(r-1)}{14v}\right).
	\]
	These weights are positive. We check the two branches
	$(a,b)=(r-t,0)$ and $(a,b)=(1,1)$, with integer
	$0\le t\le\min(r,v)$ and $0\le t\le\min(r-2,v)$, respectively.
	Write $m=r+t+1$ and $E=(2v-t-1)(r+3)+t-1$.
	Substituting the orbit ratios gives
	\begin{align*}
		P_0(t,v):&=28rv(v-1)m\,T(t,r-t,0)\\
		&=7rmtE+28(v-1)\{mt+(r-t)(r+1)(2v-t)\}-28(r+2)mv(v-1)\\
		&\quad-10r(r-1)(2t+1)(v-1),\\
		P_1(t,v):&=28rv(v-1)m(t+1)\,T(t,1,1)\\
		&=7rmt(t+1)E+28(v-1)(t+1)\{mt+(r+1)(2v-t)\}-28(r+2)mv(v-1)(t+1)\\
		&\quad+10r(r-1)(v-1)\{(1-m)(t+1)+r(2v-t+1)\}.
	\end{align*}
	At $t=0$,
	\begin{align*}
		P_0(0,v)&=2(v-1)\{14(r-2)(r+1)v-5r(r-1)\},\\
		P_1(0,v)&=4rv(v-1)(5r^2-12r-7).
	\end{align*}
	Both are positive: $(r-2)v\ge r$ gives
	$14(r-2)(r+1)v-5r(r-1)\ge r(9r+19)>0$, while
	$5r^2-12r-7\ge8r-7>0$.
	
	Now let $t\ge1$ and put $D=(r+2)t+r+4\ge2r+6$.
	At $v=0$, the defining formulas and
	$0\le r-t\le r$, $2t+1\le3t$, $m\ge r$ give
	\begin{align*}
		\frac{P_0(t,0)}{rtm}
		&\le-7D+\frac{28(r+1)+30(r-1)}m\le-14r+16<0,\\
		\frac{P_1(t,0)}{rtm}
		&\le-7(t+1)D+10(r-1)(1+r/m)\le-8r-104<0.
	\end{align*}
	For the second line we used
	$(m-1)(t+1)+r(t-1)=t(m+r)$.
	
	At $v=t=1$, both polynomials vanish. For $t\ge2$,
	$E=(t-1)(r+4)$ gives
	\begin{align*}
		&\frac{P_0(t,t)}{t(t-1)}
		=7m(r^2-4)+28(r-t)(r+1)
		-10r(r-1)(2+1/t),\\
		&\frac{P_1(t,t)}{t(t-1)(t+1)}=7m(r^2-4)-10r(r-1)+28(r+1).
	\end{align*}
	Since $r-t\ge0$ and $2+1/t\le5/2$, each expression is at least
	\[
	7m(r^2-4)-25r(r-1)
	\ge r(r-1)(7m-25)
	\ge r(r-1)(7r-4)>0,
	\]
	where $r^2-4\ge r(r-1)$ and $m\ge r+3$.
	
	For fixed $t$, both $P_i(t,v)$ are quadratic in $v$. Hence
	\[
	\frac{d^2}{dv^2}\left(\frac{P_i(t,v)}v\right)
	=\frac{2P_i(t,0)}{v^3}<0\qquad(v>0).
	\]
	Thus $P_i(t,v)/v$ is strictly concave. Since
	$V-r=(r-1)(r-2)/2>0$ and $P_i(t,t)\ge0$,
	proving $P_i(t,V)>0$ will imply positivity throughout
	$t<v\le V$. The feasible endpoint $v=t\ge2$ was already treated;
	when $t=1$, feasibility requires $v\ge2>t$.
	
	For the upper endpoint in the first branch put
	$Q_r(x)=P_0(r-x,V)/r$ and $u=r-4\ge0$.
	With $\Delta Q_r(x)=Q_r(x+1)-Q_r(x)$, direct expansion gives
	\[
	\Delta^2Q_r(x)=14(r^3-6r+2)+42(r+2)x>0\qquad(x\ge0),
	\]
	since $r^3-6r+2=r(r^2-6)+2>0$.
	The two differences adjacent to $x=2$ are
	\[
	\Delta Q_r(1)=-11u^3-75u^2-118u-18<0,\qquad
	\Delta Q_r(2)=3u^3+93u^2+512u+822>0.
	\]
	Thus $Q_r$ decreases on the integer points $0,1,2$ and increases
	on the integer points $x\ge2$. Consequently
	\[
	Q_r(x)\ge Q_r(2)=4u^4+36u^3+99u^2+103u+24>0
	\qquad(x\in\mathbb Z_{\ge0}).
	\]
	This proves $P_0(t,V)>0$ throughout the first branch.
	
	For the second branch, assume first $r\ge6$ and $1\le t\le r-2$.
	Put $z=(t+1)(r-t-1)\le r^2/4$. The defining formulas give
	\[
	\begin{aligned}
		\frac{P_1(t,V)-(t+1)P_0(t,V)}{V-1}
		&=(2V-t)\{10r^2(r-1)-28(r+1)z\}
		+10r(r-1)(r-z)\\
		&\ge r^2\left\{(r^2-2r+4)(3r-17)
		-\frac52(r-1)(r-4)\right\}>0.
	\end{aligned}
	\]
	Here $2V-t\ge r^2-2r+4$ and
	$10r^2(r-1)-28(r+1)z\ge r^2(3r-17)>0$.
	The expression in braces equals $3$ when $r=6$.
	For $r\ge7$, it is positive because $3r-17\ge4>5/2$ and
	$r^2-2r+4=(r-1)(r-4)+3r$.
	Thus $P_1(t,V)>(t+1)P_0(t,V)>0$.
	The two remaining radii are checked directly:
	\[
	\begin{array}{c|c|c}
		r&V&\bigl(P_1(t,V)/(r(t+1))\bigr)_{t=1}^{r-2}\\\hline
		4&7&(66,444)\\
		5&11&(4704,1456,3482)
	\end{array}
	\]
	The case $t=0$ was already proved.
	The endpoint argument now gives $P_i(t,v)>0$ throughout both
	feasible ranges. The normalizing factors are positive, so
	Lemma~\ref{foundation:lem:v5:four-orbit-duality} excludes the tiling.
\end{proof}

\subsection{Weighted shadows and remainder estimates}
\label{foundation:sec:v5-nonresidues}
Assume \eqref{eq:tiling}, with $r,n\ge3$, and retain the boundary
statistics, shell sums, and normalized mean defined above, with
$2n=(r+1)q+c$ and $0\le c\le r$.

Fix integers $\beta,j\ge1$ and $0\le k\le\beta$.
For $\alpha\ge j$, let $Z_\alpha$ be the normalized mean of
ordered local incidences between distinct centres, with root and
other vectors of norms $\alpha,\beta$, disjoint signed supports,
and a marked root direction of magnitude at least $j$.
When $k>0$, also mark an other-vector direction of magnitude
at least $k$; when $k=0$, omit the second mark.
Write
\[
m_\alpha=\alpha+\beta-j-k,\qquad A_\alpha=\alpha-j.
\]
The following lemma bounds these counts simultaneously over
several root norms.

\begin{lemma}\label{foundation:lem:v5-threshold-capacity}
	Let $\mathcal I\subseteq\{j,j+1,\ldots\}$ be finite and nonempty,
	and put $\ell=\max\{0,\max_{\alpha\in\mathcal I}m_\alpha-r\}$.
	If $\ell\le\min\{\min\mathcal I-j,\beta-k\}$, then
	\[
	\sum_{\alpha\in\mathcal I}
	\frac{\binom{m_\alpha}{\ell}}{\binom{m_\alpha}{A_\alpha}}Z_\alpha
	\le
	\begin{cases}
		T_\ell,&k=0,\\
		\Phi_\ell,&k>0.
	\end{cases}
	\]
\end{lemma}
\begin{proof}
Fix a root centre $\mathbf c$ and a marked root direction $\mathbf a$.
If $k>0$, fix the second marked direction $\mathbf e$; if $k=0$,
put $k\mathbf e=\boldsymbol0$ without choosing $\mathbf e$.
For an incident other centre $\mathbf c'$ and point $\mathbf z$, write
$\mathbf x=\mathbf z-\mathbf c$, $\mathbf y=\mathbf z-\mathbf c'$, and set
\[
 \mathbf w=\mathbf c'-\mathbf c-j\mathbf a+k\mathbf e
          =(\mathbf x-j\mathbf a)-(\mathbf y-k\mathbf e).
\]
The residual vectors have opposite signs in every shared nonzero
coordinate. Hence $\|\mathbf w\|_1=m_\alpha$, and
$\mathbf x-j\mathbf a$ is a sign-compatible subvector of rank
$A_\alpha$. For fixed $\mathbf c'$, the map
$\mathbf z\mapsto\mathbf x-j\mathbf a$ is injective.
The same sign calculation gives
$\|\mathbf c'-\mathbf c\|_1=\alpha+\beta$, so a neighbour has
only one type $\alpha$.
Lemma~\ref{foundation:lem:v5-rank-comparison} gives
\[
 \frac{\binom{m_\alpha}\ell}{\binom{m_\alpha}{A_\alpha}}
 |\mathcal K_{A_\alpha}(\mathbf w)|\le|\mathcal K_\ell(\mathbf w)|.
\]
The rank-$\ell$ sets are disjoint across distinct neighbours, including
different types. A common subvector $\mathbf v$ would give
\[
 \|\mathbf c'-\mathbf c''\|_1
 =\|\mathbf w-\mathbf w'\|_1
 \le(m_\alpha-\ell)+(m_{\alpha'}-\ell)\le2r.
\]
Here the marked translation is fixed, and the definition of $\ell$
gives $m_\alpha-\ell\le r$ for every type, also when $\ell=0$.
The displayed inequality contradicts centre separation.

Removing marked units never reverses coordinate signs.
Thus every $\mathbf v\in\mathcal K_\ell(\mathbf w)$ satisfies
$\mathbf a\notin-\mathcal S(\mathbf v)$; if $k>0$, also
$\mathbf e\notin\mathcal S(\mathbf v)$ and $\mathbf a\ne\mathbf e$.
For $s=s(\mathbf v)$, there are $2n-s$ possible root marks.
For two marks, there are initially $(2n-s)^2$ pairs.
The equal allowed marks lie outside
$\mathcal S(\mathbf v)\cup-\mathcal S(\mathbf v)$ and number
$2n-2s$, leaving $(2n-s)^2-2n+2s$ pairs.
Opposite marks on one axis remain allowed.

Apply the rank comparison to each neighbour and use the
disjointness proved above. Summing over neighbours and marks,
and counting the admissible marks for each $\mathbf v\in S_\ell$,
bounds the outgoing weighted count at each root centre by
$T_\ell$ or $\Phi_\ell$.
Finite-range mass transport
(Lemma~\ref{foundation:lem:v5-transport}), with
$\M(\mathbf1_{\mathcal C})=1$, gives the stated mean bound.
\end{proof}

For $\alpha,\beta\in\{r,r+1\}$, $1\le j\le\alpha$, and
$1\le k\le\beta$, set $\ell=\max\{0,m_\alpha-r\}$.
Both residual ranks $A_\alpha$ and $m_\alpha-A_\alpha$
are at most $r$, so
$\ell\le\min\{A_\alpha,m_\alpha-A_\alpha\}$.
Taking $\mathcal I=\{\alpha\}$ therefore gives
$Z_\alpha\le
\binom{m_\alpha}{A_\alpha}/\binom{m_\alpha}{\ell}
\Phi_\ell.$

We now return to the total boundary defect $Z$ and put
\[
\Xi_r=\sum_{A,B=0}^{r-2}
(1+2\mathbf1_{A\le r-3})
\frac{\binom{A+B}{A}}
{\binom{A+B}{\max(0,A+B-r)}}
\Phi_{\max(0,A+B-r)}.
\]
\begin{lemma}\label{boundary:prop:total-defect-budget}
For every arbitrary-centre tiling with $r\ge4$,
\begin{equation}\label{boundary:eq:total-collapse}
\begin{aligned}
 &(r-c)(c-1)\sigma_r+(r-2c)(r-1)\sigma_{r-1}\\
 &\quad\le\Xi_r+\left(3r-\frac2{r-1}\right)\sigma_{r-2}
 -\frac{r(r^3-6r^2+5r+4)}{(r-1)(r-2)}\sigma_{r-3}\\
 &\qquad\quad+4r\sum_{l=1}^{r-4}\frac{\sigma_l}{l+1}
                  +4(2n)^2-4(r+1)(2n).
\end{aligned}
\end{equation}
An empty sum is zero.
\end{lemma}
\begin{proof}
	The signed-support partition gives
	$Z=r-c+(r+1)(|\mathcal F_{\mathbf z}|-q)$ on the boundary.
	Since
	\[
	(r+1)|\mathcal F_{\mathbf z}|
	\ge2n-r=(r+1)(q-1)+c+1>(r+1)(q-1),
	\]
	the difference $|\mathcal F_{\mathbf z}|-q$ is a nonnegative
	integer. Hence $(Z-r+c)(Z-2r-1+c)\ge0$ on the boundary.
	Using $\M I_\partial=\sigma_r$ and the first deficit moment
	$\M Z=r\sigma_r+(r-1)\sigma_{r-1}$, we obtain
	\begin{align*}
		\M(Z^2)
		&\ge(3r+1-2c)\M Z-(r-c)(2r+1-c)\sigma_r\\
		&=\{r(r+1)+(r-c)(c-1)\}\sigma_r
		+(3r+1-2c)(r-1)\sigma_{r-1}.
	\end{align*}
	
	For $j\ge2$, let $n_j(\mathbf v)$ count the coordinates of
	magnitude at least $j$, and define
	\begin{align*}
		X_{jk}
		=\M\!\left(
		\sum_{\substack{\mathbf v,\mathbf w\in\mathcal F_{\mathbf z}\\
				\mathbf v\ne\mathbf w}}
		n_j(\mathbf v)n_k(\mathbf w)\right),\qquad
		Y_{jk}
		=\M\!\left(n_j(\mathbf b)
		\sum_{\mathbf w\in\mathcal F_{\mathbf z}}n_k(\mathbf w)\right).
	\end{align*}
	These means are nonnegative, with $X_{jk}=X_{kj}$ and
	$Y_{r+1,k}=0$. Since
	\[
	t=\sum_{j=2}^{r}n_j(\mathbf b),\qquad
	r+1-s(\mathbf v)=\sum_{j=2}^{r+1}n_j(\mathbf v),
	\]
	the square of $Z$ expands into the two individual second moments
	and the mixed terms $X_{jk}+2Y_{jk}$.
	Lemma~\ref{foundation:lem:v5-shell} gives
	\begin{align*}
		&\M(t^2)
		=(r-1)\sigma_{r-1}+(r-1)(r-2)\sigma_{r-2},\\
		&\M\!\left(\sum_{\mathbf v\in\mathcal F_{\mathbf z}}
		(r+1-s(\mathbf v))^2\right)
		=r\sigma_r+r(r-1)\sigma_{r-1}.
	\end{align*}
	Consequently
	\[
	\M(Z^2)
	=r\sigma_r+(r^2-1)\sigma_{r-1}
	+(r-1)(r-2)\sigma_{r-2}
	+\sum_{j,k=2}^{r+1}(X_{jk}+2Y_{jk}).
	\]
	
	We first bound the terms with $j,k\ge3$.
	Their residual ranks are $A=r+1-j$, $B=r+1-k$ for exterior
	pairs, and $A=r-j$, $B=r+1-k$ for covering--exterior pairs.
	Thus $0\le A,B\le r-2$, with the second kind occurring only
	when $A\le r-3$.
	The singleton bound from
	Lemma~\ref{foundation:lem:v5-threshold-capacity}, with
	multiplicities one and two respectively, gives
	\[
	\sum_{j,k=3}^{r+1}(X_{jk}+2Y_{jk})\le\Xi_r.
	\]
	
	For the remaining terms, apply the same lemma with
	$\mathcal I=\{r,r+1\}$, $\beta=r+1$, and $j=2$.
	For $2\le k\le r$, the common shadow rank is $r-k$.
	It is admissible because
	$r-k\le r-2\le\alpha-2$ and $r-k\le r+1-k$.
	The weights of $X_{2k}$ and $Y_{2k}$ are, respectively,
	\[
	\frac{\binom{2r-k}{r-k}}{\binom{2r-k}{r-1}}
	=\frac{r+1-k}{r},\qquad
	\frac{\binom{2r-k-1}{r-k}}{\binom{2r-k-1}{r-2}}
	=\frac{r+1-k}{r-1}.
	\]
	Applying the singleton bound also to $Y_{k2}$ yields
	\[
	\begin{aligned}
		&X_{2k}+\frac r{r-1}Y_{2k}
		\le\frac r{r+1-k}\Phi_{r-k}
		&&(2\le k\le r),\\
		&Y_{k2}
		\le\frac r{r-k}\Phi_{r-k-1}
		&&(2\le k<r),\\
		&X_{2,r+1}+Y_{2,r+1}\le\Phi_0,\qquad
		Y_{r2}\le\Phi_0.
	\end{aligned}
	\]
	For the second inequality, the residual ranks are $r-k,r-1$,
	the shadow rank is $r-k-1$, and the inverse weight is
	\[
	\frac{\binom{2r-k-1}{r-k}}
	{\binom{2r-k-1}{r-k-1}}=\frac r{r-k}.
	\]
	In the last line, the shadow rank is zero and the weights are one.
	
	At the corner, split
	$2Y_{22}=\frac r{r-1}Y_{22}+\frac{r-2}{r-1}Y_{22}$.
	On the two edges, use $r/(r-1)\ge1$.
	Together with $X_{2k}=X_{k2}$ and $Y_{r+1,2}=0$, this gives
	\begin{align*}
		\sum_{j,k=2}^{r+1}(X_{jk}+2Y_{jk})
		&\le\Xi_r+X_{22}+2Y_{22}
		+2\sum_{k=3}^{r+1}(X_{2k}+Y_{2k}+Y_{k2})\\
		&\le\Xi_r+\frac r{r-1}(\Phi_{r-2}+\Phi_{r-3})
		+4\Phi_0+2r\sum_{k=3}^{r}\frac{\Phi_{r-k}}{r+1-k}
		+2r\sum_{k=3}^{r-1}\frac{\Phi_{r-k-1}}{r-k}\\
		&=\Xi_r+4\Phi_0+4r\sum_{k=0}^{r-4}\frac{\Phi_k}{k+1}
		+\frac r{r-1}\Phi_{r-2}+\left(\frac r{r-1}+\frac{2r}{r-2}\right)\Phi_{r-3}.
	\end{align*}
	
	Comparing the two bounds for $\M(Z^2)$, the left-hand side of
	\eqref{boundary:eq:total-collapse} is at most the last displayed
	upper bound plus $
	-r^2\sigma_r-r(r-1)\sigma_{r-1}
	+(r-1)(r-2)\sigma_{r-2}.$
The formula for $\Phi_k$ in
Lemma~\ref{foundation:prop:v5-direction-sums} gives, for $k\ge1$,
\[
\frac{\Phi_k}{k+1}
=\bigl((k+2)\sigma_{k+2}-k\sigma_{k+1}\bigr)
-\bigl((k+1)\sigma_{k+1}-(k-1)\sigma_k\bigr)
+\frac{\sigma_k}{k+1}.
\]
Summing this identity and using
$\Phi_0=2\sigma_2-\sigma_1$ and $\sigma_1=2n$, we obtain
	\[
	\sum_{k=0}^{m}\frac{\Phi_k}{k+1}
	=(m+2)\sigma_{m+2}-m\sigma_{m+1}
	+\sum_{l=1}^{m}\frac{\sigma_l}{l+1}-2n
	\qquad(m\ge0).
	\]
	Apply this identity with $m=r-4$ and substitute the formulas
	for $\Phi_{r-2}$ and $\Phi_{r-3}$.
	The coefficients of $\sigma_r$ and $\sigma_{r-1}$ cancel.
	The next two coefficients simplify as follows:
	\begin{align*}
		&(r-1)(r-2)+4r(r-2)+\frac{r(r-2)^2}{r-1}
		-\frac{r(3r-4)(2r-5)}{r-1}
		=3r-\frac2{r-1},\\
		&-4r(r-4)+\frac{r(3r-4)(r-3)^2}{(r-1)(r-2)}
		=-\frac{r(r^3-6r^2+5r+4)}{(r-1)(r-2)}.
	\end{align*}
	Finally,
	$4\Phi_0-4r(2n)=4(2n)^2-4(r+1)(2n)$.
	This proves \eqref{boundary:eq:total-collapse}.
\end{proof}

\begin{lemma}\label{boundary:lem:common-price}
Suppose $r\ge5$ and ${(2n)}\ge H_r:=r^2+r+4$. Put
\[
 E_r=\left(3r-\frac{2}{r-1}\right)\sigma_{r-2}
       +4r\sum_{l=1}^{r-4}\frac{\sigma_l}{l+1}+4{(2n)}^2.
\]
Then every tiling satisfies $
 (r-c)(c-1)\sigma_r+(r-2c)(r-1)\sigma_{r-1}\le\Xi_r+E_r$
and $\frac{\Xi_r+E_r}{(r-1)\sigma_{r-1}}<2.$
\end{lemma}
\begin{proof}
	Since
	$r^3-6r^2+5r+4=r(r-1)(r-5)+4>0$,
	dropping the negative terms in
	\eqref{boundary:eq:total-collapse} proves the first assertion.
	Iterating the shell ratio in
	Lemma~\ref{foundation:lem:v5-shell} gives
	\[
	\frac{\sigma_u}{\sigma_v}
	\le\prod_{j=u+1}^{v}\frac{j}{2n}
	=\frac{v!}{u!(2n)^{v-u}}
	\qquad(0\le u\le v).
	\]
	In particular, $\sigma_v\ge(2n)^v/v!$, since $\sigma_0=1$.
	For $1\le l\le r-4$,
	\[
	\frac{\sigma_{l-1}/l}{\sigma_l/(l+1)}
	\le\frac{l+1}{2n}\le\frac{r-3}{2n}<1.
	\]
	Hence
	\[
	\sum_{l=1}^{r-4}\frac{\sigma_l}{l+1}
	\le\frac{\sigma_{r-4}}{r-3}
	\sum_{j\ge0}\left(\frac{r-3}{2n}\right)^j
	=\frac{\sigma_{r-4}}{r-3}\frac{2n}{2n-r+3}.
	\]
	Applying the shell estimates once more yields
	\begin{align*}
		&\frac{4r}{(r-1)\sigma_{r-1}}
		\sum_{l=1}^{r-4}\frac{\sigma_l}{l+1}
		\le\frac{4r(r-2)}{(2n)^2(2n-r+3)}
		<\frac4{r^4},\\
		&\frac{4(2n)^2}{(r-1)\sigma_{r-1}}
		\le\frac{4(r-2)!}{(2n)^{r-3}}.
	\end{align*}
	The first line uses $2n\ge r^2$ and $2n-r+3\ge r^2$.
	Together with
	$\sigma_{r-2}/\sigma_{r-1}\le(r-1)/(2n)$, these bounds give
	\begin{equation}\label{boundary:eq:collapsed-error}
		\frac{E_r}{(r-1)\sigma_{r-1}}
		\le\frac{3r-2/(r-1)}{2n}
		+\frac4{r^4}
		+\frac{4(r-2)!}{(2n)^{r-3}}.
	\end{equation}
	
	We next split the sum defining $\Xi_r$ according to whether
	$A+B\ge r$ or $A+B<r$.
	In the first range, write
	$A=r-2-i$, $B=r-2-j$, and $\ell=A+B-r$.
	The binomial quotient is $\ell!r!/(A!B!)$.
	Using $\Phi_\ell\le(2n)^2\sigma_\ell$ and the shell ratio,
	the normalized contribution before the row weight is at most
	\[
		\frac{\ell!r!}{A!B!}
		\frac{\Phi_\ell}{(r-1)\sigma_{r-1}}
		\le\frac{r(r-1)}{2n}
		\frac{(r-2)!}{A!(2n)^i}
		\frac{(r-2)!}{B!(2n)^j}\le\frac{r(r-1)}{2n}
		\left(\frac{r-2}{2n}\right)^{i+j}.
	\]
	The row weight is one when $i=0$ and three when $i\ge1$.
	Extending the sum to all $i,j\ge0$ and summing the geometric
	series gives
	\[
	\sum_{i,j\ge0}
	(1+2\mathbf1_{\{i\ge1\}})
	\left(\frac{r-2}{2n}\right)^{i+j}
	=\frac{(2n)(2n+2r-4)}{(2n-r+2)^2}.
	\]
	
In the range $A+B<r$, the shadow rank is zero and the row weight
is at most three. Extending the sum to all nonnegative $A,B$
with $A+B<r$ gives
\[
\sum_{m=0}^{r-1}\sum_{A=0}^{m}\binom mA=2^r-1<2^r.
\]
Since $\Phi_0\le(2n)^2$ and
$(r-1)\sigma_{r-1}\ge(2n)^{r-1}/(r-2)!$,
the normalized contribution is strictly less than
$3\cdot2^r(r-2)!/(2n)^{r-3}$.
Combining the two ranges therefore gives
\begin{equation*}
	\frac{\Xi_r}{(r-1)\sigma_{r-1}}
	<
	\frac{r(r-1)(2n+2r-4)}{(2n-r+2)^2}
	+\frac{3\cdot2^r(r-2)!}{(2n)^{r-3}}.
\end{equation*}
We now combine this estimate with
\eqref{boundary:eq:collapsed-error}.
The identity
\[
\frac{r(r-1)(2n+2r-4)}{(2n-r+2)^2}
=\frac{r(r-1)}{2n-r+2}
+\frac{3r(r-1)(r-2)}{(2n-r+2)^2}
\]
shows that the capacity term decreases with $2n$.
The same is true of the term $(3r-2/(r-1))/(2n)$
and both factorial remainders.
Thus $2n\ge H_r$ gives
\[
\frac{\Xi_r+E_r}{(r-1)\sigma_{r-1}}
<
\frac{r^2(r-1)(r+3)}{(r^2+6)^2}
+\frac{3r-2/(r-1)}{r^2+r+4}
+L_r+\frac4{r^4},
\]
where $L_r=(3\cdot2^r+4)(r-2)!/H_r^{r-3}$
collects the two factorial remainders.
To bound the first two terms, observe that
\[
\frac{r^2(r-1)(r+3)}{(r^2+6)^2}
=1+\frac{2r-15}{r^2+6}
-\frac{6(2r-9)}{(r^2+6)^2}
<1+\frac{2r-15}{r^2+6},
\]
while
$(3r-2/(r-1))/(r^2+r+4)<3/(r+1)$.
Their sum is consequently less than
\[
1+\frac{2r-15}{r^2+6}+\frac3{r+1}
=\frac75-
\frac{2r(r-6)^2+r^2+5r-3}{5(r+1)(r^2+6)}
<\frac75.
\]
It remains to bound $L_r$.
Since $3\cdot2^{r+1}+4<2(3\cdot2^r+4)$,
\[
\frac{L_{r+1}}{L_r}
<\frac{2(r-1)}{H_{r+1}}
\left(\frac{H_r}{H_{r+1}}\right)^{r-3}<1.
\]
Hence $L_r\le L_5=150/289$.
Substituting this and $4/r^4\le4/625$ into the combined bound yields
\[
\frac{\Xi_r+E_r}{(r-1)\sigma_{r-1}}
<\frac75+\frac{150}{289}+\frac4{625}
<\frac75+\frac{11}{20}+\frac1{20}=2.
\]

\end{proof}
The uniform price estimate now excludes every intermediate remainder in the remaining dimension range.
\begin{lemma}\label{boundary:prop:all-nonexceptional-high}
For every integer $r\ge5$ and integer $n$ with $2n\ge r^2+r+4$,
all residues $1\le {(2n)}\bmod(r+1)\le r-1$ are impossible for an
arbitrary-centre tiling.
\end{lemma}
\begin{proof}
Write $2n=(r+1)q+c$, where $1\le c\le r-1$, and put
\[
 F(c)=(r-c)(c-1)\sigma_r+(r-2c)(r-1)\sigma_{r-1}.
\]
Lemma~\ref{boundary:lem:common-price} gives $
 F(c)\le\Xi_r+E_r<2(r-1)\sigma_{r-1}.$
The shell ratio
$(r-1)\sigma_{r-1}/\sigma_r\le r(r-1)/(2n)$ and
$2n\ge H_r=r^2+r+4$ give
\begin{align*}
 \frac{F(r-1)}{(r-1)\sigma_{r-1}}
 =(r-2)\left(\frac{\sigma_r}{(r-1)\sigma_{r-1}}-1\right)\ge\frac{(r-2)(H_r-r(r-1))}{r(r-1)}\ge2
\end{align*}
and
\[
 \frac{F(1)}{(r-1)\sigma_{r-1}}=r-2\ge3.
 \]
The exact quadratic interpolation is
\[
 F(c)=\frac{r-1-c}{r-2}F(1)+\frac{c-1}{r-2}F(r-1)
       +(c-1)(r-1-c)\sigma_r.
\]
Both weights are nonnegative and sum to one, and the remainder is
nonnegative. Hence $F(c)\ge2(r-1)\sigma_{r-1}$, a contradiction.
The two endpoint estimates use the same fixed $r,n$; neither assumes
that $2n$ itself has an endpoint remainder.
\end{proof}

\subsection{Diagonal endpoint exclusions}
\label{foundation:sec:residual}
The two remaining dimension endpoints will be treated by a local
support count followed by an average over the adjacent shell.
Put $\mathcal U=\{\pm\mathbf e_i:1\le i\le n\}$ and
$a_+=\max(a,0)$.

\begin{lemma}\label{foundation:lem:v5-diagonal-excess-lower}
	Suppose $r\ge3$, $\varepsilon\in\{0,1\}$, and
	$2n=(r+1)^2-\varepsilon$.
	Let $\mathbf z$ have an all-unit covering vector
	$\mathbf b=\mathbf z-\mathbf c_0$ of norm $r-1$.
	For the vectors $\mathbf w=\mathbf z-\mathbf c$ of norm $r+2$, one has
	$\sum_{\mathbf w}\sum_i(|w_i|-2)_+\ge(r+1)(r-2)-3\varepsilon$.
\end{lemma}
\begin{proof}
	Put $H=\sum_{\mathbf w}\sum_i(|w_i|-2)_+$ and  $\mathcal U_{\mathbf b}=\mathcal U\setminus\mathcal S(\mathbf b)$.
	Since $\mathbf b$ is all-unit, $
	m:=|\mathcal U_{\mathbf b}|=r(r+1)+2-\varepsilon.$
	A signed direction is heavy in a vector if its coordinate
	magnitude is at least two.
	Centre separation implies that
	$\mathcal S(\mathbf w)\subseteq\mathcal U_{\mathbf b}$,
	that distinct supports meet in at most one direction,
	and that a direction is heavy in at most one vector.
	Indeed, a common signed unit between $\mathbf b$ and $\mathbf w$
	would give centre distance at most $2r-1$, while two common
	signed units between distinct exterior vectors would give
	centre distance at most $2r$.
	
	Choose $\mathbf a,\mathbf e\in\mathcal U_{\mathbf b}$ with
	$\mathbf e\ne-\mathbf a$, allowing $\mathbf e=\mathbf a$.
	The point $\mathbf z-\mathbf a-\mathbf e$ is at distance $r+1$
	from $\mathbf c_0$, so its covering centre $\mathbf c$ is different.
	Separation and the triangle inequality give
	\[
	r+2
	\le\|\mathbf z-\mathbf c\|_1
	\le\|\mathbf z-\mathbf a-\mathbf e-\mathbf c\|_1+2
	\le r+2.
	\]
	Equality forces the two selected signed units into
	$\mathbf w=\mathbf z-\mathbf c$.
	For distinct $\mathbf a,\mathbf e$, this places both directions
	in $\mathcal S(\mathbf w)$; when they coincide, it makes
	$\mathbf a$ heavy.
	Together with the intersection restrictions, this proves that
	every pair of distinct nonopposite directions occurs in exactly
	one support, and every direction is heavy in exactly one vector.
	
	Call a direction special if its opposite is not in
	$\mathcal U_{\mathbf b}$.
	There are $r-1$ special directions and $m-r+1$ ordinary ones.
	For $\mathbf a\in\mathcal U_{\mathbf b}$, let
	$\eta_{\mathbf a}=0$ for special directions and $1$ otherwise,
	and let $d_{\mathbf a}$ count the supports containing $\mathbf a$.
	Write
	$\delta_{\mathbf w}=r+2-|\mathcal S(\mathbf w)|$.
	The unique heavy owners give
	\[
	\sum_{\mathbf w}|\{i:|w_i|\ge2\}|=m,\qquad
	\sum_{\mathbf w}\delta_{\mathbf w}=m+H.
	\]
	Since $\delta_{\mathbf w}$ is a nonnegative integer,
$|\mathcal S(\mathbf w)|\delta_{\mathbf w}
	=(r+2-\delta_{\mathbf w})\delta_{\mathbf w}
	\le(r+1)\delta_{\mathbf w}.$
	Consequently
	\begin{equation}\label{diagonal:common-sums}
		\sum_{\mathbf w}|\mathcal S(\mathbf w)|\delta_{\mathbf w}
		\le(r+1)(m+H).
	\end{equation}
	Counting the $m-1-\eta_{\mathbf a}$ allowable partners of
	$\mathbf a$ gives
	\begin{equation}\label{diagonal:common-degree}
		\sum_{\mathbf w:\mathbf a\in\mathcal S(\mathbf w)}
		\delta_{\mathbf w}
		=(r+1)d_{\mathbf a}-(m-1-\eta_{\mathbf a})\ge1.
	\end{equation}
	Here the equality follows by summing
	$|\mathcal S(\mathbf w)|-1=r+1-\delta_{\mathbf w}$,
	and the inequality follows from the heavy owner of $\mathbf a$.
	
	We next show that some support $\mathcal S_0$ has size $r+2$.
	Otherwise every support has size at most $r+1$, and every exterior
	vector has a heavy coordinate.
	Each direction has at least
	$m-2\ge r(r+1)-1>r^2$ allowable partners, while each incident
	support contributes at most $r$ partners.
	Thus $d_{\mathbf a}\ge r+1$, and
	\[
		(r+1)m\le\sum_{\mathbf a}d_{\mathbf a}
		=\sum_{\mathbf w}|\mathcal S(\mathbf w)|\le(r+1)\sum_{\mathbf w}|\{i:|w_i|\ge2\}|
		=(r+1)m.
	\]
	Equality forces every support size and every degree to equal
	$r+1$. Every direction would then have $r(r+1)$ partners,
	contrary to the different required counts $m-1$ and $m-2$
	for the nonempty special and ordinary classes.
	Hence such a support $\mathcal S_0$ exists.
	
	By \eqref{diagonal:common-degree},
	$(r+1)d_{\mathbf a}\ge m-\eta_{\mathbf a}$.
	Since $m=r(r+1)+2-\varepsilon$, integrality gives
	$d_{\mathbf a}\ge r+1-\varepsilon\eta_{\mathbf a}$.
	Let $
	K=\sum_{\mathbf a\in\mathcal U_{\mathbf b}}
	(d_{\mathbf a}-r-1+\varepsilon\eta_{\mathbf a})$
	be the total increment above these bounds, and let $u$ count the
	special directions in $\mathcal S_0$.
	
	There are $m-2(r+2)+u$ directions of $\mathcal U_{\mathbf b}$
	outside $\mathcal S_0\cup(-\mathcal S_0)$, of which $r-1-u$
	are special.
	Each has degree at least $r+2$: its pairs with the $r+2$
	directions of $\mathcal S_0$ must occur in different supports,
	since another support cannot meet $\mathcal S_0$ twice.
	The $r+2-u$ directions in
	$(-\mathcal S_0)\cap\mathcal U_{\mathbf b}$ are ordinary.
	Each has degree at least $r+1$, by the same argument applied
	to its $r+1$ nonopposite partners in $\mathcal S_0$.
	Their contributions to $K$ therefore give
	\[
	\begin{aligned}
		K
		&\ge(1+\varepsilon)\{m-2(r+2)+u\}
		-\varepsilon(r-1-u)+\varepsilon(r+2-u)\\
		&=(1+\varepsilon)\{m-2(r+2)+u\}+3\varepsilon\\
		&=(1+\varepsilon)\{(r+1)(r-2)+u\}+\varepsilon,
	\end{aligned}
	\]
	where the last equality uses $\varepsilon^2=\varepsilon$.
	
	Finally, summing \eqref{diagonal:common-degree} over all
	directions and using the definition of $K$ gives
	\begin{align*}
		&\sum_{\mathbf w}|\mathcal S(\mathbf w)|\delta_{\mathbf w}
		=(r+1)\sum_{\mathbf a}d_{\mathbf a}
		-m(m-1)+(m-r+1),\\
		&\sum_{\mathbf a}d_{\mathbf a}=(r+1)m-\varepsilon(m-r+1)+K.
	\end{align*}
	Substituting these identities into
	\eqref{diagonal:common-sums}, with $m=r(r+1)+2-\varepsilon$,
	yields
	\[
	0\le
	(r+1)\{H-K+\varepsilon(r^2-r+2)\}+r-1-\varepsilon.
	\]
	The expression in braces is an integer, and
	$0\le r-1-\varepsilon<r+1$, so it must be nonnegative.
	The bound for $K$ now gives
	\[
	\begin{aligned}
		H\ge K-\varepsilon(r^2-r+2)
		\ge(r+1)(r-2)-3\varepsilon.
	\end{aligned}
	\]
\end{proof}

\begin{lemma}\label{foundation:thm:v5-diagonal-exclusions}
	For $r\ge4$ and $\varepsilon\in\{0,1\}$, there is no
	arbitrary-centre tiling when $2n=(r+1)^2-\varepsilon$.
\end{lemma}
\begin{proof}
	Suppose such a tiling exists.
	Extend $H$ by the same formula to all points whose covering
	vector has norm $r-1$, and set it to zero elsewhere.
	For an exterior centre $\mathbf c$, write
	$\mathbf x=\mathbf z-\mathbf c$ with $\|\mathbf x\|_1=r+2$,
	and let $\mathbf y=\mathbf z-\mathbf c_0$ be the covering vector
	of norm $r-1$.
	Separation gives
	\[
	2r+1\le\|\mathbf c_0-\mathbf c\|_1
	=\|\mathbf x-\mathbf y\|_1
	\le(r+2)+(r-1)=2r+1,
	\]
	so $\mathbf x,\mathbf y$ have disjoint signed supports.
	
	Expand
	$(|x_i|-2)_+=\sum_{j=3}^{r+2}\mathbf1_{|x_i|\ge j}$
	and apply
	Lemma~\ref{foundation:lem:v5-threshold-capacity}
	with $\mathcal I=\{r+2\}$, $\beta=r-1$, and no second mark.
	For $3\le j\le r+1$, put $\ell=r+1-j$.
	The residual ranks are $\ell+1,r-1$, and the shadow rank is
	$\ell$. Since $0\le\ell\le r-2$, the hypothesis holds, and
	the inverse weight is $
	\frac{\binom{r+\ell}{\ell+1}}{\binom{r+\ell}{\ell}}
	=\frac r{\ell+1}.$
	At $j=r+2$, the residual ranks are $0,r-1$ and the shadow rank
	is zero, giving weight one and capacity $T_0=2n$.
	Using
	$T_\ell=(\ell+1)\sigma_{\ell+1}-\ell\sigma_\ell$
	from Lemma~\ref{foundation:prop:v5-direction-sums},
	we obtain
	\begin{align*}
		\M(H)\le2n+r\sum_{\ell=0}^{r-2}\frac{T_\ell}{\ell+1}
		=2n+r\sigma_{r-1}
		+r\sum_{\ell=1}^{r-2}\frac{\sigma_\ell}{\ell+1}.
	\end{align*}
	Since $2n\ge r^2+2r$, the shell ratio gives
	$\sigma_{j-1}/\sigma_j\le j/(2n)\le1/(r+1)$
	for $1\le j\le r-1$. Hence
	\[
	\sum_{\ell=1}^{r-2}\frac{\sigma_\ell}{\ell+1}
	\le\frac{\sigma_{r-2}}{2(1-1/(r+1))}
	\le\frac{\sigma_{r-1}}{2r},
	\]
	and therefore
	\[
	\M(H)\le2n+\left(r+\frac12\right)\sigma_{r-1}.
	\]
	
	Put $G=2^{r-1}\binom n{r-1}$.
	The all-unit covering vectors of norm $r-1$ give $G$ disjoint
	translates of the centre set, so the corresponding point set
	has normalized mean $G$.
	Since $r-1<n$, the $n$ distinct cyclic intervals of length
	$r-1$ show that $\binom n{r-1}\ge n$. Thus
	$G\ge2^{r-2}(2n)\ge4(2n).$
	For $0\le t\le r-2$,
	\[
	\frac{\binom n{r-1-t}}{\binom n{r-1}}
	=\prod_{i=0}^{t-1}\frac{r-1-i}{n-r+2+i}
	\le\left(\frac{r-1}{n-r+2}\right)^t.
	\]
	Also $\binom{r-2}{t}\le(r-2)^t$.
	Grouping the shell by support size and using
	$2n-2r+4\ge r^2+4$, we obtain
	\[
	\begin{aligned}
		\frac{\sigma_{r-1}}G
		=\sum_{t=0}^{r-2}
		2^{-t}\frac{\binom n{r-1-t}}{\binom n{r-1}}
		\binom{r-2}{t}
		\le\sum_{t=0}^{r-2}
		\left(\frac{(r-1)(r-2)}{r^2+4}\right)^t
		<\frac{2r+1}{6}.
	\end{aligned}
	\]
	The strict inequality comes from extending the finite geometric
	series, and the last inequality is equivalent to $7r\ge22$.
	The mean bound now gives
	\[
	\frac{\M(H)}G
	<\frac14+\left(r+\frac12\right)\frac{2r+1}{6}
	=\frac{r^2+r+1}{3}.
	\]
	
	On the other hand,
	Lemma~\ref{foundation:lem:v5-diagonal-excess-lower}
	applies at every point with an all-unit covering vector of norm
	$r-1$, and $H$ is nonnegative elsewhere.
	Consequently
	\[
	\begin{aligned}
		\frac{\M(H)}G
		\ge(r+1)(r-2)-3\varepsilon
		\ge(r+1)(r-2)-3
		\ge\frac{r^2+r+1}{3},
	\end{aligned}
	\]
	contradicting the strict upper bound.
\end{proof}

\section{Contact identities and small radii}\label{sec:small}
We first develop the contact identities shared by the small-radius
arguments and the later algebraic construction. The modular and
profile-square arguments then exclude radii two, three, and four.
Finally, we combine these cases with the preceding geometric estimates
to state the complete parameter reduction.

\subsection{Contact identities and polynomial tools}
We encode the two adjacent contact shells by a common correlation
function and derive the identities used in the local exclusions.

Let $\tau(\mathbf v)$ be the partition formed by the nonzero
absolute coordinates of $\mathbf v$.
For a partition $\lambda$, write $|\lambda|=\sum_i\lambda_i$,
$\ell(\lambda)$, and $m_a(\lambda)$ for its size, length, and
multiplicity of the part $a$, respectively.
Thus $\lambda\vdash j$ means $|\lambda|=j$.
Subscripts such as $31^2$ abbreviate $(3,1,1)$;
$\lambda\cup(1)$ adjoins a unit part, and
$\lambda\setminus(1)$ removes one.
The empty partition has size and length zero.
The signed-permutation orbit of $\lambda$ has size
\[
g_\lambda(n)
=2^{\ell(\lambda)}\binom n{\ell(\lambda)}
\frac{\ell(\lambda)!}{\prod_{a\ge1}m_a(\lambda)!},
\]
with $g_\lambda(n)=0$ when $\ell(\lambda)>n$.
For $f=\mathbf1_{\mathcal C}$, define
$\varrho(\mathbf v)
	=\M_{\mathbf x}\bigl(f(\mathbf x)f(\mathbf x+\mathbf v)\bigr),$
\begin{align*}
	&\bar X_\lambda
	=\sum_{\tau(\mathbf v)=\lambda}\varrho(\mathbf v)
	\qquad(|\lambda|=2r+1),\\
	&\bar Z_\lambda
	=\sum_{\tau(\mathbf v)=\lambda}\varrho(\mathbf v)
	\qquad(|\lambda|=2r+2).
\end{align*}
Since $f^2=f$ and $\M f=1$, one has $\varrho(\boldsymbol0)=1$.
Positivity and translation invariance give
$\varrho(\mathbf v)=\varrho(-\mathbf v)\ge0$,
while packing gives $\varrho(\mathbf v)=0$ for
$0<\|\mathbf v\|_1\le2r$.

The nonzero vectors of $B_n(r)$ occur in opposite pairs,
so $V_r$ is odd. The coprime-dilation lemma therefore gives
$\mathsf B_r^{[2]}f=1$.
Two centres cannot differ by $2\mathbf v\ne\boldsymbol0$
with $\|\mathbf v\|_1\le2r$: writing $\mathbf v$ as the difference
of two vectors in $B_n(r)$ would make their doubled tiles intersect.
Hence $\varrho(2\mathbf v)=0$ in this range.
Every all-even vector of norm $2r+2$ is twice a vector of norm
$r+1\le2r$, so every all-even $\bar Z$ profile vanishes.
Multiplying either tiling identity at $\mathbf x+\mathbf v$
by $f(\mathbf x)$ and applying $\M_{\mathbf x}$ gives
$\mathsf B_r\varrho=1$ and $
\mathsf B_r^{[2]}\varrho=1.$

We work in the stable symmetric-function ring
$\mathbb R[p_1,p_2,\ldots]$.
Let $h_j$, $e_j$, $p_j$, and $m_\lambda$ denote the complete,
elementary, power-sum, and monomial symmetric functions.
Their formal alphabet is $\mathbf x=(x_1,x_2,\ldots)$,
with $p_j=\sum_i x_i^j$ and weighted degree $\deg p_j=j$.
Set $h_0=e_0=m_\varnothing=1$ and $h_j=e_j=0$ for $j<0$.
The notation $[m_\lambda]G$ denotes the coefficient of $m_\lambda$,
without a signed-orbit factor. Put
\[
D=\frac{\partial}{\partial p_1},\qquad
\calL(G)=\sum_\lambda[m_\lambda]G\,g_\lambda(n),\qquad
\Lam(G)=\calL(G+DG).
\]
The functional $\calL$ is linear across weighted degrees and is
not multiplicative. Differentiation is performed in the stable
ring before terms of length greater than $n$ are discarded.
For homogeneous $G$ of degree $2r+1$ or $2r+2$, respectively,
define
\[
\langle G,\bar X\rangle
=\sum_{\lambda\vdash2r+1}[m_\lambda]G\,\bar X_\lambda,
\qquad
\langle G,\bar Z\rangle
=\sum_{\lambda\vdash2r+2}[m_\lambda]G\,\bar Z_\lambda.
\]
The length restrictions are automatic, since profiles of length
greater than $n$ have empty orbits.

The following identities relate adjacent degrees without
introducing multiplicity factors.
\begin{lemma}\label{common:lem:derivative}
	In the stable ring,
	$Dm_\lambda=m_{\lambda\setminus(1)}$ and $Dh_j=h_{j-1}$,
	where the first expression is zero if $\lambda$ has no unit part.
	Consequently, $D$ preserves nonnegative coefficients in the
	monomial basis.
\end{lemma}
\begin{proof}
	Adjoin a formal variable $\varepsilon$ to the alphabet.
	Since
	$p_j(\mathbf x,\varepsilon)=p_j(\mathbf x)+\varepsilon^j$,
	the coefficient of $\varepsilon$ in $G(\mathbf x,\varepsilon)$
	is $DG$.
	In $m_\lambda(\mathbf x,\varepsilon)$, the terms linear in
	$\varepsilon$ are precisely $\varepsilon$ times the distinct
	monomials of $m_{\lambda\setminus(1)}$, each occurring once.
	There are no such terms if $\lambda$ has no unit part.
	In $h_j(\mathbf x,\varepsilon)$, the terms linear in
	$\varepsilon$ are $\varepsilon h_{j-1}(\mathbf x)$.
	These observations prove both formulas and the assertion
	about nonnegative coefficients.
\end{proof}

For $\mu\vdash r$ and $\nu\vdash r+1$, define
\[
 I_\mu(\mathbf x)=\sum_{\tau(\mathbf u)=\mu}f(\mathbf x-\mathbf u),\qquad
 F_\nu(\mathbf x)=\sum_{\tau(\mathbf v)=\nu}f(\mathbf x-\mathbf v),\qquad
 I_\partial=\sum_\mu I_\mu,\qquad F=\sum_\nu F_\nu.
\]
The $I_\mu$ are disjoint boundary-covering indicators; the signed-support
partition shows that all $F_\nu$ vanish off the boundary and that
\[
 \ell(\mu)+\sum_{\nu\vdash r+1}\ell(\nu)F_\nu=2n
 \quad\text{on }I_\mu=1,\qquad
 \M I_\mu=g_\mu(n),\quad \M F_\nu=g_\nu(n).
\]
The last two identities follow by translation invariance and $\M f=1$.
At most one exterior residual is all-even: two would have an all-even
centre difference whose divided norm is at most $r+1\le2r$, contrary
to doubled packing. For a homogeneous stable $A$ of degree $r+1$, put
\[
 F_A=\sum_\nu[m_\nu]A\,F_\nu,\qquad
 I_{DA}=\sum_\mu[m_\mu]DA\,I_\mu,\qquad Y_A=F_A+I_{DA}.
\]
Stable differentiation precedes deletion of profiles longer than $n$.
For homogeneous $Q$ of degree $2r+2$, define
\[
 \mathfrak C(Q)=\langle DQ,\bar X\rangle+\langle Q,\bar Z\rangle,
 \qquad
 \Pi_{d,j}=\sum_{\substack{\lambda\vdash d\\|\{i:2\nmid\lambda_i\}|=j}}m_\lambda.
\]

\begin{lemma}\label{common:contact-functional}
Assume $r\ge2$. If every stable non-all-even coefficient of $Q$ is nonnegative, then
$\mathfrak C(Q)\ge0$. For homogeneous $V$ of degree $r+2$ and
homogeneous $A,B$ of degree $r+1$, with arbitrary signed coefficients,
\begin{align}
 \mathfrak C(Vh_r)&=\Lam(V),\qquad
 \mathfrak C(\Pi_{2r+2,2})=(2n)^2/2,\\
 \M(Y_AY_B)&=\mathfrak C(AB)
       +\sum_{\nu\vdash r+1}[m_\nu]A[m_\nu]B\,\Lam(m_\nu),
 \qquad \M(Y_A)=\Lam(A).
 \label{common:profile-bilinear}
\end{align}
If $A$ has only zero-or-one monomial coefficients, then for every
integer $a$,
\begin{equation}\label{common:profile-square-mean}
 0\le\M\!\left(I_\partial\binom{Y_A-a}{2}\right)
 =\tfrac12\mathfrak C(A^2)-a\Lam(A)
                         +\tfrac12a(a+1)\sigma_r.
\end{equation}
If $B$ has nonnegative monomial coefficients and
\[
 \beta\le\left\lceil\frac{2n+1-\ell(\nu)}{r+1}\right\rceil
 \quad\text{whenever }[m_\nu]B>0,
\]
then
\begin{equation}\label{common:target-pair-mean}
 Y_B(F-\beta)\ge0,\qquad
 \M\bigl(Y_B(F-\beta)\bigr)=\mathfrak C(Bh_{r+1})-\beta\Lam(B).
\end{equation}
\end{lemma}
\begin{proof}
For every physical $\mu\vdash2r+1$, stable unit deletion gives
\[
 [m_\mu]DQ=[m_{\mu\cup(1)}]Q\ge0.
\]
The appended-unit profile is not all-even, including when its length
is $n+1$. Thus $\langle DQ,\bar X\rangle\ge0$. All-even profiles
have zero $\bar Z$ weight, while every other coefficient is
nonnegative, so $\langle Q,\bar Z\rangle\ge0$. This proves the first
assertion before finite-dimensional specialization.

We use one subdivision rule throughout. For arbitrary homogeneous
symmetric functions $A,B$ and the coordinate vector
$\boldsymbol\lambda=(\lambda_1,\ldots,\lambda_{\ell(\lambda)})$,
\begin{equation}\label{small:coefficient-rule}
 [m_\lambda](AB)=\sum_{0\le a_i\le\lambda_i}
 [m_{\tau(\mathbf a)}]A\,
 [m_{\tau(\boldsymbol\lambda-\mathbf a)}]B.
\end{equation}
Indeed each contribution to the coefficient of
$x_1^{\lambda_1}\cdots x_{\ell(\lambda)}^{\lambda_{\ell(\lambda)}}$
assigns exponent $a_i$ to the first factor and $\lambda_i-a_i$ to
the second. Wrong-degree coefficients are zero. In particular,
for centrally symmetric finite kernels $K,L$, translation invariance gives
\[
 \M((Kf)(Lf))
 =\sum_{\mathbf u,\mathbf v}K(\mathbf u)L(\mathbf v)
                   \varrho(\mathbf u-\mathbf v)
 =\sum_{\mathbf w}(K*L)(\mathbf w)\varrho(\mathbf w).
\]

For $\|\mathbf x\|_1=r+1$, each surviving contact in
$(\mathsf B_r\varrho)(\mathbf x)=1$ has norm $2r+1$; equality in
the triangle inequality makes $\mathbf x$ a geodesic rank-$(r+1)$
subvector. For $\|\mathbf x\|_1=r+2$, surviving contact norms are
$2r+1$ or $2r+2$. Their distances from $\mathbf x$ equal $r-1$ or
$r$, respectively: the triangle inequality gives these lower bounds,
and parity excludes a larger permitted distance in the first case.
Multiply the equations on these two shells by
$[m_{\tau(\mathbf x)}]DV$ and $[m_{\tau(\mathbf x)}]V$, respectively,
and sum over $\mathbf x$. For each fixed contact vector, its geodesic
subdivisions are counted by \eqref{small:coefficient-rule}; hence
\[
 \calL(DV)=\langle(DV)h_r,\bar X\rangle,\qquad
 \calL(V)=\langle Vh_{r-1},\bar X\rangle+\langle Vh_r,\bar Z\rangle.
\]
Linearity allows signed coefficients; multiplication cannot shorten
support, so profiles longer than $n$ contribute nothing. Since
$D(Vh_r)=(DV)h_r+Vh_{r-1}$, addition proves
$\mathfrak C(Vh_r)=\calL(DV)+\calL(V)=\Lam(V)$.

At a signed unit point $\mathbf a$, a surviving term of
$(\mathsf B_r^{[2]}\varrho)(\mathbf a)=1$ has displacement
$\mathbf c\equiv\mathbf a\pmod2$ and
$\|\mathbf c-\mathbf a\|_1\le2r$. Packing forces norm $2r+1$,
one odd coordinate and its prescribed sign. Conversely, each such
contact contributes once. At a signed two-axis unit point the same
argument forces norm $2r+2$, exactly two odd coordinates and their
prescribed signs. Summing over these two sets of points gives
\[
 \langle\Pi_{2r+1,1},\bar X\rangle=2n,\qquad
 \langle\Pi_{2r+2,2},\bar Z\rangle=(2n)(2n-2)/2.
\]
Appending and deleting a unit are inverse operations, so
$D\Pi_{2r+2,2}=\Pi_{2r+1,1}$. Consequently
$\mathfrak C(\Pi_{2r+2,2})=2n+(2n)(2n-2)/2=(2n)^2/2$.

Write $a_\nu=[m_\nu]A$, $b_\nu=[m_\nu]B$.
For two exterior vectors, a nonzero difference has even norm at most
$2r+2$ and at least $2r+1$, hence exactly $2r+2$.
For a covering and exterior vector, their different norms exclude
the zero difference and packing forces norm $2r+1$.
In each case equality in the triangle inequality gives precisely the
subdivisions counted in \eqref{small:coefficient-rule}. The kernel
identity therefore gives
\begin{align*}
 \M(F_AF_B)&=\sum_\nu a_\nu b_\nu\calL(m_\nu)
                                      +\langle AB,\bar Z\rangle
\end{align*}
and 
\[
 \M(F_AI_{DB}+F_BI_{DA})=\langle A(DB)+B(DA),\bar X\rangle.
 \]
The covering indicators are disjoint. Unit deletion is injective
among profiles containing a unit, with inverse given by appending one;
therefore
\[
 \M(I_{DA}I_{DB})=\sum_\nu a_\nu b_\nu\calL(Dm_\nu).
\]
Adding these three identities and using $D(AB)=A(DB)+B(DA)$ proves
the bilinear formula. The field means already give
$\M(Y_A)=\calL(A)+\calL(DA)=\Lam(A)$.

When every $a_\nu\in\{0,1\}$, the bilinear diagonal for $B=A$
is $\sum_\nu a_\nu^2\Lam(m_\nu)=\Lam(A)$. Thus
\[
 \M(Y_A^2)=\mathfrak C(A^2)+\Lam(A),\qquad
 2\M\!\left(I_\partial\binom{Y_A-a}{2}\right)
 =\M(Y_A^2)-(2a+1)\M(Y_A)+a(a+1)\sigma_r.
\]
The product of consecutive integers $Y_A-a$ and $Y_A-a-1$ is
nonnegative on the boundary; off it the leading indicator is zero.
Substitution proves \eqref{common:profile-square-mean}.

If a supported profile $\nu$ contributes to $Y_B$, then either an
exterior vector has defect $r+1-\ell(\nu)$, or its unit-deleted
covering vector has defect $r-\ell(\nu\setminus(1))=r+1-\ell(\nu)$.
The signed-support partition therefore gives
\[
 \begin{aligned}
 (r+1)F
 &=2n-r+(r-s(\mathbf b))
           +\sum_{\mathbf v\in\mathcal F_{\mathbf z}}(r+1-s(\mathbf v))\\
 &\ge2n-r+(r+1-\ell(\nu))=2n+1-\ell(\nu),\\
 F&\ge\left\lceil\frac{2n+1-\ell(\nu)}{r+1}\right\rceil\ge\beta.
 \end{aligned}
\]
This includes length $n+1$: its exterior field vanishes, but the
unit-deleted covering field may occur. Longer profiles contribute
neither field. Since $B$ is monomial-nonnegative, $Y_B\ge0$, so
$Y_B(F-\beta)\ge0$. Finally
$Y_{h_{r+1}}=F+I_\partial$ and $I_\partial Y_B=Y_B$, whence
\[
 \M(Y_BF)=\M(Y_BY_{h_{r+1}})-\M(Y_B)
 =\mathfrak C(Bh_{r+1})+\Lam(B)-\Lam(B)
 =\mathfrak C(Bh_{r+1}).
\]
Subtracting $\beta\M(Y_B)=\beta\Lam(B)$ proves
\eqref{common:target-pair-mean}.
\end{proof}

\begin{corollary}\label{common:square-target-certificate}
Let $A$ have zero-or-one monomial coefficients and let $B$ have
nonnegative monomial coefficients, both of degree $r+1$. Choose
$\kappa,b\ge0$, $a\in\Z$, and $\beta\in\Z$ satisfying the target
bound in Lemma~\ref{common:contact-functional}. For homogeneous $W$
of degree $r+2$ and $\alpha\in\mathbb R$, put
\[
 Q=Wh_r-\tfrac\kappa2A^2-bBh_{r+1}+\alpha\Pi_{2r+2,2}.
\]
If every stable non-all-even coefficient of $Q$ is nonnegative, then
\begin{equation*}
 0\le\Lam(W)-\kappa a\Lam(A)+\tfrac12\kappa a(a+1)\sigma_r
       -b\beta\Lam(B)+\alpha(2n)^2/2.
\end{equation*}
\end{corollary}
\begin{proof}
Put $P=\kappa I_\partial\binom{Y_A-a}{2}+bY_B(F-\beta)\ge0$.
Linearity of $\mathfrak C$ and the three identities in the lemma give
\[
 \begin{aligned}
 0&\le\M P+\mathfrak C(Q)\\
  &=\tfrac\kappa2\mathfrak C(A^2)-\kappa a\Lam(A)
       +\tfrac12\kappa a(a+1)\sigma_r
       +b\mathfrak C(Bh_{r+1})-b\beta\Lam(B)\\
  &\quad+\Lam(W)-\tfrac\kappa2\mathfrak C(A^2)
       -b\mathfrak C(Bh_{r+1})+\alpha(2n)^2/2\\
  &=\Lam(W)-\kappa a\Lam(A)+\tfrac12\kappa a(a+1)\sigma_r
       -b\beta\Lam(B)+\alpha(2n)^2/2.
 \end{aligned}
\]
When $b=0$, take $B=0$; its support condition is empty.
\end{proof}

Let $I$ be the identity operator and put
\[
 \mathsf S_a=I+\sum_{i=1}^n([a\mathbf e_i]+[-a\mathbf e_i])
 \qquad(a\ge1).
\]
In $\mathsf S_1^2$, the coefficients at
$\boldsymbol0$, $\pm\mathbf e_i$, $\pm2\mathbf e_i$, and
$\pm\mathbf e_i\pm\mathbf e_j$ ($i\ne j$) are respectively
$2n+1$, $2$, $1$, and $2$: the zero displacement has the zero pair and
$2n$ opposite ordered pairs; a unit displacement has two placements
of the zero mark; a doubled unit has one identical pair; and two
distinct signed units have two orders. All other coefficients vanish.
These are also the coefficients of
$2\mathsf B_2+(2n)I-\mathsf S_2$. Therefore, including every dilation,
\[
 \mathsf S_1^2+\mathsf S_2=2\mathsf B_2+(2n)I,\qquad
 \mathsf S_a^2+\mathsf S_{2a}=2\mathsf B_2^{[a]}+(2n)I
 \quad(a\ge1).
\]

\subsection{The radius-two exclusion}
The quadratic star identity supplies both modular constraints and both
second moments. Define
\[
 t_a=\mathsf S_af,\qquad g=\mathsf S_2\mathsf S_1f,
 \qquad k=\mathsf S_4\mathsf S_1f.
\]
All power-of-two dilations below are permitted because $V_2$ is odd.

\begin{lemma}\label{lem:rank5moment}
Suppose $r=2$. Then, pointwise,
\begin{align}
 g+t_3&\equiv2-2n+2nt_1\pmod3,
 \label{r2u:three}\\
 k+t_5&\equiv2(2n+1)^3+((2n)^2+2n)t_1-4ng\pmod5.
 \label{eq:rank5r2}
\end{align}
Moreover,
\begin{equation}\label{eq:rank5bounds}
 \M\binom g2=(2n)^2,\qquad
 (2n)^2\le\M\binom k2\le (2n)^2+n.
\end{equation}
\end{lemma}
\begin{proof}
The quadratic identity and the ordinary and doubled tiling equations give
\[
 \mathsf S_1^2f=2+2nf-t_2,\qquad
 \mathsf S_2^2f=2+2nf-t_4.
\]
Every star sends the constant function $1$ to $2n+1$.
Commutation therefore gives
\begin{align*}
 \mathsf S_1^3f&=2(2n+1)+2nt_1-g,\\
 \mathsf S_1^4f
 &=2(2n+1)^2+2n(2+2nf-t_2)
                   -\mathsf S_2(2+2nf-t_2)\\
 &=2(2n+1)^2+4n+(2n)^2f-4nt_2-2(2n+1)
                   +(2+2nf-t_4)\\
 &=2(2n+1)^2+((2n)^2+2n)f-4nt_2-t_4,\\
 \mathsf S_1^5f&=2(2n+1)^3+((2n)^2+2n)t_1-4ng-k.
\end{align*}
For $p=3,5$, the multinomial theorem gives
$\mathsf S_1^p\equiv\mathsf S_p\pmod p$: every mixed multinomial
coefficient is divisible by $p$, while the pure powers give the
translations in $\mathsf S_p$. The first and last equalities thus give
\eqref{r2u:three} and \eqref{eq:rank5r2}, using
$2(2n+1)\equiv2-2n\pmod3$.

For $a\in\{2,4\}$ put $H_a=\mathsf S_a\mathsf S_1f$.
Coprime dilation and averaging give
$\mathsf B_2^{[a]}\varrho=1$. The quadratic identity then gives
\[
 \mathsf S_a^2\varrho=2+2n\varrho-\mathsf S_{2a}\varrho,
 \qquad 0\le\mathsf S_{2a}\varrho
       \le\mathsf B_2^{[a]}\varrho=1.
\]
Ordinary and $a$-dilated packing give $\varrho(\boldsymbol0)=1$,
$(\mathsf S_2\varrho)(\boldsymbol0)=1$, and
$(\mathsf S_{2a}\varrho)(\boldsymbol0)=1$; hence
$(\mathsf S_a^2\varrho)(\boldsymbol0)=2n+1$.
Writing $C_a=(\mathsf S_2\mathsf S_{2a}\varrho)(\boldsymbol0)$,
we obtain
\[
 C_a=1+\sum_{\mathbf e\in\mathcal U}
                (\mathsf S_{2a}\varrho)(2\mathbf e),
 \qquad 1\le C_a\le 2n+1.
\]
Put $\mathcal U_0=\mathcal U\cup\{\boldsymbol0\}$. Expanding the two copies
of $H_a$ and using the symmetry of $\mathcal U_0$ gives the complete
second-moment calculation:
\begin{align*}
 \M(H_a^2)
 &=\sum_{\mathbf v,\mathbf w,\mathbf u,\mathbf z\in\mathcal U_0}
       \varrho\bigl(a(\mathbf v-\mathbf w)+\mathbf u-\mathbf z\bigr)
   =(\mathsf S_a^2\mathsf S_1^2\varrho)(\boldsymbol0)\\
 &=2(2n+1)^2+2n(2n+1)
                 -(\mathsf S_2\mathsf S_a^2\varrho)(\boldsymbol0)\\
 &=2(2n+1)^2+2n(2n+1)-\{2(2n+1)+2n-C_a\}\\
 &=3(2n)^2+4n+C_a,\\
 \M\binom{H_a}{2}
 &=\frac{3(2n)^2+4n+C_a-(2n+1)^2}{2}
   =(2n)^2+\frac{C_a-1}{2}\in[(2n)^2,(2n)^2+n].
\end{align*}
Here $\M(H_a)=(2n+1)^2\M(f)=(2n+1)^2$.
For $a=2$, every term of $C_2-1$ is
$\varrho(2\mathbf e-4\mathbf v)$ with
$\mathbf e\in\mathcal U$, $\mathbf v\in\mathcal U_0$.
The vector $\mathbf e-2\mathbf v$ is nonzero and has norm at most
three, so doubled packing makes every term zero. Thus $C_2=1$.
Since $H_2=g$ and $H_4=k$, this proves \eqref{eq:rank5bounds}.
\end{proof}

The following congruence bounds convert the moment identities into
upper bounds.
\begin{lemma}\label{unified:endpoint-product}
Let $G$ be a nonnegative integer, $t\in\{0,1\}$, $p\ge3$, and
$G+t\equiv2\pmod {p}$. If $G\le B$, then
\[
 G+t-2\ge0,\qquad
 0\le(G-1)(G-2)\le(B-1)(G+t-2).
\]
\end{lemma}
\begin{proof}
The integer $G+t-2$ is a multiple of $p$ and at least $-2$, so it
is nonnegative. The congruence also excludes $G=0$.
At $G=1$ the product is zero. At $G\ge2$, use
$0\le G-1\le B-1$ and $0\le G-2\le G+t-2$.
\end{proof}

\begin{lemma}\label{lem:rank5transfer}
Let $p\ge2$ be an integer. Suppose $H,T$ are nonnegative
integer-valued local fields and $H+T\equiv\theta\pmod {p}$, where
$0\le\theta<p$. Put
\[
 w_p(a)=\frac{a(p-a)}{2}\quad(0\le a<p).
\]
Then, pointwise,
\[
 \theta\le H+T,\qquad
 \frac{p-1}{2}(H-T)-\binom {H}{2}\le w_p(\theta).
\]
When $\theta=0$, equality in the second inequality holds exactly for
\[
 (H,T)\in\{(0,0),(p,0),(p-1,1)\}.
\]
\end{lemma}
\begin{proof}
Write $H+T=pj+\theta$ with $j\in\mathbb Z_{\ge0}$; this gives
$H+T\ge\theta$. Define the discrepancy
\[
 2K=2w_p(\theta)-(p-1)(H-T)+H(H-1)
     =\theta(p-\theta)+H(H-p)+(p-1)T.
\]
Substituting $T=pj+\theta-H$ gives
\begin{align*}
 2K
 &=H^2-(2p-1)H+p(p-1)j+\theta(2p-1-\theta)\\
 &=\begin{cases}
 (\theta-H)(2p-1-\theta-H),&j=0,\\
 (H-p)(H-p+1)+p(p-1)(j-1)
      +\theta(2p-1-\theta),&j\ge1.
 \end{cases}
\end{align*}
If $j=0$, then $0\le H\le\theta\le p-1$, so both factors
are nonnegative. If $j\ge1$, the product $(H-p)(H-p+1)$
is nonnegative for integral $H$, and both remaining terms are
nonnegative. Thus $K\ge0$, which is the second pointwise inequality;
linearity and positivity of the mean give the corresponding averaged inequalities.

At $\theta=0$, equality in the second inequality is equivalent to $K=0$.
If $j=0$, then $H=T=0$. If $j\ge1$, the displayed nonnegative
sum vanishes precisely when $j=1$ and $H-p\in\{0,-1\}$.
Since $T=p-H$, these alternatives give $(p,0)$ and $(p-1,1)$.
This proves the complete equality classification at the same time.
\end{proof}

For radius two,  we have the following result.

\begin{lemma}\label{radius2:thm:radius-two}
For every integer $n\ge3$, the Lee ball $B_n(2)$ does not tile
$\mathbb Z^n$ by translations with an arbitrary centre set.
\end{lemma}
\begin{proof}[Proof of Lemma~\ref{radius2:thm:radius-two}]
Suppose otherwise. Set $r=2$, so ${V_2}=2n^2+2n+1$,
and retain $t_a,g,k$ from the preceding lemma.
The dilation lemma gives
\[
({\mathcal C}-{\mathcal C})\cap aB_n(4)=\{\boldsymbol0\},\qquad0\le t_a\le1
 \quad(\gcd(a,{V_2})=1).
\]
Since ${V_2}$ is odd and has residues $1,2,1$ for $n\equiv0,1,2\pmod{3}$,
this applies to $a=1,2,3,4,8$; always $0\le t_a\le {(2n+1)}$.
Each star is a sum of ${(2n+1)}$ translations, so
\begin{equation}\label{r2u:means}
 \M(1)={V_2},\quad\M(f)=1,\quad\M(t_a)={(2n+1)},
 \qquad\M(g)=\M(k)={(2n+1)}^2.
\end{equation}
For the pointwise steps we require the following local facts
\begin{equation}\label{r2u:local}
 t_1t_3=f,\quad g\ge t_3,\quad g=1\text{ on }{\mathcal C},\quad
 g=2\text{ on }\{t_1-f=1\},\quad g\le n+1,\quad k\le {(2n+1)}.
\end{equation}
Indeed simultaneous representations
$\mathbf x=\mathbf c+\mathbf v=\mathbf d+3\mathbf w$ with
$\mathbf v,\mathbf w\in\mathcal U_0$ give
$\|\mathbf c-\mathbf d\|_1\le4$, hence
$\mathbf c=\mathbf d$ and $\mathbf v=3\mathbf w=\boldsymbol0$.
Each signed axial mark $3\mathbf w$ has the representation
$2\mathbf w+\mathbf w$, including the zero mark; hence
$\mathsf S_2\mathsf S_1\succeq \mathsf S_3$ coefficientwise and $g\ge t_3$.
At a centre only the zero pair in $\mathsf S_2\mathsf S_1f$ contributes. At $\mathbf x=\mathbf c+\mathbf e$,
$\mathbf e\in \mathcal U$, packing forces $2\mathbf v+\mathbf w=\mathbf e$;
its two solutions $(\boldsymbol0,\mathbf e)$ and $(\mathbf e,-\mathbf e)$
give $g=2$.

On the boundary the exterior residuals of types $3$ and $21$ each
contribute one representation to $\mathsf S_2\mathsf S_1f$;
the covering residual contributes one exactly when its type is $2$.
Thus
\[
 g=F_3+F_{21}+I_2,\qquad t_3=F_3
 \quad\text{on the boundary}.
\]
The signed-support partition makes these exterior supports pairwise
disjoint and disjoint from the covering support. Put $v=F_3+F_{21}$.
Dilation-three packing gives $F_3\le1$, so these exterior residuals
use $F_3+2F_{21}=2v-F_3\ge2v-1$ signed directions. For covering type $[2]$,
$2v-1\le {(2n)}-1$ and $g=v+1\le n+1$; for covering type $[1,1]$,
$2v-1\le {(2n)}-2$ and $g=v\le n-1$. The case $v=0$ is immediate.
Finally $k$ is a sum of $2n+1$ indicators, proving \eqref{r2u:local}.

Lemma~\ref{lem:rank5moment} supplies both second moments. The boundary means needed below follow directly by removing
the centre and neighbour states in \eqref{r2u:local}:
\begin{align*}
 \M(I_\partial)&=\sigma_2=\frac{(2n)^2}{2},\\
 \M(I_\partial g)&=\M(g)-\M(f)-2\M(t_1-f)
              =(2n+1)^2-1-4n=(2n)^2,\\
 \M(I_\partial t_3)&=\M(t_3)-\M(f)=2n,\\
 \M\!\left(I_\partial\binom g2\right)
       &=\M\binom g2-\M(t_1-f)=(2n)^2-2n.
\end{align*}

If $3\mid n$, the boundary tuple has $g+t_3\equiv2\pmod{3}$.
Lemma~\ref{unified:endpoint-product} with $B=n+1$ gives
\[
 g+t_3-2\ge0,\qquad(g-1)(g-2)\le n(g+t_3-2).
\]
The same inequalities hold in the interior: at a centre $(g,t_3)=(1,1)$,
and at a neighbour $(g,t_3)=(2,0)$, so both sides vanish.
Applying the mean and \eqref{r2u:means} and \eqref{eq:rank5bounds} gives
\[
 \M(g+t_3-2)=(2n+1)^2+(2n+1)-2V_2=2n,\qquad
 4n^2-4n=\M((g-1)(g-2))\le2n^2,
\]
contrary to $n\ge3$. Thus $3\nmid n$.

A single shifted cyclic-transfer gap determines both residue distributions.
Put
\[
 \varepsilon=\begin{cases}
 0,&n\equiv1\pmod{3},\\
 1,&n\equiv2\pmod{3}.
 \end{cases}
 \qquad
 \varepsilon^2=\varepsilon,\qquad
 2-{(2n)}\equiv\varepsilon\pmod{3}.
\]
On the boundary, \eqref{r2u:three} gives
$g+t_3\equiv\varepsilon\pmod{3}$. We also have
$g-\varepsilon\ge0$ there. This is immediate when $\varepsilon=0$.
When $\varepsilon=1$, the only excluded possibility is $g=0$;
then $g\ge t_3\ge0$ would force $t_3=0$, contradicting
$g+t_3\equiv1\pmod{3}$. Thus the zero-residue equality clause of Lemma~\ref{lem:rank5transfer},
with $p=3$, $H=g-\varepsilon$, and $T=t_3$, proves
\[
 P_\varepsilon
 :=\frac{(g-\varepsilon)(g-\varepsilon-3)}{2}+t_3\ge0
 \quad\hbox{on the boundary},
\]
with precisely the following zero states:
\[
 (g,t_3)=(\varepsilon,0),\quad
 (\varepsilon+3,0),\quad(\varepsilon+2,1).
\]
On the boundary, expand the same gap as
\[
 P_\varepsilon=\binom g2-(1+\varepsilon)g
                +\frac{\varepsilon(\varepsilon+3)}2+t_3.
\]
Since $\varepsilon\in\{0,1\}$, its constant term is $2\varepsilon$.
The four boundary means just computed give
\begin{align*}
 \M(I_\partial P_\varepsilon)
 &=(2n)^2-2n-(1+\varepsilon)(2n)^2
             +2\varepsilon\frac{(2n)^2}{2}+2n=0.
\end{align*}

The same zero-gap identity computes every bounded function of $g,t_1$.
For such a function $\Phi$, write
\[
 u=\Phi(1,1),\quad v=\Phi(2,1),\quad
 x=\Phi(\varepsilon,0),\quad y=\Phi(\varepsilon+2,0),\quad
 z=\Phi(\varepsilon+3,0),
\]
and, on the boundary, put
\[
 \mathcal I\Phi
 =x+\frac{z-x}{3}(g-\varepsilon)
       +\left(y-\frac{x+2z}{3}\right)t_3.
\]
At $(g,t_3)=(\varepsilon,0),(\varepsilon+3,0),(\varepsilon+2,1)$
the right side is, respectively,
\[
 x,\qquad x+(z-x)=z,\qquad
 x+\frac23(z-x)+y-\frac{x+2z}{3}=y.
\]
Thus $\Phi(g,0)-\mathcal I\Phi=0$ wherever $P_\varepsilon=0$.
Moreover $P_\varepsilon$ is integer-valued, since
$(g-\varepsilon)(g-\varepsilon-3)\equiv
(g-\varepsilon)(g-\varepsilon-1)\equiv0\pmod2$.
All fields are bounded. With
$C=\|I_\partial(\Phi(g,0)-\mathcal I\Phi)\|_\infty$, therefore,
\[
 \bigl|I_\partial(\Phi(g,0)-\mathcal I\Phi)\bigr|
 \le C I_\partial P_\varepsilon,
 \qquad
 \bigl|\M(I_\partial\Phi)-\M(I_\partial\mathcal I\Phi)\bigr|
 \le C\M(I_\partial P_\varepsilon)=0.
\]
Using the three boundary means and then the centre and neighbour values,
we obtain one evaluation formula for every $\Phi$:
\begin{align}
 \M(\Phi(g,t_1))
 &=u+2nv+x\frac{(2n)^2}{2}
    +\frac{z-x}{3}\left((2n)^2-\varepsilon\frac{(2n)^2}{2}\right)
    +\left(y-\frac{x+2z}{3}\right)2n\notag\\
 &=u+2n(v+y)
   +\frac{(1+\varepsilon)(2n)^2-4n}{6}\,x
   +\frac{(2-\varepsilon)(2n)^2-8n}{6}\,z.
 \label{r2u:distribution}
\end{align}
This is an identity of means; it asserts no pointwise restriction beyond
the already proved local facts.

For an integer $z$, let $[z]_5\in\{0,1,2,3,4\}$ be its least
nonnegative residue modulo five. Suppose $5\nmid n$, and put
$a=[n]_5\in\{1,2,3,4\}$.
The modulo-five relation gives
\[
 k+t_5\equiv\Theta_a(g,t_1)\pmod5,\qquad
 \Theta_a(g,t)=[2(2a+1)^3-4ag+(4a^2+2a)t]_5.
\]
All residues required by \eqref{r2u:distribution} are
\[
\begin{array}{c|cc|ccc|ccc}
 a&\Theta_a(1,1)&\Theta_a(2,1)
  &\Theta_a(0,0)&\Theta_a(2,0)&\Theta_a(3,0)
  &\Theta_a(1,0)&\Theta_a(3,0)&\Theta_a(4,0)\\\hline
 1&1&2&4&1&2&0&2&3\\
 2&2&4&0&4&1&2&1&3\\
 3&1&4&1&2&0&4&0&3\\
 4&4&3&3&1&0&2&0&4
\end{array}
\]
The middle three columns are the boundary inputs for $\varepsilon=0$,
and the final three for $\varepsilon=1$.
Set
\[
 \eta=\varepsilon+(1-2\varepsilon)\mathbf1_{\{a=1\}}\in\{0,1\},
 \qquad
 \Phi(g,t)=\eta\Theta_a(g,t)-(1-\eta)w_5(\Theta_a(g,t)).
\]
Lemma~\ref{lem:rank5transfer}, with $p=5,H=k,T=t_5$, and
$\M\binom k2\le (2n)^2+n$ give
\[
 \M(\Theta_a)\le(2n+1)^2+(2n+1)=(2n)^2+6n+2,
\]
\[
 \M(w_5(\Theta_a))
 \ge2\{(2n+1)^2-(2n+1)\}-\left((2n)^2+n\right)
 =(2n)^2+\frac{6n}{2}.
\]
Multiplying the first inequality by $\eta$ and the negative of the
second by $1-\eta$, then adding, proves
\begin{equation}\label{r2u:tests}
 \mathcal G_5:=\M(\Phi)+(1-2\eta)(2n)^2+\frac{3-9\eta}{2}2n-2\eta\le0.
\end{equation}
No indicator hypothesis on $t_5$ has been used.

For the five values $u,v,x,y,z$ of this $\Phi$, define
\[
 \begin{aligned}
 A&=2\{(1+\varepsilon)x+(2-\varepsilon)z\}+12(1-2\eta),\\
 B&=10A+6(v+y)-2x-4z+9-27\eta,\\
 C&=5B-25A+3u-6\eta.
 \end{aligned}
\]
Substitution of \eqref{r2u:distribution} in \eqref{r2u:tests}
expands every discrepancy by the same calculation:
\begin{align*}
 3\mathcal G_5
 &=A n^2+\{6(v+y)-2x-4z+9-27\eta\}n+3u-6\eta\\
 &=A n^2+(B-10A)n+C-5B+25A\\
 &=A(n-5)^2+B(n-5)+C.
\end{align*}
Applying the displayed residue table and the definition of $\Phi$ to
these three coefficient formulas gives the complete positive coefficient table:
\[
\begin{array}{c|cccc}
 &a=1&a=2&a=3&a=4\\\hline
 \varepsilon=0&(4,24,17)&(4,33,56)&(8,63,109)&(6,45,69)\\
 \varepsilon=1&(6,45,69)&(2,16,30)&(10,86,177)&(4,20,6)
\end{array}
\qquad\text{each entry is }(A,B,C).
\]
Consequently, for $n\ge5$,
\[
 3\mathcal G_5\ge2(n-5)^2+16(n-5)+6>0,
\]
contrary to \eqref{r2u:tests}. Among $3\le n<5$, the condition
$3\nmid n$ leaves only $n=4$. Its parameters are
$(\varepsilon,a)=(0,4)$, and the same formula gives
\[
 3\mathcal G_5=6(4-5)^2+45(4-5)+69=6-45+69=30>0.
\]
Thus every dimension with $3\nmid n$ and $5\nmid n$ is excluded.

It remains that $5\mid n$ and $3\nmid n$. Here valid dilation five
makes $t_5$ an indicator, and $k+t_5\equiv2\pmod{5}$.
We have $k\le {(2n+1)}$ since it is a sum of ${(2n+1)}$ indicators. We first prove
that equality is impossible pointwise. If $k(\mathbf x)={(2n+1)}$, all
$t_1(\mathbf x-4\mathbf v)$, $\mathbf v\in\mathcal U_0$, are one, including $t_1(\mathbf x)$.
The quadratic translation identity at dilation two, multiplied by ${\mathsf S_1}f$, is \(\mathsf S_2g+k=2{(2n+1)}+{(2n)}t_1.\) 
It follows that
\[
 g(\mathbf x)+\sum_{\mathbf v\in \mathcal U}g(\mathbf x-2\mathbf v)=2{(2n+1)}+{(2n)}-{(2n+1)}=2{(2n)}+1.
\]
The first term is at least one. Each other term is at least two,
because it includes the distinct summands $t_1(\mathbf x)$ and $t_1(\mathbf x-4\mathbf v)$.
Their lower bounds already sum to $2{(2n)}+1$, so
$g(\mathbf x)=1,\qquad g(\mathbf x-2\mathbf v)=2\quad(\mathbf v\in \mathcal U).$
Since $t_1(\mathbf x)=1$, \eqref{r2u:local} implies $\mathbf x\in {\mathcal C}$.
For $\mathbf y=\mathbf x-2\mathbf v$, packing gives $t_1(\mathbf y)=0$: a centre within one of $\mathbf y$
would be within three of $\mathbf x$, while $\mathbf x$ itself is not within one of $\mathbf y$.
The pointwise mod-three relation now gives
$t_3(\mathbf y)\equiv-2n\pmod3$. Since $3\nmid 2n$ and
$t_3(\mathbf y)\in\{0,1\}$, we must have $t_3(\mathbf y)=1$.
Thus $\mathbf y=\mathbf c+3\mathbf w$ for a centre $\mathbf c$
and $\mathbf w\in\mathcal U$; the zero mark is excluded by packing.
Now $\mathbf y-2\mathbf w=\mathbf c+\mathbf w$, so
$t_1(\mathbf y-2\mathbf w)=1$. The two known unit summands exhaust
$g(\mathbf y)=2$, forcing $\mathbf w=\pm\mathbf v$.
The two alternatives give, respectively,
\[
 \mathbf c=\mathbf x-2\mathbf v-3\mathbf v=\mathbf x-5\mathbf v,
 \qquad
 \mathbf c=\mathbf x-2\mathbf v+3\mathbf v=\mathbf x+\mathbf v,
\]
contradicting dilation-five or ordinary packing. Thus $k\ne 2n+1$.

Now $k\le2n$ and $k\equiv1,2\pmod5$. Since $5\mid n$, the
three values $2n,2n-1,2n-2$ have residues $0,4,3$, respectively,
so $k\le2n-3$.
Lemma~\ref{unified:endpoint-product} gives
\[
 0\le(k-1)(k-2)\le(2n-4)(k+t_5-2).
\]
Moreover,
\[
 \M(k+t_5-2)=\M(k)+\M(t_5)-2\M(1)
 =(2n+1)^2+(2n+1)-2V_2=2n.
\]
Consequently,
\[
\begin{aligned}
 \frac{(2n)^2}{2}-2n
 &=(2n)^2-(2n+1)^2+V_2\\
 &\le\M\binom{k}{2}-\M(k)+\M(1)\\
 &=\frac12\M\bigl((k-1)(k-2)\bigr)\\
 &\le\frac{2n-4}{2}\M(k+t_5-2)
   =\frac{(2n)^2}{2}-4n.
\end{aligned}
\]
Since ${(2n)}>0$, this is impossible. Together with the common mod-three
stage, it excludes every radius-two dimension $n\ge3$.

\end{proof}

\subsection{The radius-three and radius-four exclusions}
The same square calculation treats radii three and four. At radius four,
Lemma~\ref{foundation:v5:triangular} already excludes $3\le n\le12$.
\begin{lemma}\label{unified:r3-odd}\label{unified:r4-chain}
There is no radius-three tiling in an odd dimension $n\ge5$.
There is no radius-four tiling when $n\ge13$ and
$2n\bmod5\in\{1,2,3\}$.
\end{lemma}
\begin{proof}
Write $2n=(r+1)q+c$ and put $t=r+1\in\{4,5\}$.
At radius three, $n$ odd and $n\ge5$ give $(t,c)=(4,2)$ and
$q=(n-1)/2\ge2$. At radius four, $n\ge13$ and $1\le c\le3$
give $5q=2n-c\ge23$, hence $q\ge5$.

For $t=4,5$ define
\[
 \begin{gathered}
 A_4=h_4-m_4=m_{31}+m_{22}+m_{211}+m_{1111},\qquad A_5=h_5,\\
 (\kappa_4,\alpha_4)=(192,-36),\qquad
 (\kappa_5,\alpha_5)=(150,24).
 \end{gathered}
\]
Both $A_t$ have zero-or-one monomial coefficients, and stable unit
deletion gives $DA_4=h_3$ and $DA_5=h_4$.

Choose degree-$(t+1)$ multipliers $W_t$ by the complete tables
\[
\begin{array}{c|rrrrrrr}
\lambda&5&41&32&311&221&2111&11111\\\hline
[m_\lambda]W_4&48&39&136&102&116&114&120
\end{array}
\]
\[
\begin{array}{c|rrrrrrrrrrr}
\lambda&6&51&42&411&33&321&3111&222&2211&21111&111111\\\hline
[m_\lambda]W_5&60&91&85&73&132&89&81&75&86&87&90
\end{array}
\]
and put
\[
 Q_t=W_th_{t-1}-\frac{\kappa_t}{2}A_t^2+\alpha_t\Pi_{2t,2}.
\]
All unlisted coefficients in the multiplier tables are zero.
For each $\lambda\vdash2t$, put $N_j=|\{i:\lambda_i\ge j\}|$.
Coordinatewise subdivision gives
\[
 [m_\lambda](m_\nu h_{2t-|\nu|})
 =\frac{\prod_{i=1}^{\ell(\nu)}(N_{\nu_i}-i+1)}
        {\prod_a m_a(\nu)!}.
\]
The product selects distinct coordinates for the descending parts of
$\nu$ and divides by permutations of equal parts. Together with
\eqref{small:coefficient-rule}, these formulas
give the complete expansions, modulo all-even profile monomials,
\begin{align*}
 Q_4\equiv{}&51m_{71}+131m_{53}+60m_{521}+9m_{51^3}
 +58m_{431}+23m_{421^2}\\
 &+100m_{3^22}+44m_{32^21}+16m_{321^3}+8m_{2^21^4},\\
 Q_5\equiv{}&3h_2m_{521}+p_3(13m_{43}+7m_{31^4})
       +m_{21}(4m_{61}+5m_{52}+2m_{421}+2m_{31^4})\\
 &+h_3(4m_{52}+11m_{51^2}+2m_{32^2})
       +m_{31}(25p_6+10m_{51}+13m_{42}+12m_{33})\\
 &+m_{41}(18m_{32}+11m_{221})
       +m_{32}(7m_{32}+24m_{311}+5h_5)+8m_{311}m_{221}.
\end{align*}
Both displayed right sides are monomial-nonnegative, so $Q_t$ satisfies the
nonnegativity hypothesis of Lemma~\ref{common:contact-functional}.

Apply Corollary~\ref{common:square-target-certificate} with
$(\kappa,b,a,B,\beta,W,\alpha)=(\kappa_t,0,q+1,0,0,W_t,\alpha_t)$.
Since $DA_t=h_r$, one has $\Lam(A_t)=\calL(A_t)+\sigma_r$, so
the scalar inequality simplifies to
\begin{align*}
 0\le{}&\Lam(W_t)-\kappa_t(q+1)\Lam(A_t)
      +\frac{\kappa_t(q+1)(q+2)}2\sigma_r
      +\frac{\alpha_t(2n)^2}2\\
 ={}&\Lam(W_t)-\kappa_t(q+1)\calL(A_t)
      +\frac{\kappa_tq(q+1)}2\sigma_r
      +\frac{\alpha_t(2n)^2}2=:\Delta_t.
\end{align*}

For the scalar evaluation, define $U_j(2n)=2^j\binom nj$ for $j\ge0$
and write $U_j=U_j(2n)$, so $j!U_j=\prod_{i=0}^{j-1}(2n-2i)$.
The orbit formula and unit deletion give
\[
 \calL(m_\nu)=\frac{\ell(\nu)!}{\prod_{a\ge1}m_a(\nu)!}U_{\ell(\nu)},
 \qquad
 \Lam(m_\nu)=\calL(m_\nu)+\calL(m_{\nu\setminus(1)}),
\]
where the second summand is zero if $\nu$ has no unit part.
Applying these formulas to the multiplier tables yields
\begin{align*}
 \Lam(W_4)&=87(2n)+670U_2+996U_3+576U_4+120U_5,\\
 \Lam(W_5)&=151(2n)+808U_2+1329U_3+1188U_4+525U_5+90U_6,\\
 \calL(A_4)&=3U_2+3U_3+U_4,\qquad
 \sigma_3=(2n)+2U_2+U_3,\\
 \calL(A_5)&=(2n)+4U_2+6U_3+4U_4+U_5,\qquad
 \sigma_4=(2n)+3U_2+3U_3+U_4.
\end{align*}
Expanding the displayed falling products in the formula for $\Delta_t$
gives
\begin{align*}
 \frac{\Delta_4}{(2n)}={}&(2n)^4-4(2n)^3+18(2n)^2-87(2n)+169
 +(16(2n)^2+32)q^2\\
 &+(-8(2n)^3+16(2n)^2-64(2n)+224)q,\\
 \frac{8\Delta_5}{2n}={}&(2n)^5-5(2n)^4+36(2n)^3-80(2n)^2+504(2n)-728\\
 &+(25(2n)^3+200(2n))q^2\\
 &+(-10(2n)^4+25(2n)^3-200(2n)^2+200(2n)-240)q.
\end{align*}
Substituting $(2n)=tq+c$ gives $\Delta_t=-(2n) R_{t,c}(q)$, where
\[
\begin{array}{c|c|l}
(t,c)&q_0&R_{t,c}(q)\\\hline
(4,2)&2&28q-51\\
(5,1)&5&50q^2-245q+34\\
(5,2)&5&(125q^3+50q^2-1040q-100)/4\\
(5,3)&5&(125q^3-175q^2-1365q-437)/4
\end{array}
\]
For $q=q_0+x$, the complete translated expansions are
\begin{align*}
 R_{4,2}(2+x)&=28x+5,\\
 4R_{5,1}(5+x)&=200x^2+1020x+236,\\
 4R_{5,2}(5+x)&=125x^3+1925x^2+8835x+11575,\\
 4R_{5,3}(5+x)&=125x^3+1700x^2+6260x+3988.
\end{align*}
Every coefficient in these four expansions is positive. Thus $q\ge q_0$
implies $x\ge0$, $-\Delta_t/(2n)>0$, and $\Delta_t<0$ since $(2n)>0$,
contradicting the scalar inequality above in every asserted class.
\end{proof}

\begin{lemma}\label{unified:r3-complete}
For every $n\ge3$, the Lee ball $B_n(3)$ does not tile $\mathbb Z^n$
by translations with an arbitrary centre set.
\end{lemma}
\begin{proof}
Lemma~\ref{foundation:v5:small-dim} excludes $n=3,4$, and
Lemma~\ref{unified:r3-odd} excludes odd $n\ge5$.
For the remaining even $n\ge6$, write $2n=4q$, $q\ge3$.
Abbreviate $Y_\nu=Y_{m_\nu}$, so the five completed fields are
\[
 Y_4=F_4,\quad Y_{31}=F_{31}+I_3,\quad Y_{22}=F_{22},\quad
 Y_{211}=F_{211}+I_{21},\quad Y_{1111}=F_{1111}+I_{111}.
\]
The signed-support identity, doubled packing and its parity reduction give,
on the boundary,
\begin{gather*}
 Y_4+2Y_{31}+2Y_{22}+3Y_{211}+4Y_{1111}=4q+1,\\
 Y_4,Y_{22}\in\{0,1\},\qquad Y_4+Y_{22}\le1,
 \qquad Y_4+Y_{211}\equiv1\pmod2.
\end{gather*}
Indeed the original signed-support equation is
\[
 I_3+2I_{21}+3I_{111}+F_4+2F_{31}+2F_{22}+3F_{211}+4F_{1111}=4q.
\]
Completion increases its left side by $I_\partial=1$, giving the first
identity. Reducing modulo two and adding $I_{21}$ gives
$Y_4+Y_{211}\equiv I_3+I_{21}+I_{111}=1\pmod2$.
All five completed fields vanish off the boundary.

Consider the single quadratic
\begin{align*}
 P={}&2Y_{31}^2+2Y_{211}^2+7Y_{31}Y_{211}
       +16Y_{31}Y_{1111}+8Y_{211}Y_{1111}\\
 &+3Y_{31}-2Y_{211}-24Y_{1111}-12I_\partial\\
 &+Y_4(12I_\partial-13Y_{31}-3Y_{211}+24Y_{1111})\\
 &+Y_{22}(22I_\partial-18Y_{31}-10Y_{211}+16Y_{1111}).
\end{align*}
We prove its nonnegativity from the three possible binary pairs
$(Y_4,Y_{22})$. On the boundary put
\[
 k=(Y_{211}+Y_4-1)/2,\qquad z=Y_{31}+k.
\]
These are nonnegative integers, $k\le z$, and the support identity is
$z+2k+2Y_{1111}=2q-1+Y_4-Y_{22}$.
If $(Y_4,Y_{22})=(1,0)$, then $q=z/2+k+Y_{1111}$ and direct substitution gives
\begin{align*}
 P/2&=z(z-5)+k(5z-2k)+8Y_{1111}z\\
 &\ge z(z-5)+3z(k+Y_{1111})
 =z(3q-5-z/2)\ge z(2q-5)\ge0,
\end{align*}
since $0\le k\le z\le2q$ and $q\ge3$.
If $(Y_4,Y_{22})=(0,0)$, the support identity makes $z$ odd and at least one;
substitution gives the exact nonnegative factorization
\[
 P/2=(z-1)(z+6+3k+8Y_{1111})+2k(z-k)\ge0.
\]
If $(Y_4,Y_{22})=(0,1)$, then $z$ is even and
\[
 P/2=z(z-4)+k(5z-2k-4)+8Y_{1111}z.
\]
At $z=0$ this is zero. At $z\ge4$ every term is nonnegative because
$5z-2k-4\ge3z-4>0$. The only remaining value is $z=2$, when
\[
 P/2=-4+2k(3-k)+16Y_{1111},\qquad 0\le k\le2.
\]
If $Y_{1111}\ge1$ this is at least $12$. If $Y_{1111}=0$, the support identity gives
$k=q-2\ge1$, and the expression is
$-2(k-1)(k-2)\ge0$. Thus $P\ge0$ on the boundary and is zero off it.

The quadratic part of $P$ corresponds to the homogeneous polynomial
\begin{align*}
 R={}&2m_{31}^2+2m_{211}^2+7m_{31}m_{211}
          +16m_{31}m_{1111}+8m_{211}m_{1111}\\
 &-13m_4m_{31}-3m_4m_{211}+24m_4m_{1111}
          -18m_{22}m_{31}-10m_{22}m_{211}+16m_{22}m_{1111}.
\end{align*}
The common bilinear identity~\eqref{common:profile-bilinear} gives
\[
 \M P=C+\mathfrak C(R),
 \qquad
 C=12\Lam(m_4)+5\Lam(m_{31})+22\Lam(m_{22})
          -24\Lam(m_{1111})-12\sigma_3.
\]
Here the only quadratic diagonals are $2\Lam(m_{31})$ and
$2\Lam(m_{211})$; the latter cancels the linear term $-2Y_{211}$.
Choose the single multiplier
\[
 W=m_5+4m_{41}-7m_{32}+10m_{311}-4m_{221}+6m_{2111},
 \qquad Q=Wh_3-R-18\Pi_{8,2}.
\]
Coordinatewise subdivision gives the full identity
\begin{align*}
 Q={}&4(m_{53}+m_{521}+m_{5111}+m_{431}+m_{3311})\\
 &+m_8-4m_{62}-10m_{44}-24m_{422}-48m_{2222}.
\end{align*}
The general coefficient rule~\eqref{small:coefficient-rule} gives this
identity term by term. The last five terms are all-even:
they have zero $\bar Z$ pairing and stable derivative zero.
All other coefficients are nonnegative, so $Q$ satisfies the stable
coefficient hypothesis of Lemma~\ref{common:contact-functional}.

For the scalar evaluation use $U_j=U_j(2n)$ as above.
The orbit formula and the preceding expression for $\sigma_3$ give
\begin{gather*}
 \Lam(m_4)=2n,\quad \Lam(m_{31})=2n+2U_2,\quad
 \Lam(m_{22})=U_2,\quad \Lam(m_{1111})=U_3+U_4,\\
 C=5(2n)+8U_2-36U_3-24U_4,\\
 \calL(W)=2n-6U_2+18U_3+24U_4,\qquad
 \calL(DW)=4(2n)+16U_2+18U_3,\\
 \Lam(W)=5(2n)+10U_2+36U_3+24U_4.
\end{gather*}
Since $P\ge0$ and $\mathfrak C(Q)\ge0$, the same functional gives
\begin{align*}
 0&\le\M P+\mathfrak C(Q)=C+\mathfrak C(R+Q)\\
  &=C+\mathfrak C(Wh_3-18\Pi_{8,2})
   =C+\Lam(W)-9(2n)^2\\
  &=10(2n)+18U_2-9(2n)^2=-8(2n)<0,
\end{align*}
where $2U_2=(2n)(2n-2)$. This contradiction excludes every even $n\ge6$.
\end{proof}

\begin{lemma}\label{small:exclusions}
For every $n\ge3$, the Lee ball $B_n(4)$ does not tile $\mathbb Z^n$
by translations with an arbitrary centre set.
\end{lemma}
\begin{proof}
Lemma~\ref{foundation:v5:triangular} excludes $3\le n\le12$,
and Lemma~\ref{unified:r4-chain} excludes
$2n\equiv1,2,3\pmod5$. Thus the remaining classes are
\[
 2n=5q,\quad q\ge6,\qquad\hbox{and}\qquad
 2n=5q+4,\quad q\ge6.
\]
Indeed, their least even values above $24$ are $30$ and $34$.
Set
\[
 A_0=m_{311}+m_{2111}+m_{11111},\qquad
 A_1=m_{221}+m_{2111}+m_{11111},\qquad
 B_0=m_{11111},\qquad B_1=m_{2111}.
\]
Apply the square-and-target criterion with $A=A_j$, $B=B_j$ and the
following parameters; both $A_j$ are zero-or-one profile sums:
\[
\begin{array}{c|c|c|c|c|c|c}
j&2n&\kappa&b&a&\beta&\alpha\\\hline
0&5q&10&20&q&q&8\\
1&5q+4&50&25&q+1&q+1&197
\end{array}
\]
The sole supported profile of $B_0$ has length five, and that of
$B_1$ has length four. The target conditions are therefore exactly
\[
 \left\lceil\frac{5q+1-5}{5}\right\rceil=q,\qquad
 \left\lceil\frac{5q+4+1-4}{5}\right\rceil=q+1.
\]
Thus every nonnegativity hypothesis for the local gap follows from
Corollary~\ref{common:square-target-certificate}.

Choose degree-six multipliers from the complete table
\[
\begin{array}{c|rrrrrrrrrrr}
\lambda&6&51&42&411&33&321&3111&222&2211&21111&111111\\\hline
{[m_\lambda]W_0}&0&0&0&3&3&-4&11&16&-2&13&30\\
{[m_\lambda]W_1}&33&1&-20&1&113&2&2&-75&52&49&30
\end{array}
\]
and put
\begin{align*}
 Q_0&=W_0h_4-5A_0^2-20B_0h_5+8\Pi_{10,2},\\
 Q_1&=W_1h_4-25A_1^2-25B_1h_5+197\Pi_{10,2}.
\end{align*}
The following full identities give nonnegativity without a contact-row table:
\begin{align*}
 Q_0={}&p_1\left(\frac32p_1p_8+p_4m_{41}+2p_5h_4
                  +\frac32m_{522}+\frac72m_{3222}\right)\\
 &+p_2\left(\frac52p_3m_{32}+9p_4e_4+3m_{41^4}\right)\\
 &+p_3\left(3p_3e_4+\frac72m_{43}+\frac52m_{2221}
                   +3m_{221^3}+5m_{21^5}\right)\\
 &+m_{21}m_{52}+e_3m_{43}+e_4m_{33}+5e_5m_{32}
                 +4m_{3331}+2m_{331^4}+m_{2221^4}\\
 &-\frac72m_{10}-9m_{82}-9m_{64}-\frac{17}2m_{622}
                  -3m_{442}+7m_{4222}+40m_{22222},\\[2mm]
 Q_1={}&\frac{53}2m_{541}+60m_{51^5}+17m_{4411}+141m_{433}
                   +96m_{3331}\\
 &+m_{331^4}+5m_{32221}+29m_{222211}
                   +\frac{217}2p_1m_{81}+43p_2m_{3311}\\
 &+h_2\left(\frac{209}2m_{62}+7m_{521}\right)
            +m_{421}(2p_3+5m_{21})+12p_3m_{221^3}\\
 &+25e_3m_{52}+\frac{39}2p_4m_{51}+87p_6m_{31}
            +\frac12m_{42}(171m_{31}+141m_{211})\\
 &+h_3(16p_7+31m_{43})\\
 &+17m_{10}-215m_{82}-\frac{119}2m_{64}-308m_{622}
              -133m_{442}-336m_{4222}-720m_{22222}.
\end{align*}
The coefficient rule~\eqref{small:coefficient-rule}, applied successively
to the displayed products, proves both identities including their complete
all-even remainders. The positive-product portions are
monomial-nonnegative, while the remainders have zero $\bar Z$ pairing
and stable derivative zero. Thus the coefficient hypotheses hold.
Corollary~\ref{common:square-target-certificate} with the two parameter
rows gives
\begin{align*}
 0\le\Delta_0:={}&\Lam(W_0)-10q\Lam(A_0)+5q(q+1)\sigma_4
                  -20q\Lam(B_0)+4(2n)^2,\\
 0\le\Delta_1:={}&\Lam(W_1)-50(q+1)\Lam(A_1)
                   +25(q+1)(q+2)\sigma_4\\
 &\hspace{14mm}-25(q+1)\Lam(B_1)+\frac{197}{2}(2n)^2.
\end{align*}

Use $U_j=U_j(2n)$ and the expression for $\sigma_4$ computed above.
The signed-orbit formula and the multiplier table give
\begin{align*}
 \Lam(W_0)&=U_2+28U_3+84U_4+95U_5+30U_6,\\
 \Lam(W_1)&=34(2n)+81U_2+102U_3+516U_4+275U_5+30U_6,\\
 \Lam(A_0)&=2U_2+6U_3+5U_4+U_5,\\
 \Lam(A_1)&=U_2+6U_3+5U_4+U_5,\\
 \Lam(B_0)&=U_4+U_5,\qquad
 \Lam(B_1)=3U_3+4U_4.
\end{align*}
Substitute these expressions into $\Delta_j$ and expand the displayed
falling products. Using $2n=5q$ in the first class and $2n=5q+4$ in
the second yields
\[
 \frac{\Delta_0}{2n}=-\frac{75q^3-475q^2+201q-74}{6},\qquad
 \frac{\Delta_1}{2n}=-175q^2+940q+659.
\]
Their signs hold uniformly, rather than by a finite check: with
$q=6+x$, $x\ge0$, the complete translated expansions are
\[
 -\frac{\Delta_0}{2n}
   =\frac{116}{3}+\frac{867}{2}x+\frac{875}{6}x^2+\frac{25}{2}x^3>0,
 \qquad
 -\frac{\Delta_1}{2n}=1+1160x+175x^2>0.
\]
Thus $\Delta_0,\Delta_1<0$, contradicting the respective nonnegative
sums and excluding both remaining classes.
\end{proof}

\subsection{Completion of the geometric reduction}\label{sec:parameter-reduction}
The preceding lemmas exclude radii two, three and four. For the
remaining radii, the geometric and remainder bounds give one reduction.
\begin{lemma}\label{foundation:thm:main-nonresidue}

If a tiling exists with $r\ge5$ and $n\ge3$, then
\[
 2n\ge(r+1)(r+2),\qquad 2n\equiv0,-1\pmod{r+1}.
\]
\end{lemma}
\begin{proof}
Lemma~\ref{foundation:v5:triangular} gives
$n>r(r+1)/2+2$. Since $2n$ and $r(r+1)$ are even integers,
$2n\ge r(r+1)+6$. Lemma~\ref{boundary:prop:all-nonexceptional-high}
therefore excludes every remainder $1,\ldots,r-1$ modulo $r+1$.
Suppose now that $2n<(r+1)(r+2)$. The lower bound $2n\ge r(r+1)+6$ rules out all multiples of
$r+1$ at most $r(r+1)$ and all integers one below them. Thus any
remaining value in this interval must be one of
\[
 (r+1)^2-1,\qquad (r+1)^2,\qquad (r+1)(r+2)-1.
\]
The last number is odd. If $r$ is odd, only $(r+1)^2$ among the
first two is even; if $r$ is even, only $(r+1)^2-1$ is even.
Both possibilities contradict
Lemma~\ref{foundation:thm:v5-diagonal-exclusions}.
Thus $2n\ge(r+1)(r+2)$, as required.
\end{proof}

\section{Algebraic exclusion of the remaining residue classes}
\label{sec:analytic}
The geometric reduction leaves $r\ge5$,
$2n\ge(r+1)(r+2)$, and $2n\equiv0,-1\pmod{r+1}$.
We first turn the contact identities into one coefficient criterion.
The stable construction then treats the uniform ranges; exact finite
identities handle the five remaining rank--remainder pairs.

\subsection{A paired contact criterion}
The contact functional also gives the monomial-target criterion
used in the remaining arguments.
Let $V$ be homogeneous of degree $r+2$. For $T=m_\mu$,
$\mu\vdash r+1$, set
\[
 \delta=r+1-\ell(\mu),\qquad
 \beta=\left\lceil\frac{2n-r+\delta}{r+1}\right\rceil,
 \qquad Q=Vh_r-Th_{r+1}.
\]
If every stable non-all-even coefficient of $Q$ is nonnegative, then
\begin{equation}\label{unified:local-master-identity}
 \Lam(V)-\beta\Lam(T)\ge0.
\end{equation}
The sole supported profile of $T$ is $\mu$, and
$2n+1-\ell(\mu)=2n-r+\delta$. Lemma~\ref{common:contact-functional}
therefore gives
\[
 \begin{aligned}
 0&\le\M\bigl(Y_T(F-\beta)\bigr)+\mathfrak C(Q)\\
  &=\mathfrak C(Th_{r+1})-\beta\Lam(T)
       +\mathfrak C(Vh_r-Th_{r+1})
   =\Lam(V)-\beta\Lam(T).
 \end{aligned}
\]

\subsection{Stable coefficients and comparison tools}
For the stable construction, assume $r\ge8$ for the zero target and
$r\ge10$ for the minus target. The smaller radii are treated below.
For a list $\tau=(\tau_1\ge\cdots\ge\tau_k\ge2)$ set
$d_\tau=\sum_i(\tau_i-1)$ and $\operatorname{aut}(\tau)=\prod_{b\ge2}m_b(\tau)!$.
For a profile $\lambda$ of length $L$, let $N_b(\lambda)=|\{i:\lambda_i\ge b\}|$.
We abbreviate $N_b(\lambda)$ to $N_b$ when the profile is fixed. Define
\begin{equation*}
 \mathsf N_\tau(\lambda)=\prod_{i=1}^k(N_{\tau_i}-i+1),\qquad
 \widetilde m_{\ell,\tau}=\frac{\operatorname{aut}(\tau)}{(\ell)_k}m_{(\tau,1^{\ell-k})}.
\end{equation*}
The definition of $\widetilde m_{\ell,\tau}$ requires $k\le\ell$.
An impossible injection contributes zero; $\mathsf N_\varnothing=1$
and $\widetilde m_{\ell,\varnothing}=e_\ell$.
Thus $\mathsf N_\tau$ counts ordered injective choices of coordinates
satisfying the listed thresholds. If $k>L$, no such choice exists
and the corresponding monomial coefficient is zero.
For $k\le L$ and matching degrees,
\[
 [m_\lambda](\widetilde m_{\ell,\tau}h_{|\lambda|-\ell-d_\tau})
 =\frac{\operatorname{aut}(\tau)}{(\ell)_k}\frac{\mathsf N_\tau(\lambda)}{\operatorname{aut}(\tau)}
   \binom{L-k}{\ell-k}
 =\binom {L}{\ell}\frac{\mathsf N_\tau(\lambda)}{(L)_k}.
\]
Indeed, choose the nonunit coordinates, divide by the permutations of equal
markers, and then choose the remaining unit coordinates. The orbit formula gives
\[
 \calL(\widetilde m_{\ell,\tau})
 =\frac{\operatorname{aut}(\tau)}{(\ell)_k}
   \frac{2^\ell(n)_\ell}{\operatorname{aut}(\tau)(\ell-k)!}
 =U_\ell(2n).
\]
For $\ell>k$, deleting one unit part gives
$D\widetilde m_{\ell,\tau}
 =((\ell-1)_k/(\ell)_k)\widetilde m_{\ell-1,\tau}
 =((\ell-k)/\ell)\widetilde m_{\ell-1,\tau}$.
Applying $\calL$ yields
$\calL(D\widetilde m_{\ell,\tau})=(\ell-k)U_{\ell-1}(2n)/\ell$.
For $\ell=k$, there is no unit part and the derivative is zero.

Let $T_0=e_{r+1}$ and $T_1=m_{(2,1^{r-1})}$.
For $i\in\{0,1\}$, write a degree-$(r+2)$ multiplier as
$\sum_\tau v_{i,\tau}\widetilde m_{{(r+1)}+1-d_\tau,\tau}$.
At defect $d$, put $L=2{(r+1)}-d$ and evaluate on the profile consisting of $\tau$
and $L-k$ units. Its $k$ nonunit positions give
\[
 [m_{(\tau,1^{L-k})}](T_i h_{r+1})
 =\begin{cases}
 \displaystyle\binom L{r+1},&i=0,\\[2pt]
 \displaystyle k\binom{L-1}{r-1},&i=1
 \end{cases}
 =(k)_i\binom{L-i}{r+1-2i}.
\]
A fitting marker of defect $d$ must use all nonunit positions and match
each threshold exactly: an unused position or a strict inequality would
make its defect smaller than $d$. Thus the only same-defect marker that
fits is $\tau$ itself, and
$\mathsf N_\tau(\tau)=\operatorname{aut}(\tau)$. Solving for its coefficient gives
the triangular rule
\[
 v_{i,\tau}=\frac{(L)_k}{\operatorname{aut}(\tau)\binom L{{(r+1)}+1-d}}
 \left((k)_i\binom{L-i}{r+1-2i}-
 \sum_{d_\sigma<d}v_{i,\sigma}\binom {L}{{(r+1)}+1-d_\sigma}
       \frac{\mathsf N_\sigma(\tau)}{(L)_{\ell(\sigma)}}\right).
\]
All denominators here are positive for the listed defects. For example,
the first two zero-target entries are
\[
 v_{0,\varnothing}
 =\frac{\binom{2r+2}{r+1}}{\binom{2r+2}{r+2}}
 =\frac{r+2}{r+1},\qquad
 v_{0,2}=(2r+1)\left(1-\frac{r+2}{r+1}\frac r{r+2}\right)
 =\frac{2r+1}{r+1}.
\]
The remaining entries follow successively from the same rule, which
also verifies matching. The listed markers exhaust the partitions of
each defect after subtracting one from every nonunit part.

For the target $T_0$, define
\begin{equation*}
 W_4=\sum_{d_\tau\le4}v_{0,\tau} \widetilde m_{{(r+1)}+1-d_\tau,\tau},
\end{equation*}
with coefficients
\[
\begin{array}{c|c@{\qquad}c|c}
\tau&v_{0,\tau}&\tau&v_{0,\tau}\\\hline
\varnothing&(r+2)/{(r+1)}&2&(2{(r+1)}-1)/{(r+1)}\\
3&(2{(r+1)}-3)(r)/{(r+1)}^2&22&-(2{(r+1)}-3)(r)/{(r+1)}^2
\end{array}
\]
For the next two layers put $b_3=(r+1)^2$ and
$b_4=(r+1)^2(r-1)$. The table gives $b_dv_{0,\tau}$:
\[
\begin{gathered}
\begin{array}{c|c|l}
d&\tau&b_dv_{0,\tau}\\\hline
3&4&2{(r+1)}^2-11{(r+1)}+13\\
 &32&-2(4{(r+1)}^2-18{(r+1)}+19)\\
 &222&(2{(r+1)}-5)(3{(r+1)}-5)\\\hline
4&5&2{(r+1)}^3-23{(r+1)}^2+74{(r+1)}-71\\
 &42&-(r-4)(2{(r+1)}-5)(6{(r+1)}-11)\\
 &33&-(3{(r+1)}-8)(2{(r+1)}^2-13{(r+1)}+17)\\
 &322&38{(r+1)}^3-313{(r+1)}^2+812{(r+1)}-669\\
 &2222&-(2{(r+1)}-7)(11{(r+1)}^2-46{(r+1)}+47)
\end{array}
\end{gathered}
\]
Consequently $T_0h_{r+1}-W_4h_{r}$ vanishes on every profile of defect at most four.

For the target $T_1$, define
\begin{equation*}
 V_5=\sum_{d_\tau\le5}v_{1,\tau} \widetilde m_{{(r+1)}+1-d_\tau,\tau},
\end{equation*}
with $v_{1,\varnothing}=0$ and
$v_{1,2}={(r+1)}-1$, $v_{1,3}=(r)/{(r+1)}$, $v_{1,22}=(2{(r+1)}-3)(r)/{(r+1)}$, and
$v_{1,\tau}=\varphi_d A_{d,\tau}$ in the remaining layers, where
\[
 \varphi_3=\frac{r}{r+1},\qquad \varphi_4=\frac{r}{{(r+1)}(r-1)},\qquad
 \varphi_5=\frac{r}{{(r+1)}(r-1)(r-2)},
\]
\[
\begin{array}{c|c|l}
d&\tau&A_{d,\tau}\\\hline
3&4&1\\ &32&2(r-2)\\ &222&-(2{(r+1)}-5)\\\hline
4&5&{(r+1)}-3\\ &42&2{(r+1)}^2-16{(r+1)}+27\\ &33&-(3{(r+1)}-8)\\
 &322&-(8{(r+1)}^2-53{(r+1)}+81)\\ &2222&(2{(r+1)}-7)(3{(r+1)}-7)\\\hline
5&6&{(r+1)}^2-8{(r+1)}+13\\ &52&2{(r+1)}^3-30{(r+1)}^2+123{(r+1)}-147\\
 &43&-(10{(r+1)}^2-67{(r+1)}+101)\\ &422&-(12{(r+1)}^3-145{(r+1)}^2+525{(r+1)}-582)\\
 &332&-(6{(r+1)}^3-91{(r+1)}^2+381{(r+1)}-466)\\
 &3222&2(19{(r+1)}^3-208{(r+1)}^2+711{(r+1)}-762)\\
 &22222&-(2{(r+1)}-9)(11{(r+1)}^2-60{(r+1)}+79)
\end{array}
\]
Thus $T_1h_{r+1}-V_5h_{r}$ vanishes through defect five. In each layer
$\sum_\tau A_{d,\tau}=0$, while
\[
 \sum_\tau\ell(\tau)A_{d,\tau}=
 \begin{cases}-2(r-1),&d=3,\\
 2(r-1)(2{(r+1)}-5),&d=4,\\
 -4(r-1)(r-2)(2{(r+1)}-5),&d=5.
 \end{cases}
\]
These identities follow by adding the displayed polynomials; they retain the
derivative charge even though each layer has zero ordinary mean.

Put $\eta_1={(r+1)}(r-10)/18$ and
$\rho={(r+1)}(r-10)^2/100$. Write
\begin{align*}
 C_0&=\sum_\tau y_{0,\tau}\widetilde m_{{(r+1)}-4,\tau},\\
 C_1&=\sum_\tau y_{1,\tau}\widetilde m_{{(r+1)}-5,\tau},
\end{align*}
with nonzero coefficients
\[
\begin{array}{c|c@{\qquad}c|c}
\tau&y_{0,\tau}&\tau&y_{1,\tau}\\\hline
6&{(r+1)}^2/6&7&3{(r+1)}^2\\
422&66(r-6)/{(r+1)}&62&34(r-10)+17{(r+1)}/2+2\rho\\
332&10(r-6)&44&4+\rho\\
22222&120(r-5)/{(r+1)}&432&48+(r+1)(r-10)^2/12\\
&&4222&132(r-7)+\rho\\
&&333&64+\rho\\
&&3322&132(r-8)\\
 &&222222&180(r-5)\\
 &&522&2+\eta_1\\
 &&53&\eta_1
\end{array}
\]
All coefficients are nonnegative in their respective rank ranges. The added profiles have defects five and
six respectively, so they do not disturb the preceding exact matching.
Lemma~\ref{alg:uniform} will combine three ingredients: nonnegative
coefficients for profiles containing a large part, comparisons on the
remaining bounded alphabets, and strictly negative scalar contributions.
We first establish the tools shared by those comparisons.

Sampling the nonunit coordinates gives the same coefficient formula
for the bounded-alphabet and finite-rank arguments.
\begin{lemma}\label{alg:sample}
Let $\lambda\vdash M$ have $L$ parts, of which $k\ge h\ge1$ exceed one.
Choose a uniform $h$-subset of the nonunit positions and call its
sublist $X$. Write $\nu=(\alpha,2^b,1^u)$, where all parts of
$\alpha$ exceed two, $p=\ell(\alpha)\le h$, and $|\nu|\le M$.
For $p+b\le k$,
\begin{equation}
 [m_\lambda](m_\nu h_{M-|\nu|})
   =\mathbb E\left[
       \frac{(k)_{p+b}\mathsf N_\alpha(X)}
            {(h)_p\operatorname{aut}(\alpha)b!}
             \binom{L-p-b}{u}\right].
       \label{common:sample-coefficient}
\end{equation}
When $p+b>k$, both sides are zero; set the sampled expression
to zero before evaluating its binomial factor. With $E=M-L-k$,
\begin{equation}\label{common:sample-excess}
 \mathbb E\sum_{x\in X}(x-2)=\frac{hE}{k}.
\end{equation}
\end{lemma}
\begin{proof}
Each ordered admissible $p$-tuple is sampled with probability
\[
 \frac{\binom{k-p}{h-p}}{\binom kh}
 =\frac{h!(k-p)!}{(h-p)!k!}=\frac{(h)_p}{(k)_p}.
\]
After fixing the $\alpha$-slots, the ordered two-slots have
$(k-p)_b$ choices. Thus
\begin{equation}\label{common:sample-flag}
 \begin{aligned}
 \mathbb E\mathsf N_\alpha(X)
 &=\frac{(h)_p}{(k)_p}\mathsf N_\alpha(\lambda),\\
 \mathsf N_{(\alpha,2^b)}(\lambda)
 &=(k-p)_b\mathsf N_\alpha(\lambda)
 =(k)_{p+b}\mathbb E\frac{\mathsf N_\alpha(X)}{(h)_p}.
 \end{aligned}
\end{equation}
Divide by $\operatorname{aut}(\alpha)b!$ to forget the order of equal
nonunit slots, then choose the $u$ unit-slots from the $L-p-b$
remaining positions. Each placement leaves one monomial of
$h_{M-|\nu|}$, proving \eqref{common:sample-coefficient}.
For $p+b>k$ no such placement exists. Finally, each nonunit position
is sampled with probability $h/k$, so
\[
 \mathbb E\sum_{x\in X}(x-2)
 =\frac hk\sum_{\lambda_j\ge2}(\lambda_j-2)
 =\frac hk\{M-(L-k)-2k\}=\frac{hE}{k}.
\]
\end{proof}

The coefficient comparisons use the same two positive bases. We record
both coefficient transforms and the common envelope rule once.
\begin{lemma}\label{common:bernstein-envelope}
Let $P_\sigma=\sum_\alpha p_{\sigma\alpha}b_\alpha$ be a finite family,
where every $b_\alpha$ is nonnegative on the domain in question.
If $0\le\theta_\sigma\le1$, with $\theta_\sigma=1$ for $\sigma\in I$, then
\begin{equation}\label{common:positive-envelope}
 \sum_\sigma\theta_\sigma P_\sigma\ge
 \sum_\alpha\left(\sum_{\sigma\in I}p_{\sigma\alpha}
       +\sum_{\sigma\notin I}\min(0,p_{\sigma\alpha})\right)b_\alpha.
\end{equation}
Also $P_\sigma\ge\sum_\alpha(\min_\nu p_{\nu\alpha})b_\alpha$.
In the ordinary monomial basis we denote the latter polynomial by
$\operatorname{Env}_{\mathbf z}\{P_\nu\}$, taking missing coefficients as zero.

For $P(\mathbf x)=\sum_\beta p_\beta\mathbf x^\beta$ of degree at most $d$
and $\zeta\ge0$, define, for $|\alpha|=d$ and $0\le h\le d$,
\[
 \operatorname{Rad}_{d,\zeta}(P)_{h,\alpha}
 =\sum_{\substack{\beta\le\alpha\\h\le|\beta|}}
 p_\beta\binom{|\beta|}{h}\zeta^{|\beta|-h}
       \binom{d-|\beta|}{\alpha-\beta}.
\]
Here the last factor is multinomial and $0^0=1$. Then
\[
 P((\zeta+t)\mathbf u)
 =\sum_{h=0}^d\sum_{|\alpha|=d}
       \operatorname{Rad}_{d,\zeta}(P)_{h,\alpha}t^h\mathbf u^\alpha
 \quad(t,\mathbf u\ge0,\ \textstyle\sum_j u_j=1).
\]
For $P(\mathbf x,z)=\sum_{\alpha,a}p_{\alpha a}\mathbf x^\alpha z^a$
and $d\ge\deg_zP$, define
\[
 \operatorname{Ber}_{d,z}(P)_{\alpha,b}
 =\sum_{a=0}^b\frac{\binom ba}{\binom da}p_{\alpha a}
 \quad(0\le b\le d).
\]
Then
\[
 P(\mathbf x,z)=\sum_{\alpha,b}
 \operatorname{Ber}_{d,z}(P)_{\alpha,b}\mathbf x^\alpha
       \binom db z^b(1-z)^{d-b}.
\]
These are nonnegative bases on their displayed domains, so the envelope
rule applies. In particular, on $\mathbf x\ge0$, $0\le z\le1$,
\[
 P\ge\sum_\alpha\min_{0\le b\le d}
             \operatorname{Ber}_{d,z}(P)_{\alpha,b}\mathbf x^\alpha.
\]
For a finite family, the minimum may also range over its members.
\end{lemma}
\begin{proof}
For $0\le\theta\le1$ and real $p$ we have
$\theta p\ge\min(0,p)$. Apply this coefficient by coefficient and
multiply by $b_\alpha\ge0$ to obtain \eqref{common:positive-envelope}.
The ordinary envelope follows from
$p_{\sigma\alpha}-\min_\nu p_{\nu\alpha}\ge0$.
For the radial formula, substitution and the binomial theorem give
\[
 \begin{aligned}
 P((\zeta+t)\mathbf u)
 &=\sum_\beta p_\beta\sum_{h=0}^{|\beta|}
       \binom{|\beta|}{h}\zeta^{|\beta|-h}t^h\mathbf u^\beta
         (\textstyle\sum_j u_j)^{d-|\beta|}\\
 &=\sum_\beta\sum_{h=0}^{|\beta|}\sum_{|\gamma|=d-|\beta|}
       p_\beta\binom{|\beta|}{h}\zeta^{|\beta|-h}
       \binom{d-|\beta|}{\gamma}t^h\mathbf u^{\beta+\gamma}.
 \end{aligned}
\]
Collecting $\alpha=\beta+\gamma$ gives the asserted coefficients.
For the Bernstein formula, factorial cancellation gives
\[
 \sum_{b=a}^d\frac{\binom ba}{\binom da}\binom db
       z^b(1-z)^{d-b}
 =z^a\sum_{b=a}^d\binom{d-a}{b-a}z^{b-a}(1-z)^{d-b}=z^a.
\]
Multiply by $p_{\alpha a}\mathbf x^\alpha$ and sum over $\alpha,a$.
The Bernstein weights are nonnegative and sum to
$(z+(1-z))^d=1$. Replacing every coefficient by its minimum therefore
gives the lower bound, separately for each family member and its degree.
\end{proof}

\subsection{A common radial comparison for both high-profile certificates}
\label{alg:high-profiles}

For the multipliers and corrections defined above, set
\[
 Q_0=(W_4+C_0)h_r-T_0h_{r+1},\qquad
 Q_1=(V_5+C_1)h_r-T_1h_{r+1}.
\]
\begin{lemma}\label{alg:high-zero}
For $i=0,1$ and every profile $\lambda\vdash2(r+1)$, one has
\begin{equation}\label{eq:unified-r1}
[m_\lambda]Q_i\ge0,
 \qquad r\ge8+2i,\quad\max\lambda\ge6+i.
\end{equation}
\end{lemma}
\begin{proof}
Write $L=\ell(\lambda)$, $\delta=2r+2-L$, $k=N_2(\lambda)$, and
$s=r-3-i$. Below $L=s$ all terms vanish. At $L=s$ the target and signed
multiplier product vanish and $C_i h_r$ is nonnegative. It therefore
suffices to work with $L\ge s$.

For a source marker write $d=d_\tau$ and $j=\ell(\tau)$.
The coefficients $v_{0,\tau},v_{1,\tau}$ are exactly the displayed
$W_4,V_5$ tables; omitted entries are zero.
Let $X_\tau$ be independent formal placeholders indexed by nonunit
markers, including the empty marker; evaluation means
$X_\tau=\mathsf N_\tau$. The positive normalizing factors and
normalized target flags are
\[
\begin{array}{c|l|l}
 i&\mathcal D_i&Z_i(X)\\\hline
0&6r(r+1)^2(r-1)(r-2)(L)_5
 &\displaystyle\frac{(L-r+3)_4}{(r+1)_4}X_\varnothing\\[2pt]
1&900(r-3)(r-2)(r-1)(r+1)(L)_6
 &\displaystyle\frac{(L-r+4)_4}{L(r-1)(r-2)(r-3)}X_2.
\end{array}
\]
These target entries follow from the exact coefficients
$\binom L{r+1}$ and $N_2\binom{L-1}{r-1}$ after division by
$\binom Ls$. Define the raw coefficients by the finite identity
\begin{equation}\label{eq:unified-r3}
\sum_\tau p_{i,\tau}X_\tau
 =\mathcal D_i\left\{
 \sum_{d_\tau\le4+i}v_{i,\tau}
 \frac{(L-s)_{5+i-d}}{(r+2-d)_{5+i-d}(L)_j}X_\tau
 -Z_i(X)
 +\sum_{\substack{y_{i,\tau}\ne0\\\tau\ne(6+i)}}
                 \frac{y_{i,\tau}}{(L)_j}X_\tau\right\}.
\end{equation}
Cancel and then set $L=2r+2-\delta$. All $p_{i,\tau}$ are polynomials.
For $i=0$, the denominators of the five signed layers are respectively
$r+1,r+1,(r+1)^2,(r+1)^2,(r+1)^2(r-1)$; after multiplication by the
falling-factorial denominators in \eqref{eq:unified-r3}, they divide
$r(r+1)^2(r-1)(r-2)$, using the displayed numerator factors $r+2$
in the empty layer and $r$ in the defect-two layer.
For $i=1$, the first two layers cancel their numerator factor $r$;
the remaining layer denominators are
$r+1,(r+1)(r-1),(r+1)(r-1)(r-2)$, with complementary falling factors
in $(r-3)(r-2)(r-1)(r+1)$.
The constants $6,900$ clear the correction denominators, and
$(L)_{5+i}$ clears every flag denominator. This accounts for all
cancellations in \eqref{eq:unified-r3} without coefficient fitting.

Choose one actual coordinate at least $6+i$ and delete it, leaving
$\lambda'$. Every marker in \eqref{eq:unified-r3} is satisfied at that coordinate.
Define the anchored coefficients by
\begin{equation}\label{eq:unified-r4}
\sum_\sigma b_{i,\sigma}X_\sigma
 =\sum_\tau p_{i,\tau}
   \left(X_\tau+\sum_{a=1}^{\ell(\tau)}X_{\tau\setminus\tau_a}\right).
\end{equation}
This is exact: the first term counts injections avoiding the deleted
coordinate, and the sum counts those using it in a specified slot.
Repeated equal slots are counted separately. Put $m=k-1$ and
$P_{i,\sigma}=b_{i,\sigma}(m)_{\ell(\sigma)}$.
For every surviving flag,
\begin{equation*}
0\le\frac{\mathsf N_\sigma(\lambda')}{(m)_{\ell(\sigma)}}\le1;
 \qquad \frac{\mathsf N_{2^j}(\lambda')}{(m)_j}=1.
\end{equation*}
The empty flag is included in $2^j$; flags longer than $m$ vanish.
Apply \eqref{common:positive-envelope} with
$\theta_\sigma=\mathsf N_\sigma(\lambda')/(m)_{\ell(\sigma)}$
and $I=\{2^j\}$. The resulting envelope $\mathcal E_i$ is a lower
bound for the anchored polynomial in either positive basis below.
The separated correction on $6+i$ contributes at least
\begin{equation}\label{eq:unified-r6}
H_0=r(r+1)^4(r-1)(r-2)(L-1)_4,
 \qquad
 H_1=2700(r+1)^3(r-3)(r-2)(r-1)(L-1)_5,
\end{equation}
since its anchored flag is $1+N_{6+i}(\lambda')\ge1$.
Combining this contribution with the surviving flags gives
\[
 \frac{\mathcal D_i}{\binom Ls}[m_\lambda]Q_i
 \ge H_i+\sum_\sigma\theta_\sigma P_{i,\sigma}
 \ge H_i+\mathcal E_i.
\]
Thus $H_i+\mathcal E_i\ge0$ suffices for \eqref{eq:unified-r1}.

Write the lower-defect bound as $\delta\ge k+e$; here $e=4+i$,
and the bounded-alphabet proof below will use the same chart with other $e$.
For $k\ge5+i$, the rank and support bounds place all required
parameters in the following domain:
\begin{equation}\label{eq:unified-r7}
\begin{gathered}
k=5+i+x,\quad\delta=5+i+e+x+y,\\
 r=e+x+y+z,\
x,y,z\ge0,
 x+y+z\ge\max(0,8+2i-e),\\
(x,y,z)=\bigl(\max(0,8+2i-e)+t\bigr)(u,v,w),\
t,u,v,w\ge0,\quad u+v+w=1.
\end{gathered}
\end{equation}
Indeed the inverse is $(x,y,z)=(k-5-i,\delta-k-e,r+5+i-\delta)$,
whose sum is $r-e\ge\max(0,8+2i-e)$.
Set $t=r-e-\max(0,8+2i-e)$ and divide $(x,y,z)$ by $r-e$;
when $r=e$, any simplex point works. Conversely the chart has
$r\ge8+2i$, $\delta-k-e=y\ge0$ and $r+5+i-\delta=z\ge0$.
It may include parameters that do not arise from a profile; proving
nonnegativity on this larger domain is sufficient.
For the present $e=4+i$, the radial shift is $4+i$ and
$L=1+x+y+2z$.
The substituted $P_{i,\sigma}$ has degree at most $11+2i$:
the correction weights have degree at most one or three, and their
cleared normalized products have degree at most ten; the signed terms
and target have no greater degree.
Lemma~\ref{common:bernstein-envelope} gives the exact radial coefficients
\begin{equation}\label{eq:unified-r8}
 q_{i,\sigma;h,A,B,C}
 =\operatorname{Rad}_{11+2i,\,4+i}(P_{i,\sigma})_{h,(A,B,C)},
 \quad 0\le h\le11+2i,\quad A+B+C=11+2i.
\end{equation}
The envelope coefficient is therefore exactly
\begin{equation}\label{eq:unified-r9}
e_{i;h,A,B,C}=\sum_{\sigma=2^j}q_{i,\sigma;h,A,B,C}
        +\sum_{\sigma\ne2^j}\min(0,q_{i,\sigma;h,A,B,C}).
\end{equation}
The paid polynomials have the positive radial factorizations
\begin{equation}\label{eq:unified-r10}
\begin{aligned}
H_0={}&(t+6)(t+7)(t+8)(t+9)^4(u+v+w)^7
       \prod_{a=1}^4\bigl((t+a)(u+v)+(2t+a+4)w\bigr),\\
 H_1={}&2700(t+7)(t+8)(t+9)(t+11)^3(u+v+w)^8
       \prod_{a=1}^5\bigl((t+a)(u+v)+(2t+a+5)w\bigr).
\end{aligned}
\end{equation}
For $i=0$, every coefficient $H_{0;h,A,B,C}$ in \eqref{eq:unified-r10} is positive
for $0\le h\le11$, $A+B+C=11$. The exact maximum of
$\max(0,-e_{0;h,A,B,C})/H_{0;h,A,B,C}$ over those triples is zero
for $h=0,1,2,3,4,11$. Its remaining values and attaining triples are
\[
\begin{array}{c|c|c}
 h&\text{maximum ratio}&(A,B,C)\\\hline
5&110846/6264375&(0,0,11)\\
6&5985189/143567846&(1,0,10)\\
7&24005115/330815261&(3,0,8)\\
8&1503592/14542665&(4,0,7)\\
9&82602/656075&(5,0,6)\\
10&13938/132293&(6,0,5).
\end{array}
\]
These are finite rational coefficient identities specified by
\eqref{eq:unified-r3}, \eqref{eq:unified-r4}, \eqref{eq:unified-r8}, \eqref{eq:unified-r9}, and the explicit product \eqref{eq:unified-r10}; every index
range is displayed. Each numerator multiplied by seven is smaller
than its denominator. Hence $\mathcal E_0\succeq-H_0/7$ coefficientwise.

For $i=1$, the complete negative part is even shorter:
\[
(\mathcal E_1)_-=-t^{11}\bigl(695468u^7w^6+122056u^6w^7\bigr).
\]
Every unlisted coefficient is nonnegative. For a formal series $F$,
$[z^j]F$ denotes the coefficient of $z^j$; brackets with several
variables denote the corresponding multivariate coefficient.
For each required coefficient of $H_1$, one positive term in the
binomial expansion already exceeds the corresponding loss:
\[
 \begin{aligned}
 [t^{11}u^7w^6]H_1
 &=2700[z^6](1+z)^8(1+2z)^5
 \ge2700\binom85\binom51\,2>695468,\\
 [t^{11}u^6w^7]H_1
 &=2700[z^7](1+z)^8(1+2z)^5
 \ge2700\binom86\binom51\,2>122056.
 \end{aligned}
\]
All other coefficients in \eqref{eq:unified-r10} are nonnegative. Thus $H_i+\mathcal E_i\ge0$
for both cases on the whole domain \eqref{eq:unified-r7}.

For a fixed $k$ and lower bound $\delta\ge k+e$, set
$r_0=\max(8+2i,k+e-5-i)$. The required parameters are covered by
\begin{equation}\label{eq:unified-r13}
 r=r_0+t,\quad\delta=k+e+(r_0+5+i-k-e+t)v,
 \qquad t\ge0,\quad0\le v\le1.
\end{equation}
Indeed $t=r-r_0\ge0$ and $A=r+5+i-k-e\ge0$;
for $A>0$ set $v=(\delta-k-e)/A\in[0,1]$, while for $A=0$
we have $\delta=k+e$ and any $v$ works. Conversely
$0\le\delta-k-e=Av\le A$ recovers both defect bounds.
For $m=0,\ldots,3+i$, take $k=m+1$, $e=4+i$; then
$r_0=8+2i$ and $r_0+5+i-k-e=8+2i-m$.
Discard flags longer than $m$ and substitute this chart in $P_{i,\sigma}$.
Their degree-eleven Bernstein coefficients are
\begin{equation}\label{eq:unified-r14}
 q_{i,\sigma;h,b}=\operatorname{Ber}_{11,v}(P_{i,\sigma})_{h,b},
 \qquad0\le b\le11,\quad0\le h\le11+2i.
\end{equation}
Define $e_{i;h,b}$ by the same pure-two/negative-part sum as \eqref{eq:unified-r9}, and
let $H_{i;h,b}$ be the coefficients of \eqref{eq:unified-r6} in the same basis.
They are nonnegative, because each $L$-dependent factor has positive
endpoint polynomials on \eqref{eq:unified-r13}.

For $i=0$, every $H_{0;h,b}$ is positive for $0\le h,b\le11$.
Over this rectangle the maximum negative-envelope ratio is
$\max_{h,b}\{\max(0,-e_{0;h,b})/H_{0;h,b}\}$. Its four values are:
\[
\begin{array}{c|c|c}
 m&\text{maximum ratio}&(h,b)\\\hline
0&15272/450985&(8,1)\\
1&5985971/116198226&(0,5)\\
2&162143/3792258&(0,3)\\
3&816/41539&(9,1).
\end{array}
\]
Each ratio is smaller than $1/18$. Equations \eqref{eq:unified-r3}, \eqref{eq:unified-r4}, \eqref{eq:unified-r14}, and
\eqref{eq:unified-r9}'s stated summation rule specify every coefficient in these finite
identities. Thus $\mathcal E_0\succeq-H_0/18$ in the Bernstein basis.
For $i=1$ and $m=1,2,3,4$, the complete negative-part identity is
\[
\min(0,e_{1;h,b})=0
 \quad(1\le m\le4,\ 0\le h\le13,\ 0\le b\le11).
\]
Thus $\mathcal E_1\ge0$ on those four intervals.

Only $i=1,k=1$ remains. Every surviving single-marker coefficient of
$V_5$ is positive for $r\ge10$. Moreover
\[
 \frac{[m_\lambda]T_1h_{r+1}}
      {[m_\lambda]\widetilde m_{r-4,7}h_r}
 =\frac{(L-r+4)_4}{(r-1)(r-2)(r-3)}\le r,
\]
since $0\le L-r+4\le r$ follows from support and $\delta\ge6$.
At these integers falling factorials are nonnegative and increasing.
The marker-seven weight $3(r+1)^2>r$ absorbs the target;
all other surviving terms are nonnegative.
All high profiles are now covered. Finally $\mathcal D_i/\binom Ls>0$,
so undoing the normalization proves \eqref{eq:unified-r1} for both certificates.
\end{proof}

\subsection{A common three-nonunit comparison for both bounded alphabets}
\begin{lemma}\label{alg:direct-zero-bounded}
Let $i\in\{0,1\}$, $r\ge8+2i$, and
$\lambda\vdash2(r+1)$ with largest part at most $5+i$. Then
$[m_\lambda]Q_i\ge0$.
\end{lemma}
\begin{proof}
Put $L=\ell(\lambda)$, $\delta=2r+2-L$, and $k=N_2(\lambda)$.
Below $L=r-3-i$ every term vanishes. Assume henceforth
$L\ge r-3-i\ge5+i$, so $\delta\le r+5+i$.
Use the raw coefficients $p_{i,\tau}$ and positive normalizers
$\mathcal D_i$ from the preceding high-profile lemma. Their flag
identity is obtained before choosing a high coordinate and is valid
on every profile. The separated marker $6+i$ vanishes here, giving
\begin{equation*}
 \frac{\mathcal D_i}{\binom L{r-3-i}}[m_\lambda]Q_i
       =\sum_\tau p_{i,\tau}(r,\delta)\mathsf N_\tau(\lambda).
\end{equation*}
The source tables give $p_{i,\tau}=0$ for $d_\tau<i$;
exact matching gives $[m_\lambda]Q_i=0$ whenever $\delta\le4+i$.
The following single induction includes both first elimination stages.

Use sorted marker lists in the alphabet $\{2,\ldots,5+i\}$ and define
\[
 M^{(i)}_{\alpha\nu}
 =\sum_{a=2}^{5+i}\mathbf1_{\{\operatorname{sort}(\alpha\cup\{a\})=\nu\}}
 +\sum_{j=1}^{\ell(\alpha)}\sum_{a=\alpha_j+1}^{5+i}
 \mathbf1_{\{\operatorname{sort}((\alpha\setminus(\alpha_j))\cup\{a\})=\nu\}}.
\]
Here $\alpha\setminus(\alpha_j)$ means deleting one occurrence
of the part $\alpha_j$. The sum over $j$ counts positions,
so repeated equal parts are counted separately.
For a fixed selection counted by $\mathsf N_\alpha$, an unselected
coordinate $x$ contributes $\sum_{a=2}^{5+i}\mathbf1_{x\ge a}$
to the defect; at selected slot $j$ the contribution is
$(\alpha_j-1)+\sum_{a=\alpha_j+1}^{5+i}\mathbf1_{x\ge a}$.
Summing over all selections therefore gives
\[
 (\delta-d_\alpha)\mathsf N_\alpha
 =\sum_\nu M^{(i)}_{\alpha\nu}\mathsf N_\nu.
\]
Set $c^{[i]}_{i,\tau}=p_{i,\tau}$, with absent entries zero. For $d=i,\ldots,4+i$ set
\begin{equation}\label{direct:elimination}
 q^{[d]}_{i,\alpha}=\frac{c^{[d]}_{i,\alpha}}{\delta-d}\quad(d_\alpha=d),
 \qquad
 c^{[d+1]}_{i,\nu}=\mathbf1_{\{d_\nu\ne d\}}c^{[d]}_{i,\nu}
       +\sum_{d_\alpha=d}q^{[d]}_{i,\alpha}M^{(i)}_{\alpha\nu}.
\end{equation}
For $i=0$ the first stage removes the empty flag; for $i=1$ that flag
is absent. Every quotient is polynomial. Indeed $Q_i$ matches all
profiles through defect $4+i$. After lower defects have been removed,
evaluate at the population consisting of $\alpha$ and enough units,
with $\delta=d_\alpha$. A higher-defect flag cannot fit. Another
same-defect flag cannot fit either: an injection with equal total
excess must use the same marker multiset. Consequently the remaining
coefficient of $\mathsf N_\alpha>0$ vanishes for every sufficiently
large integer $r$, hence identically on $\delta=d_\alpha$.
More explicitly, if $c^{[d]}_{i,\alpha}=\sum_b f_b(r)\delta^b$, then
\[
 q^{[d]}_{i,\alpha}
    =\sum_{b\ge1}f_b(r)\sum_{v=0}^{b-1}\delta^v d^{b-1-v}.
\]
This includes $d=0$ with $0^0=1$ and uses no division by a numerical
zero. Substituting the mass identity above proves equality of the
flag sums at every stage. Put
$a_{i,\tau}=c^{[5+i]}_{i,\tau}$; then
\begin{equation}\label{direct:eliminated}
 \frac{\mathcal D_i}{\binom L{r-3-i}}[m_\lambda]Q_i
       =\sum_\tau a_{i,\tau}(r,\delta)\mathsf N_\tau(\lambda).
\end{equation}
The two recurrences leave respectively twenty-six and fifty-three
nonzero markers. They satisfy
\[
 5+i\le d_\tau\le8+2i,\qquad 2\le\ell(\tau)\le5+i,
 \qquad|\{j:\tau_j>2\}|\le3.
\]
Indeed an eliminated marker has defect at most $4+i\le5$ and thus
at most two slots larger than two; a step appends at most one slot.
The surviving corrections have the same stated bounds.
The coefficient bounds needed below are
\begin{equation}\label{direct:degree}
 \begin{gathered}
 \deg a_{i,\tau}+\ell(\tau)\le11+2i,
 \qquad \deg_\delta a_{i,\tau}\le10+i-\ell(\tau),\\
 L-4-i\mid a_{i,\tau}\quad\hbox{when }\ell(\tau)\le4+i.
 \end{gathered}
\end{equation}
For the initial raw coefficient of length $j$, its flag denominator
leaves the factor $(L-j)_{5+i-j}$. Its signed $L$-degree is at most
$10+2i-d_\tau-j\le10+i-j$; the correction weights have degree at
most $1+2i$, and the cleared normalized products have weighted degree
ten. Division by $\delta-d$ lowers the degree bounds by one;
adjoining a slot raises the weighted bound by one, and replacement
does not increase length. Finally $L-4-i=2r-\delta-2-i$ and
$\delta-d$ are relatively prime in $\mathbb Q[r,\delta]$, so division
preserves this factor. Marker length never decreases. These observations
prove \eqref{direct:degree} inductively.

For $k\ge3$, choose a uniform three-position sample $X$ from the
nonunit coordinates. Write $\tau=(\alpha,2^b)$, where every entry of
$\alpha$ exceeds two, and put $p=\ell(\alpha)\le3$.
Apply \eqref{common:sample-flag} with $h=3$ and define
\begin{equation}\label{direct:sample-kernel}
 K_{i,X}(r,\delta,k)
 =6\sum_\tau a_{i,\tau}(r,\delta)(k)_{\ell(\tau)}
                  \frac{\mathsf N_\alpha(X)}{(3)_p},
 \qquad
 \mathbb E K_{i,X}=6\sum_\tau a_{i,\tau}\mathsf N_\tau(\lambda).
\end{equation}
For $X=(x_1,x_2,x_3)$ the flag is explicitly
$\mathsf N_\alpha(X)=\prod_{j=1}^p(|\{a:x_a\ge\alpha_j\}|-j+1)$.
It is enough to prove $K_{i,X}\ge0$ for every sample that can occur.

Let $X$ range over the twenty or thirty-five sorted triples from
$\{2,\ldots,5+i\}$, and put $e_X=\sum_{a\in X}(a-2)$.
Every occurrence satisfies $\delta\ge k+e_X$, since the other
nonunit coordinates each contribute at least one to the defect.
For $k\ge5+i$ use \eqref{eq} with $e=e_X$ and
write $c_{i,X}=\max(0,8+2i-e_X)$ for its radial shift.
Every kernel is divisible by $k(k-1)$; after division and substitution,
call the polynomial $P_{i,X}(x,y,z)$. Its degree is at most $9+2i$.
The common radial transform gives
\begin{equation}\label{direct:radial-coefficient}
 b_{i,X;h,A,B,C}
 =\operatorname{Rad}_{9+2i,\,c_{i,X}}(P_{i,X})_{h,(A,B,C)},
 \quad0\le h\le9+2i,\quad A+B+C=9+2i.
\end{equation}
Use $0^0=1$. The complete coefficient minima over all the displayed
samples and triples are
\begin{equation}\label{direct:radial-envelope}
 \min_{X,\,A+B+C=9+2i}b_{i,X;h,A,B,C}=[t^h]\mathcal B_i(t),
 \qquad0\le h\le9+2i,
\end{equation}
where the two positive polynomials are
\begin{align*}
 \mathcal B_0(t)
 &=216(t+8)\bigl\{15+(t+7)\bigl[955+(t+13)\\
 &\qquad\cdot\bigl(3+(t+1)(29+30(t+6))\bigr)\bigr]\bigr\},\\
 \mathcal B_1(t)
 &=129600\bigl\{40+(t+8)\bigl[145+(t+9)\bigl(167+(t+12)\\
 &\qquad\cdot\bigl(171+(t+4)(77+90(t+7))\bigr)\bigr)\bigr]\bigr\}.
\end{align*}
The minimum in every row is attained at $X=222$ and
$(A,B,C)=(0,0,9+2i)$. At that endpoint $k=\delta=5+i$, and only
the pure-two marker $2^{5+i}$ survives. Thus the witness polynomial
is directly $6(3+i)!\,a_{i,2^{5+i}}(8+2i+t,5+i)$, giving the
two displayed rows. Identity \eqref{direct:radial-envelope} is
obtained by the binomial theorem from the coefficient tables and
\eqref{direct:elimination}, \eqref{direct:sample-kernel},
\eqref{direct:radial-coefficient}; all its finite index ranges are
specified. In particular every coefficient is nonnegative, including
all zero rows $h\ge6$. Hence $K_{i,X}\ge0$ on the whole required domain.

Suppose $3\le k\le4+i$. Flags longer than $k$ vanish, and
\eqref{direct:degree} shows that
$J_{i,k,X}=K_{i,X}/(L-4-i)$ is polynomial of total degree at most
$8+2i$ and $\delta$-degree at most $7+i$. Its divisor is positive.
If $k+e_X\le4+i$, the kernel is identically zero: no surviving
marker can fit $X$ together with $k-3$ twos. This removes three
zero-target pairs $(k,X)$ and seven marked-target pairs.
For the remaining thirty-seven or ninety-eight pairs, use
\eqref{eq:unified-r13} with $e=e_X$; thus
$r_0=\max(8+2i,k+e_X-5-i)$.
Only the following two types of domain replacement are needed:
\[
 \begin{array}{c|c|c}
 (i,k,X)&r&\delta\\\hline
 (i,4+i,223),\ i=0,1&8+2i+t&5+i+(i+1)(i+3)v\\
 (1,3,225)&10+t&6.
 \end{array}
\]
For the first row the other $k-3=1+i$ coordinates contribute at most
$(1+i)(3+i)$ excess units; for the second row the sample is the whole
nonunit population. Thus these are genuine covers of the actual domains.
After substitution, $J_{i,k,X}$ has $t$-degree at most $8+2i$
and $v$-degree at most $7+i$. Its coefficients in the common
Bernstein basis are
\begin{equation}\label{direct:small-coefficient}
 B_{i,k,X;h,b}=\operatorname{Ber}_{7+i,v}(J_{i,k,X})_{h,b},
 \quad0\le h\le8+2i,\quad0\le b\le7+i.
\end{equation}
Thus $J_{i,k,X}=\sum_{h,b}B_{i,k,X;h,b}t^h\binom{7+i}b v^b(1-v)^{7+i-b}$.
The complete coefficient envelopes are
\begin{equation*}
 \min_{(k,X),\,0\le b\le7+i}B_{i,k,X;h,b}=[t^h]\mathcal C_i(t),
 \qquad0\le h\le8+2i,
\end{equation*}
where the minimum is over the nonzero families just defined, and
\begin{align*}
 \mathcal C_0(t)&=144(t+5)(t+8)
       \bigl\{16+(t+7)\bigl[464+t(88+133(t+6))\bigr]\bigr\},\\
 \mathcal C_1(t)&=450(11270400+24969600t+12634128t^2
                         +3441954t^3+643975t^4+92160t^5).
\end{align*}
These identities use exactly \eqref{direct:elimination},
\eqref{direct:sample-kernel}, \eqref{eq:unified-r13}, the two
domain replacements, and \eqref{direct:small-coefficient}.
For $i=0$ the minimum in every degree is attained at $(k,X,b)=(3,224,0)$.
For $i=1$, the positive minima are attained at $(5,223,8)$ for $h=0$,
at $(3,234,8)$ for $1\le h\le4$, and at $(5,223,0)$ for $h=5$;
the remaining minima are zero. Every Bernstein coefficient is therefore
nonnegative. Hence every required $J_{i,k,X}$ and $K_{i,X}$ is nonnegative.

It remains to consider $k\le2$.
If $k\le1$, or if the two nonunits $a\le b$ have $a+b-2\le4+i$,
exact matching applies. The remaining pairs are
\[
 \begin{array}{c|l}
 i& (a,b)\\\hline
 0&(2,5),(3,4),(3,5),(4,4),(4,5),(5,5)\\
 1&(2,6),(3,5),(3,6),(4,4),(4,5),(4,6),(5,5),(5,6),(6,6).
 \end{array}
\]
For each pair define the exact normalized residual polynomial
\[
 F_{i,a,b}(t)=\sum_\tau a_{i,\tau}(8+2i+t,a+b-2)
 \prod_{j=1}^{\ell(\tau)}
       \bigl(\mathbf1_{a\ge\tau_j}+\mathbf1_{b\ge\tau_j}-j+1\bigr).
\]
Flags longer than two give zero, and the degree is at most $9+2i$.
The complete coefficient envelopes over the listed pairs are
\begin{align*}
 \operatorname{Env}_t\{F_{0,a,b}\}
 ={}&12(5416320+29999104t+27613918t^2+12267147t^3\\
 &\hspace{2em}+3212395t^4+516982t^5+50436t^6+2744t^7+64t^8),\\
 \operatorname{Env}_t\{F_{1,a,b}\}
 ={}&F_{1,3,5}(t)\\
 ={}&200(t+6)(t+7)(2t+11)(2t+13)\\
 &\quad\cdot\bigl\{60+(t+10)\bigl[366+(t+8)\bigl(642+(t+9)\\
 &\qquad\cdot\bigl(12+(t+3)(225+(t+8)^2)\bigr)\bigr)\bigr]\bigr\}.
\end{align*}
For $i=0$ the minimum is attained by $(5,5)$ in degrees zero and one,
and by $(2,5)$ in all higher degrees; for $i=1$ it is always attained
by $(3,5)$. These finite identities include the zero top-degree
coefficients. All displayed coefficients and factors are positive.
Thus \eqref{direct:eliminated} is nonnegative for $k\le2$ as well.
For $k\ge3$ use \eqref{direct:sample-kernel}; in every case the factor
$\mathcal D_i/\binom L{r-3-i}$ is positive. This completes both proofs.
\end{proof}

\subsection{Scalar prices and the stable conclusion}
\label{subsec:stable-prices}
Both stable and finite prices are linear combinations of the moments $U_j$.
The following product comparison extends their endpoint estimates.
\begin{lemma}\label{v25:moment-comparison}
Let $m\in\mathbb Z_{\ge1}$, $s\in\{0,\ldots,m\}$, $2n\ge2n_0>2(m-1)$, and
$J(2n)=\sum_{j=0}^m a_jU_j(2n)$. If $a_j(j-s)\le0$ for every $j$, then
\[
 \frac{J(2n)}{U_s(2n)}\le\frac{J(2n_0)}{U_s(2n_0)}.
\]
In particular, $a_j\ge0$ for $j<m$ suffices when $s=m$.
\end{lemma}
\begin{proof}
Every factor $2n-2a$ and $2n_0-2a$ below is positive. Cancelling
factorials in $U_j(2n)=2^j\binom nj$ gives
\[
 \frac{U_j(2n)}{U_s(2n)}=
 \begin{cases}
 \displaystyle\frac{s!}{j!\prod_{a=j}^{s-1}(2n-2a)},&j<s,\\[6pt]
 1,&j=s,\\[2pt]
 \displaystyle\frac{s!}{j!}\prod_{a=s}^{j-1}(2n-2a),&j>s.
 \end{cases}
\]
For $j<s$ the ratio decreases with $2n$, and $a_j\ge0$; for $j>s$
it increases with $2n$, and $a_j\le0$. Thus, in both cases,
\[
 a_j\frac{U_j(2n)}{U_s(2n)}
 \le a_j\frac{U_j(2n_0)}{U_s(2n_0)}.
\]
At $j=s$ both sides equal $a_s$. Summing over $0\le j\le m$
proves the assertion.
\end{proof}

We apply this comparison after discarding the single internal negative
term in each stable price. This loss still leaves both required strict bounds.
\begin{lemma}\label{alg:zero-price}
For ${(2n)}\ge {(r+1)}(r+2)$, the displayed multipliers and corrections satisfy
\begin{align*}
 J_0(2n)&:=\Lam(W_4+C_0)-\frac {2n}{r+1}\Lam(e_{r+1})<0
 &&({(r+1)}\ge9),\\
 J_1(2n)&:=\Lam(V_5+C_1)-\frac{{(2n)}+1}{r+1}\Lam(T_1)<0
 &&({(r+1)}\ge11).
\end{align*}
\end{lemma}
\begin{proof}
The identities
\[
 \Lam(\widetilde m_{\ell,\tau})=U_\ell+
 \frac{\ell-\ell(\tau)}{\ell} U_{\ell-1},\qquad
 {(2n)}U_j=(j+1)U_{j+1}+2jU_j
\]
give the following correction sums:
\begin{align*}
 A_{r+1}&=\frac{2(2{(r+1)}-5)(2{(r+1)}-3)}{{(r+1)}^2(r-2)}+\sum_\tau y_{0,\tau}\\
 &=\frac{{(r+1)}^5+57{(r+1)}^4+516{(r+1)}^3-9132{(r+1)}^2+21084{(r+1)}+180}{6{(r+1)}^2(r-2)},\\
 B_{r+1}&=\sum_\tau\frac{{(r+1)}-4-\ell(\tau)}{r-3}y_{0,\tau}\\
 &=\frac{{(r+1)}^4+55{(r+1)}^3+276{(r+1)}^2-13404{(r+1)}+58284}{6{(r+1)}(r-3)},\\
 y_{\mathrm{tot}}&=\sum_\tau y_{1,\tau}\\
 &=\frac{120{(r+1)}^3+160{(r+1)}^2+451270{(r+1)}-3222000}{900},\\
 y_{\mathrm{der}}&=\sum_\tau({(r+1)}-5-\ell(\tau))y_{1,\tau}\\
 &=\frac{120{(r+1)}^4-782{(r+1)}^3+455044{(r+1)}^2-7515882{(r+1)}+30378600}{900}.
\end{align*}
With these coefficients, the signed tables and corrections give
\begin{align*}
 J_0(2n)={}&-\frac{r}{{(r+1)}^2}U_{r}
 +\frac{2{(r+1)}-3}{{(r+1)}^2}U_{r-1}
 -\frac{2(2{(r+1)}-3)}{{(r+1)}^2}U_{r-2}+A_{r+1}U_{r-3}+B_{r+1}U_{r-4},\\
 J_1(2n)={}&-\frac{r}{r+1}U_{r-1}
 +\frac{2(r)}{r+1} U_{r-2}
 -\frac{2(r)(2{(r+1)}-5)}{{(r+1)}(r-2)}U_{r-3}\\
 &+\left(y_{\mathrm{tot}}+\frac{4(r)(2{(r+1)}-5)}{{(r+1)}(r-3)}\right)U_{r-4}
 +\frac{y_{\mathrm{der}}}{r-4}U_{r-5}.
\end{align*}

All four correction coefficients are nonnegative, by their displayed
nonnegative inventories and admissible marker lengths.

Put $n_0=(r+1)(r+2)/2$ and remove the internal negative terms:
\[
 J_0^+=J_0+\frac{2(2r-1)}{(r+1)^2}U_{r-2},\qquad
 J_1^+=J_1+\frac{2r(2r-3)}{(r+1)(r-2)}U_{r-3}.
\]
Thus $J_i\le J_i^+$. Every coefficient of $J_i^+$ below its highest
moment $U_{r-i}$ is nonnegative, by the displayed inventories.
Lemma~\ref{v25:moment-comparison} gives the common reduction
\begin{equation}\label{v41:upper-price-reduction}
 \frac{J_i(2n)}{U_{r-i}(2n)}
 \le\frac{J_i^+(2n)}{U_{r-i}(2n)}
 \le\frac{J_i^+(2n_0)}{U_{r-i}(2n_0)}
 \qquad(i=0,1).
\end{equation}
The hypotheses hold because $n_0>r-i-1$ in both rank ranges.

For both endpoints set
\[
 \Pi_i=\prod_{a=0}^3\bigl((r+1)^2-(r+1)+4+2i+2a\bigr).
\]
At $2n_0=(r+1)(r+2)$ the moment-ratio formula in the preceding lemma becomes
\[
 \frac{U_{r-i-b}(2n_0)}{U_{r-i}(2n_0)}
 =\prod_{a=0}^{b-1}
   \frac{r-i-a}{(r+1)^2-(r+1)+4+2i+2a}
 \qquad(0\le b\le4).
\]
Define the numerator polynomials by
\begin{align*}
 K_0(u)={}&\frac{339}{5}+u\left[\frac45+(u-11)
       \left(\frac{12}{5}+(u-13)\left(\frac u6+\frac95\right)\right)\right],\\
 N_0(u)={}&30\left\{792+u\left[84+(u-2)
       \bigl(900+(u-3)(582+uK_0(u))\bigr)\right]\right\},\\
 N_1(11+x)={}&166600569600+19687148886x
       +3815989342(x)_2+3069689708(x)_3\\
 &+2024212535x^4+418520745x^5+43618912x^6
       +2462302x^7+70415x^8+765x^9.
\end{align*}

Substitute these four ratios and the displayed correction sums into
$J_i^+$ and multiply by the common denominators. Collection of powers gives
\begin{align*}
 \frac{J_0^+(2n_0)}{U_r(2n_0)}
 &=-\frac{rN_0(r+1)}{6(r+1)^2\Pi_0},\\
 \frac{J_1^+(2n_0)}{U_{r-1}(2n_0)}+\frac1{60}
 &=-\frac{N_1(r+1)}{900(r+1)\Pi_1}.
\end{align*}

For $u\ge13$, every term in the nested expression for $K_0(u)$
is nonnegative and its constant term is positive. The four remaining
values are
\[
 \bigl(K_0(9),K_0(10),K_0(11),K_0(12)\bigr)
 =\frac15(1347,779,383,303).
\]
Hence $K_0(u)>0$ and therefore $N_0(u)>0$ at every integer $u\ge9$.
For $N_1$, every falling factorial $(x)_a$ is nonnegative at integers
$x\ge0$, and the displayed constant term is positive. Thus
$N_1(r+1)>0$ for every integer $r+1\ge11$.
The denominators are positive. Applying \eqref{v41:upper-price-reduction}
now gives $J_0(2n)<0$ and
$J_1(2n)<-U_{r-1}(2n)/60<0$ throughout the claimed ranges.
\end{proof}

The coefficient comparisons and scalar estimates now combine through the paired contact criterion.
\begin{lemma}\label{alg:uniform}
For $2n\ge(r+1)(r+2)$, no tiling with an arbitrary centre set exists
in the zero residue class modulo $r+1$ when $r+1\ge9$, or in the
minus-one residue class when $r+1\ge11$ is odd.
\end{lemma}
\begin{proof}
For $i=0,1$, profiles with a part at least $6+i$ are covered by
Lemma~\ref{alg:high-zero}; every remaining
profile is covered directly by
Lemma~\ref{alg:direct-zero-bounded}.
Thus both stable residuals have nonnegative coefficients in their
stated rank ranges.

Apply \eqref{unified:local-master-identity} with $V=W_4+C_0$ or $V_5+C_1$,
whose stable residual is respectively $Q_0$ or $Q_1$.
For $T_0$ the defect is zero and
${(2n)}\equiv0\pmod {r+1}$ gives
$\lceil({(2n)}-{(r+1)}+1)/{(r+1)}\rceil={(2n)}/{(r+1)}$.
For $T_1$ the defect is one and ${(2n)}\equiv-1\pmod {r+1}$ gives
$\lceil({(2n)}-{(r+1)}+2)/{(r+1)}\rceil=({(2n)}+1)/{(r+1)}$.
The criterion therefore gives
$\Lam(V)-\beta\Lam(T_i)\ge0$, contradicting
the corresponding inequality in Lemma~\ref{alg:zero-price}.
\end{proof}

\subsection{Finite ranks and completion of the proof}\label{sec:finite}
The remaining pairs are $(r,\delta)=(5,0),(6,0),(7,0),(6,1),(8,1)$.
We construct the multipliers in one marker basis, prove their residuals
nonnegative on every non-all-even profile by the same two-coordinate
criterion, and bound all five prices by one moment comparison.

For a partition $\tau$ with all parts at least two, define
\begin{equation*}
 G_{r,\delta}=\frac1{D_{r,\delta}}
       \sum_{|\tau|\le r+2}b_{r,\delta,\tau}
             m_{(\tau,1^{r+2-|\tau|})}.
\end{equation*}
The complete integer coefficients are below. Every omitted marker and
every marker with $|\tau|>r+2$ has coefficient zero; hence no negative
unit multiplicity is used. Each multiplier has largest part at most six
and at most two parts greater than two.
\begin{longtable}{c|rrrrr}
$(r,\delta)$&$(5,0)$&$(6,0)$&$(7,0)$&$(6,1)$&$(8,1)$\\*
$D_{r,\delta}$&$4320$&$490$&$23040$&$35$&$315$\\*\hline
$\tau$&\multicolumn{5}{c}{$b_{r,\delta,\tau}$}\\*\hline
$\varnothing$&$5040$&$560$&$25920$&$0$&$0$\\
$2$&$1320$&$130$&$5400$&$30$&$280$\\
$3$&$1080$&$110$&$4680$&$5$&$35$\\
$4$&$570$&$68$&$3180$&$6$&$40$\\
$2^{2}$&$-540$&$-44$&$-1442$&$22$&$150$\\
$5$&$-230$&$3$&$876$&$6$&$40$\\
$3\,2$&$-1100$&$-89$&$-2617$&$12$&$80$\\
$6$&$115$&$0$&$0$&$3$&$30$\\
$4\,2$&$-875$&$-86$&$-4013$&$7$&$60$\\
$3^{2}$&$-454$&$-104$&$-6025$&$-13$&$-50$\\
$2^{3}$&$2730$&$144$&$2410$&$-27$&$-104$\\
$5\,2$&$6328$&$99$&$-1093$&$4$&$-4$\\
$4\,3$&$6688$&$0$&$0$&$0$&$-70$\\
$3\,2^{2}$&$11430$&$452$&$5054$&$-34$&$-168$\\
$6\,2$&$0$&$40$&$5741$&$44$&$0$\\
$5\,3$&$0$&$2665$&$16443$&$160$&$-123$\\
$4^{2}$&$0$&$520$&$4426$&$16$&$-140$\\
$4\,2^{2}$&$0$&$350$&$31340$&$-37$&$-254$\\
$3^{2}\,2$&$0$&$1612$&$57828$&$98$&$190$\\
$2^{4}$&$0$&$-1400$&$-6693$&$588$&$587$\\
$6\,3$&$0$&$0$&$394764$&$0$&$3425$\\
$5\,4$&$0$&$0$&$0$&$0$&$76$\\
$5\,2^{2}$&$0$&$0$&$0$&$0$&$230$\\
$4\,3\,2$&$0$&$0$&$0$&$0$&$784$\\
$3\,2^{3}$&$0$&$0$&$0$&$0$&$1760$\\
$6\,4$&$0$&$0$&$0$&$0$&$2018$\\
$6\,2^{2}$&$0$&$0$&$0$&$0$&$2298$\\
$5^{2}$&$0$&$0$&$0$&$0$&$1859$\\
$5\,3\,2$&$0$&$0$&$0$&$0$&$1639$\\
$4^{2}\,2$&$0$&$0$&$0$&$0$&$2032$\\
$4\,2^{3}$&$0$&$0$&$0$&$0$&$5440$\\
$3^{2}\,2^{2}$&$0$&$0$&$0$&$0$&$-2156$\\
$2^{5}$&$0$&$0$&$0$&$0$&$-8200$\\
\hline
\end{longtable}
\begin{lemma}\label{finite:pair-sampling}
Fix $r\ge2$, $\delta\in\{0,1\}$ and
$G=\sum_{\nu\vdash r+2}g_\nu m_\nu$, with at most two parts greater
than two in every occurring shape. Put $M=2r+2$.
For a profile $\lambda\vdash M$ let $L=\ell(\lambda)$,
$K=N_2(\lambda)\ge3$, and $E=M-L-K>0$.
Write $\nu=(\tau,2^j,1^u)$, with every entry of $\tau$ at least three,
$p=\ell(\tau)\le2$, and set
\begin{equation}\label{finite:pair-coefficient}
 \begin{aligned}
 P_\nu(a,b;L,K)
 &=\frac{(K)_{p+j}\mathsf N_\tau(a,b)}
      {(2)_p\operatorname{aut}(\tau)j!}\binom{L-p-j}{u},\\
 H_{a,b}
 &=\sum_\nu g_\nu P_\nu(a,b;L,K)
   -(1-\delta)\binom L{r+1}-\delta K\binom{L-1}{r-1}.
 \end{aligned}
\end{equation}
Set $P_\nu=0$ when $p+j>K$, before evaluating its binomial factor.
Let $\mathcal P_{L,K}$ contain the integer pairs
$a\ge b\ge2$, $a+b\le E+4$, deleting a pair only when
\[
 L=K,\qquad a,b\text{ are even},\qquad
 (K=3\ \text{or}\ a+b=E+4).
\]
If a real number $\pi_{L,K}$ satisfies
\begin{equation}\label{finite:pair-price-condition}
 H_{a,b}+\pi_{L,K}\left(a+b-4-\frac{2E}{K}\right)\ge0
 \qquad((a,b)\in\mathcal P_{L,K}),
\end{equation}
then $[m_\lambda](Gh_r-T_\delta h_{r+1})\ge0$ for every non-all-even
profile with these values of $L,K$.
\end{lemma}
\begin{proof}
Sample uniformly two of the $K$ nonunit positions, sorting their entries
as $(a,b)$. Equation~\eqref{common:sample-coefficient}, with $h=2$,
gives $\mathbb E P_\nu=[m_\lambda](m_\nu h_r)$, including zero
when $p+j>K$. The target placements number $\binom L{r+1}$ and
$K\binom{L-1}{r-1}$, respectively. Thus
\[
 \mathbb E H_{a,b}=[m_\lambda](Gh_r-T_\delta h_{r+1}),\qquad
 \mathbb E(a+b)=4+\frac{2E}{K},
\]
where the second equality is \eqref{common:sample-excess} with $h=2$.
The remaining $K-2$ nonunits have sum
$M-L+K-a-b\ge2(K-2)$, so $a+b\le E+4$.
A deleted pair would force an all-even profile: when $K=3$, its
one remaining entry is even because $M$ is even; when $a+b=E+4$,
all remaining entries equal two. In both cases $L=K$ leaves no units.
Thus every sampled pair belongs to $\mathcal P_{L,K}$. Taking expectations
in \eqref{finite:pair-price-condition} proves the claim.
\end{proof}

Put $M=2r+2$ and write $g_\nu=[m_\nu]G_{r,\delta}$, with absent
coefficients zero.
For integers $u\le x\le v$ and $0\le j\le v-u$,
\begin{equation}\label{finite:boundary-cardinal}
 \binom{x-u}{j}\binom{v-x}{v-u-j}
 =\mathbf1_{\{j=x-u\}}.
\end{equation}
Indeed, if $j>x-u$, the first factor is zero; if $j<x-u$, then
$v-u-j>v-x$ and the second factor is zero. At $j=x-u$ both
factors are one.
For each interval row below, the displayed vector $(a_0,\ldots,a_{v-u})$
and positive integer $s$ give the complete positive expansion
\[
 Q(x)=\frac1s\sum_{j=0}^{v-u}a_j
       \binom{x-u}{j}\binom{v-x}{v-u-j}\ge0.
\]
The entries are obtained from the displayed finite coefficient sums;
thus the tables include every coefficient of these expansions.

For $K=0$, the only profile is $1^M$, and its residual is
\[
 g_{1^{r+2}}\binom M{r+2}-(1-\delta)\binom M{r+1}
 =(1-\delta)\left\{\frac{r+2}{r+1}\binom M{r+2}
                         -\binom M{r+1}\right\}=0,
\]
because $\binom M{r+2}/\binom M{r+1}=(r+1)/(r+2)$.
For $K=1$, write $\lambda=(A,1^{M-A})$, $2\le A\le M-1$,
$L=M+1-A$, and $c_b=g_{(b,1^{r+2-b})}$, $c_0=g_{1^{r+2}}$.
The excluded endpoint $A=M$ is all-even. The exact residual is
\begin{equation}\label{finite:boundary-one}
 R(A)=c_0\binom L{r+2}
    +\sum_{b=2}^{\min(A,r+2)}c_b\binom{L-1}{r+2-b}
    -(1-\delta)\binom L{r+1}-\delta\binom{L-1}{r-1}.
\end{equation}
For $2\le A\le5+\delta$, no marker with two nonunit parts can fit.
Put $c_b^*=v_{\delta,(b)}/(r+3-b)$. The empty and one-marker rows give
\[
 c_0=v_{\delta,\varnothing},\qquad
 c_b=c_b^*+\varepsilon_b,\qquad
 \varepsilon_b=\frac{2}{945}\mathbf1_{\{(r,\delta,b)=(8,1,6)\}}\ge0
 \quad(2\le b\le5+\delta).
\]
Here $r+3-b\ge r-2-\delta>0$ and $r+2-b\ge0$, so every term
used is admissible, including the smallest listed radii. In this
one-marker restriction that rule reads exactly
\[
 c_A^*\binom{M-A}{r+2-A}
 =(1-\delta)\binom{M+1-A}{r+1}
   +\delta\binom{M-A}{r-1}
   -c_0\binom{M+1-A}{r+2}
   -\sum_{b=2}^{A-1}c_b^*\binom{M-A}{r+2-b}.
\]
Substitution in \eqref{finite:boundary-one} cancels the matched terms:
\[
 R(A)=\sum_{b=2}^A\varepsilon_b\binom{M-A}{r+2-b}\ge0
 \qquad(2\le A\le5+\delta).
\]
The remaining transition values are
\[
\begin{array}{c|c|r|l}
(r,\delta)&[u,v]&s&(sR(u),\ldots,sR(v))\\\hline
(5,0)&[6,11]&864&(1512,1275,488,45,0,23)\\
(6,0)&[6,8]&490&(2128,2345,1380)\\
(7,0)&[6,9]&960&(10005,13374,9675,4690)\\
(6,1)&[7,9]&35&(63,105,90)\\
(8,1)&[7,11]&63&(132,756,924,700,378)\\
\end{array}
\]
Outside the first row, the remaining range is $A\ge r+3+\delta$.
Then $L\le r-\delta$, both target binomials vanish, and every
occurring $c_b$ and $c_0$ is nonnegative. Thus every summand in
\eqref{finite:boundary-one} is nonnegative. The first row already
includes all $6\le A\le M-1$.

If $E=M-L-K=0$ and $K\ge2$, the profile is
$\lambda=(2^K,1^{M-2K})$. Its non-all-even range is $2\le K\le r$.
Choosing the $j$ two-slots and then the unit slots gives
\begin{equation*}
 B(K)=\sum_{j=0}^{\min(K,\lfloor(r+2)/2\rfloor)}
 g_{(2^j,1^{r+2-2j})}\binom Kj\binom{M-K-j}{r+2-2j}
 -(1-\delta)\binom{M-K}{r+1}
 -\delta K\binom{M-K-1}{r-1}.
\end{equation*}
Its complete positive expansions are
\[
\begin{array}{c|c|r|l}
(r,\delta)&[u,v]&s&(sB(u),\ldots,sB(v))\\\hline
(5,0)&[2,5]&36&(0,0,167,580)\\
(6,0)&[2,6]&245&(0,0,1148,2400,0)\\
(7,0)&[2,7]&23040&(93456,1308,1512,162425,262150,763)\\
(6,1)&[2,6]&7&(0,0,0,198,1044)\\
(8,1)&[2,8]&63&(0,0,3,3532,12252,18529,70)\\
\end{array}
\]

Use the formula in \eqref{finite:pair-coefficient} also at $K=2$,
writing $H_{a,b}(L)$ to display its $L$-dependence. If $E>0$ and
$\lambda=(a,b,1^{L-2})$ with $a\ge b\ge2$, $a+b=M+2-L$,
Equation~\eqref{common:sample-coefficient} with $h=k=2$ gives
\begin{equation}\label{finite:boundary-pair}
 [m_\lambda](Gh_r-T_\delta h_{r+1})=H_{a,b}(L).
\end{equation}
When $a\le5$, the possible pairs, excluding $(2,2)$ already covered
by $E=0$, are, in order,
\[
 (3,2),(3,3),(4,2),(4,3),(4,4),
 (5,2),(5,3),(5,4),(5,5).
\]
In this order, substitution of $L=M+2-a-b$ into
\eqref{finite:boundary-pair} gives the complete coefficient vectors
\[
\begin{array}{c|r|l}
(r,\delta)&s&\bigl(sH_{a,b}(M+2-a-b)\bigr)\\\hline
(5,0)&4320&(0,3876,0,7083,9720,343,4860,11292,19184)\\
(6,0)&245&(0,0,98,0,0,0,877,1055,2595)\\
(7,0)&23040&(228426,396,98406,2592,220,246,176,2393,71060)\\
(6,1)&35&(0,0,14,224,421,161,484,556,612)\\
(8,1)&315&(0,330,0,5610,7980,0,4065,5580,3719)\\
\end{array}
\]
Every entry is nonnegative. Since every marker part is at most six,
$H_{a,b}(L)=H_{6,\min(b,6)}(L)$ for $a\ge6$.
Set $x=L-2$, $h=\min(6,r-1)$, $e=\max(0,r-7)$ and
$c_*=g_{(h,3,1^e)}$. In $\nu=(\tau,1^u)$, $\tau$ lists all nonunit parts,
including twos. For $p=\ell(\tau)\in\{1,2\}$, changing the second coordinate from two to $b$ gives
\[
 \Delta_\nu=\frac{\binom{x+2-p}{u}}{\operatorname{aut}(\tau)}
                  \sum_{t\in\tau,\,t\ge3}\mathbf1_{b\ge t}\ge0.
\]
Indeed for $p=2$, $\tau=(s,t)$ with $s\ge t$, the ordered count is
$(1+\mathbf1_{b\ge s})\mathbf1_{b\ge t}=\mathbf1_{b\ge s}+\mathbf1_{b\ge t}$; subtract its value at $b=2$.
The $p=1$ case is the same single-indicator difference. Empty markers
and markers with $p>2$ have zero difference.
For every negative row, $u\ge e+2-p$. If $x<e$, then
$u>x+2-p$ and its binomial factor is zero. For $x\ge e$, the factorial identity
\[
 \frac{\binom{x+2-p}{u}}{\binom xe}
 =\frac{e!}{u!}\prod_{k=e}^{u+p-3}(x-k)
                    \prod_{k=1}^{2-p}(x+k)
\]
shows that the ratio increases wherever its numerator is nonzero; below that range the
numerator vanishes. If $b\ge t$, then $x=M-a-b\le X_t:=M-6-t$.
Consequently
\[
 \Delta_\nu\le\binom xe w_\nu,\qquad
 w_\nu=\frac1{\operatorname{aut}(\tau)}
        \sum_{t\in\tau,\,t\ge3}\frac{\binom{X_t+2-p}{u}}{\binom{X_t}{e}},
 \qquad N=\sum_{g_\nu<0}(-g_\nu)w_\nu.
\]
Here only rows with nonzero difference enter $N$; all their
$X_t\ge e$, so the displayed denominators are positive.
With the positive integer scales $s$ shown below, the full arithmetic,
in the descending order of the negative marker rows, is
\[
\begin{array}{c|r|l|r}
(r,\delta)&s&sN&sc_*\\\hline
(5,0)&4320&230+2\cdot875+3\cdot454+3\cdot1100=6642&6688\\
(6,0)&490&6\cdot86+10\cdot104+10\cdot89=2446&2665\\
(7,0)&23040&10\cdot1093+20\cdot4013+35\cdot6025+35\cdot2617=393660&394764\\
(6,1)&35&10\cdot13=130&160\\
(8,1)&945&7\cdot369+5\cdot12+7\cdot210+49\cdot70+14\cdot150=9643&10275
\end{array}
\]
Thus $c_*>N$. The positive $(h,3,1^e)$ contribution is
$c_*\binom xe(\mathbf1_{b\ge h}+\mathbf1_{b\ge 3})\ge c_*\binom xe$ when $b\ge3$.
All other positive contributions can be discarded, so the same bound proves
\[
 H_{6,b}(L)-H_{6,2}(L)=\sum_\nu g_\nu\Delta_\nu\ge(c_*-N)\binom xe\ge0\quad(3\le b\le6,\ 0\le x\le M-6-b).
\]
It remains only to check $H_{6,2}(L)$. For $a\ge6$, $b=2$ the range is
$2\le L\le M-6$. Its complete positive expansions, including $L=2$, are
\[
\begin{array}{c|c|r|l}
(r,\delta)&[u,v]&s&(sH_{6,2}(u),\ldots,sH_{6,2}(v))\\\hline
(5,0)&[2,6]&4320&(6443,5453,3703,2303,1763)\\
(6,0)&[2,8]&490&(40,139,155,70,0,75,159)\\
(7,0)&[2,10]&23040&(0,5741,10389,10807,5297,18,7844,30302,42030)\\
(6,1)&[2,8]&35&(44,51,74,137,227,289,281)\\
(8,1)&[2,12]&315&(0,0,0,26,234,1070,3025,5950,8456,8484,6690)\\
\end{array}
\]
Together with the small-alphabet table, this proves all $K=2$ boundary cases.

Assume $K\ge3$, $E=M-L-K>0$. For $3\le t\le s\le6$ put
\begin{equation}\label{finite:cross-coefficients}
 B_{s,t}=\sum_{\nu=(s,t,2^j,1^u)}
 g_\nu\frac{(K)_{2+j}}{2\operatorname{aut}(s,t)j!}
                  \binom{L-2-j}{u},\qquad
 I_x(s)=\mathbf1_{\{x\ge s\}}.
\end{equation}
As in \eqref{finite:pair-coefficient}, a summand is zero when $2+j>K$.
Write $B^+=\max(B,0)$ and $B^-=\min(B,0)$, and define
\begin{equation*}
 \psi_x=B_{3,3}+\sum_{s=4}^6 B_{s,3}I_x(s)
              +\sum_{4\le t\le s\le6}B_{s,t}^-I_x(s).
\end{equation*}
For $a,b\ge3$, the empty and one-high marker contributions cancel in
$C_{a,b}:=H_{a,b}-H_{a,2}-H_{b,2}+H_{2,2}$: their counts have
respective differences $1-1-1+1=0$ and
$(I_a(s)+I_b(s))-I_a(s)-I_b(s)=0$.
The target cancels by the same first identity. Two-high counts are
\[
 \mathsf N_{(s,t)}(a,b)=I_a(s)I_b(t)+I_b(s)I_a(t),
 \quad
 \mathsf N_{(3,3)}(a,b)=2,
 \quad
 \mathsf N_{(s,3)}(a,b)=I_a(s)+I_b(s).
\]
The coefficient formula therefore gives the exact positive decomposition
\begin{align*}
 C_{a,b}-\psi_a-\psi_b
 &=\sum_{4\le t\le s\le6}\Bigl[
 B_{s,t}^+\{I_a(s)I_b(t)+I_b(s)I_a(t)\}\notag\\
 &\hspace{25mm}
 +(-B_{s,t}^-)\{I_a(s)(1-I_b(t))+I_b(s)(1-I_a(t))\}\Bigr]
 \ge0.
\end{align*}
Indeed $B=B^++B^-$, and expanding the two braces recovers
$B\mathsf N_{(s,t)}-B^-(I_a(s)+I_b(s))$ term by term.
All factors in the displayed decomposition are nonnegative.

Define $d=d_{r,\delta,L,K}$ as follows.
Take $d=0$ for the three zero targets. For the minus targets take
$d=0$ except at the following entries; each vector is indexed by the
listed $(L,K)$ pairs in the same order:
\[
\begin{array}{c|c|l}
 r&\text{quantity}&\text{ordered entries}\\\hline
 6&(L,K)&(3,3),(4,3),(4,4),(5,3),(5,4),(6,3),(6,4),(7,4)\\
  &630d&(2240,2940,10584,3294,8640,2040,7128,6345)\\
 8&(L,K)&(4,4),(5,3),(5,4),(5,5),(6,3),(7,3),(8,3),(9,3)\\
  &630d&(44100,4200,14175,154288,3780,9872,9730,2340)
\end{array}
\]
Thus $d\ge0$ everywhere. Put
\[
 \pi=\frac{K(H_{2,2}-d)}{2E},\qquad
 R_{a,b}=H_{a,b}+\pi\left(a+b-4-\frac{2E}{K}\right).
\]
Then $R_{2,2}=H_{2,2}-(H_{2,2}-d)=d$. Define
\begin{equation}\label{finite:common-single}
 S_a=R_{a,2}+\psi_a-\frac d2,\qquad
 Z_a=\min(R_{a,2},S_a)
     =R_{a,2}+\min\left(\psi_a-\frac d2,0\right).
\end{equation}
The affine price terms cancel in the mixed difference, so
\[
 R_{a,b}-R_{a,2}-R_{b,2}+d=C_{a,b},\qquad
 R_{a,b}=S_a+S_b+(C_{a,b}-\psi_a-\psi_b)\ge S_a+S_b
 \quad(a,b\ge3).
\]
Thus $Z_a\ge0$ proves both $R_{a,2}\ge0$ and $S_a\ge0$.
Every permitted double-high pair satisfies $a,b\le E+1$, since
$a+b\le E+4$ and $a,b\ge3$. It consequently suffices to prove
\begin{equation}\label{finite:common-reduced-domain}
 Z_a\ge0\quad(3\le a\le E+1),\qquad
 R_{E+2,2}\ge0\quad(L>K).
\end{equation}
These also cover all pairs with one entry two, except $(2,2)$,
whose value is $d\ge0$. At the omitted boundary $L=K$, the pair
$(E+2,2)$ forces all other $K-2$ entries to equal two; moreover
$E+2=M-2K+2$ is even. The full profile is all-even and is excluded.

Every multiplier coefficient containing a six is nonnegative.
The increment when a first coordinate rises from five to six is
\[
 \Delta H:=H_{6,2}-H_{5,2}
 =\sum_{\nu=(6,2^j,1^u)}g_\nu
       \frac{(K)_{1+j}}{2j!}\binom{L-1-j}{u}\ge0.
\]
Terms with $1+j>K$ are zero. The same coefficient signs imply
$B_{6,t}\ge0$ for every $3\le t\le6$. Hence
\[
 \psi_6-\psi_5=B_{6,3}+\sum_{t=4}^6B_{6,t}^-=B_{6,3}\ge0,
\]
and \eqref{finite:common-single} yields
\[
 \begin{aligned}
 R_{6,2}-\pi&=R_{5,2}+\Delta H,\\
 S_6-\pi&=S_5+\Delta H+B_{6,3},\\
 Z_6-\pi&=\min(R_{5,2}+\Delta H,S_5+\Delta H+B_{6,3})
             \ge\min(R_{5,2},S_5)=Z_5.
 \end{aligned}
\]
Put $\Delta_Z=Z_6-Z_5-\pi\ge0$ and $A=E+1$.
For $a\ge6$, all marker thresholds have saturated, so
$H_{a,2}=H_{6,2}$ and $\psi_a=\psi_6$. Therefore
\[
 Z_a=Z_6+\pi(a-6)=Z_5+\pi(a-5)+\Delta_Z.
\]
For $6\le a\le A$, eliminating $\pi$ between this identity and its
value at $A$ gives
\[
 (A-5)Z_a=(A-a)Z_5+(a-5)Z_A+(A-a)\Delta_Z\ge0
 \quad\text{whenever }Z_5,Z_A\ge0.
\]
Since $A-5>0$, only $a=3,4,5$ and the endpoint $a=A\ge6$
are needed in the first inequality of \eqref{finite:common-reduced-domain}.

The one price formula above applies to every
population, with no separate pair verification at these entries.
In the tables below $a=A$ denotes $A=E+1$ and $a=\partial$ denotes
the ordinary boundary. Set
\[
\begin{array}{c|c|c|c}
 a&\mathcal F(E)&e&v\\\hline
 3,4,5&2E Z_a/K&a-1&M-2K\\
 A&2E Z_{E+1}/K&5&M-2K\\
 \partial&2E R_{E+2,2}/K&1&M-2K-1
\end{array}
 \qquad L=M-K-E.
\]
Since $L=M-K-E\ge K$, the possible excesses satisfy
$1\le E\le M-2K$; at the ordinary boundary $L>K$, the upper
bound is $M-2K-1$. The condition $a\le E+1$ gives $E\ge a-1$,
and the endpoint $a=E+1\ge6$ gives $E\ge5$. These are exactly
the intervals in the preceding table.
For each listed $K$, the positive scale $s$ and full integer vector
$\mathbf z=(z_0,\ldots,z_{v-e})$ give, at every integer $e\le E\le v$,
\begin{equation*}
 s\mathcal F(E)=\sum_{j=0}^{v-e}z_j
             \binom{E-e}{j}\binom{v-E}{v-e-j}\ge0.
\end{equation*}
This is \eqref{finite:boundary-cardinal}, with the explicitly listed
values $z_j=s\mathcal F(e+j)$. Substitution in
\eqref{finite:pair-coefficient}, \eqref{finite:cross-coefficients}
and \eqref{finite:common-single} gives the complete vectors below.
Every nonempty interval with $K\ge3$ occurs. Thus the tables and the
positive interpolation identity prove \eqref{finite:common-reduced-domain},
and Lemma~\ref{finite:pair-sampling} proves all five coefficient comparisons.

\begin{longtable}{ccr >{\raggedright\arraybackslash}p{109mm}}
$a$&$K$&$s$&$\mathbf z\quad((r,\delta)=(5,0))$\\*\cline{1-4}\noalign{\nobreak}
$3$&$3$&$360$&$(0,\allowbreak566,\allowbreak1836,\allowbreak3695,\allowbreak5488)$\\
$3$&$4$&$1440$&$(18364,\allowbreak29627,\allowbreak43144)$\\
$3$&$5$&$15$&$(544)$\\
$4$&$3$&$144$&$(323,\allowbreak648,\allowbreak1232,\allowbreak1936)$\\
$4$&$4$&$288$&$(6455,\allowbreak8472)$\\
$5$&$3$&$1080$&$(17096,\allowbreak24880,\allowbreak33159)$\\
$5$&$4$&$60$&$(2933)$\\
$A$&$3$&$4320$&$(103975,\allowbreak134016)$\\
$\partial$&$3$&$4320$&$(0,\allowbreak3120,\allowbreak42108,\allowbreak74064,\allowbreak106705)$\\
$\partial$&$4$&$1440$&$(11750,\allowbreak23900,\allowbreak56079)$\\
$\partial$&$5$&$2$&$(51)$\\
\hline
\multicolumn{4}{c}{$ (r,\delta)=(6,0)$}\\*\cline{1-4}\noalign{\nobreak}
$3$&$3$&$245$&$(0,\allowbreak930,\allowbreak1428,\allowbreak2035,\allowbreak2250,\allowbreak1582,\allowbreak0)$\\
$3$&$4$&$490$&$(8180,\allowbreak11008,\allowbreak12504,\allowbreak11401,\allowbreak6736)$\\
$3$&$5$&$245$&$(7748,\allowbreak8673,\allowbreak7356)$\\
$3$&$6$&$49$&$(1076)$\\
$4$&$3$&$245$&$(570,\allowbreak780,\allowbreak1330,\allowbreak2028,\allowbreak2205,\allowbreak1400)$\\
$4$&$4$&$490$&$(10244,\allowbreak11936,\allowbreak12297,\allowbreak10088)$\\
$4$&$5$&$245$&$(7677,\allowbreak6848)$\\
$5$&$3$&$490$&$(5864,\allowbreak7070,\allowbreak7836,\allowbreak7203,\allowbreak4384)$\\
$5$&$4$&$245$&$(8436,\allowbreak7994,\allowbreak6135)$\\
$5$&$5$&$5$&$(132)$\\
$A$&$3$&$490$&$(7770,\allowbreak8604,\allowbreak7763,\allowbreak5024)$\\
$A$&$4$&$245$&$(7882,\allowbreak5095)$\\
$\partial$&$3$&$490$&$(0,\allowbreak0,\allowbreak3849,\allowbreak6288,\allowbreak8070,\allowbreak8748,\allowbreak7763)$\\
$\partial$&$4$&$490$&$(4911,\allowbreak7980,\allowbreak15645,\allowbreak17416,\allowbreak14940)$\\
$\partial$&$5$&$245/2$&$(2685,\allowbreak3614,\allowbreak4269)$\\
$\partial$&$6$&$7$&$(108)$\\
\hline
\multicolumn{4}{c}{$ (r,\delta)=(7,0)$}\\*\cline{1-4}\noalign{\nobreak}
$3$&$3$&$11520$&$(294,\allowbreak23158,\allowbreak77,\allowbreak7627,\allowbreak29159,\allowbreak35953,\allowbreak20216,\allowbreak0,\allowbreak0)$\\
$3$&$4$&$23040$&$(136634,\allowbreak220058,\allowbreak282736,\allowbreak291814,\allowbreak222064,\allowbreak99441,\allowbreak0)$\\
$3$&$5$&$23040$&$(485455,\allowbreak561520,\allowbreak518721,\allowbreak359690,\allowbreak148479)$\\
$3$&$6$&$2304$&$(64377,\allowbreak59747,\allowbreak40155)$\\
$3$&$7$&$23040$&$(366413)$\\
$4$&$3$&$11520$&$(102911,\allowbreak123254,\allowbreak160679,\allowbreak224356,\allowbreak264994,\allowbreak238832,\allowbreak141030,\allowbreak0)$\\
$4$&$4$&$23040$&$(1151566,\allowbreak1483412,\allowbreak1675493,\allowbreak1640822,\allowbreak1324755,\allowbreak752160)$\\
$4$&$5$&$11520$&$(1321960,\allowbreak1471593,\allowbreak1416160,\allowbreak1137591)$\\
$4$&$6$&$4608$&$(736709,\allowbreak740288)$\\
$5$&$3$&$11520$&$(106839,\allowbreak151006,\allowbreak200853,\allowbreak235623,\allowbreak216104,\allowbreak131193,\allowbreak0)$\\
$5$&$4$&$7680$&$(475196,\allowbreak529024,\allowbreak515104,\allowbreak418445,\allowbreak241976)$\\
$5$&$5$&$23040$&$(2802451,\allowbreak2680230,\allowbreak2168277)$\\
$5$&$6$&$256$&$(38169)$\\
$A$&$3$&$11520$&$(318888,\allowbreak378397,\allowbreak399986,\allowbreak353888,\allowbreak234531,\allowbreak57410)$\\
$A$&$4$&$23040$&$(2027476,\allowbreak1951172,\allowbreak1596939,\allowbreak1001496)$\\
$A$&$5$&$23040$&$(3096680,\allowbreak2514699)$\\
$\partial$&$3$&$11520$&$(426,\allowbreak608,\allowbreak510,\allowbreak220344,\allowbreak314540,\allowbreak381054,\allowbreak401191,\allowbreak353888,\allowbreak234531)$\\
$\partial$&$4$&$23040$&$(36947,\allowbreak718636,\allowbreak1097952,\allowbreak1855216,\allowbreak2037305,\allowbreak1947426,\allowbreak1590246)$\\
$\partial$&$5$&$11520$&$(163377,\allowbreak1010435,\allowbreak1281870,\allowbreak1621658,\allowbreak1526925)$\\
$\partial$&$6$&$768$&$(17224,\allowbreak102291,\allowbreak114300)$\\
$\partial$&$7$&$11520$&$(147439)$\\
\hline
\multicolumn{4}{c}{$ (r,\delta)=(6,1)$}\\*\cline{1-4}\noalign{\nobreak}
$3$&$3$&$945$&$(729,\allowbreak2673,\allowbreak972,\allowbreak1365,\allowbreak6642,\allowbreak9744,\allowbreak14560)$\\
$3$&$4$&$140/3$&$(24,\allowbreak3,\allowbreak32,\allowbreak480,\allowbreak1536)$\\
$3$&$5$&$35/2$&$(454,\allowbreak465,\allowbreak615)$\\
$3$&$6$&$7$&$(1154)$\\
$4$&$3$&$945$&$(10611,\allowbreak11772,\allowbreak5460,\allowbreak2754,\allowbreak987,\allowbreak3208)$\\
$4$&$4$&$140$&$(387,\allowbreak256,\allowbreak600,\allowbreak2448)$\\
$4$&$5$&$35/3$&$(1011,\allowbreak1124)$\\
$5$&$3$&$945$&$(22464,\allowbreak13980,\allowbreak5184,\allowbreak735,\allowbreak1576)$\\
$5$&$4$&$35/2$&$(178,\allowbreak225,\allowbreak396)$\\
$5$&$5$&$7$&$(1266)$\\
$A$&$3$&$945/2$&$(11805,\allowbreak5508,\allowbreak2919,\allowbreak3896)$\\
$A$&$4$&$35/6$&$(200,\allowbreak273)$\\
$\partial$&$3$&$315$&$(1071,\allowbreak2484,\allowbreak5427,\allowbreak8928,\allowbreak5770,\allowbreak1782,\allowbreak476)$\\
$\partial$&$4$&$140/3$&$(48,\allowbreak312,\allowbreak219,\allowbreak1152,\allowbreak1600)$\\
$\partial$&$5$&$35/12$&$(80,\allowbreak227,\allowbreak465)$\\
$\partial$&$6$&$7$&$(1084)$\\
\hline
\multicolumn{4}{c}{$ (r,\delta)=(8,1)$}\\*\cline{1-4}\noalign{\nobreak}
$3$&$3$&$945/4$&$(360,\allowbreak7830,\allowbreak8106,\allowbreak3150,\allowbreak45,\allowbreak6545,\allowbreak8360,\allowbreak5211,\allowbreak5100,\allowbreak0,\allowbreak0)$\\
$3$&$4$&$2520$&$(76544,\allowbreak108920,\allowbreak117160,\allowbreak137360,\allowbreak175456,\allowbreak196400,\allowbreak95384,\allowbreak5787,\allowbreak5880)$\\
$3$&$5$&$1575$&$(276695,\allowbreak378725,\allowbreak465050,\allowbreak445650,\allowbreak322185,\allowbreak94635,\allowbreak192)$\\
$3$&$6$&$63$&$(31739,\allowbreak34911,\allowbreak31374,\allowbreak19942,\allowbreak27)$\\
$3$&$7$&$63$&$(40516,\allowbreak34814,\allowbreak15295)$\\
$3$&$8$&$9$&$(1060)$\\
$4$&$3$&$945$&$(3978,\allowbreak67392,\allowbreak56490,\allowbreak8460,\allowbreak8785,\allowbreak15232,\allowbreak810,\allowbreak21780,\allowbreak0,\allowbreak0)$\\
$4$&$4$&$1260$&$(68696,\allowbreak81608,\allowbreak85160,\allowbreak111256,\allowbreak159496,\allowbreak207896,\allowbreak302643,\allowbreak320500)$\\
$4$&$5$&$1575$&$(514090,\allowbreak657700,\allowbreak802225,\allowbreak933570,\allowbreak1014850,\allowbreak1124352)$\\
$4$&$6$&$63/2$&$(29292,\allowbreak34644,\allowbreak38617,\allowbreak40281)$\\
$4$&$7$&$63$&$(90463,\allowbreak101342)$\\
$5$&$3$&$945$&$(192,\allowbreak68940,\allowbreak49680,\allowbreak29575,\allowbreak26032,\allowbreak13986,\allowbreak29040,\allowbreak7590,\allowbreak0)$\\
$5$&$4$&$1260$&$(165436,\allowbreak192240,\allowbreak212784,\allowbreak245516,\allowbreak276228,\allowbreak323541,\allowbreak421600)$\\
$5$&$5$&$1575$&$(881150,\allowbreak1008400,\allowbreak1099395,\allowbreak1132585,\allowbreak1154944)$\\
$5$&$6$&$63$&$(77558,\allowbreak81501,\allowbreak80715)$\\
$5$&$7$&$7$&$(9949)$\\
$A$&$3$&$945$&$(135690,\allowbreak91044,\allowbreak46963,\allowbreak31504,\allowbreak51732,\allowbreak61080,\allowbreak83424,\allowbreak82728)$\\
$A$&$4$&$1260$&$(360000,\allowbreak412096,\allowbreak467288,\allowbreak511028,\allowbreak431055,\allowbreak168160)$\\
$A$&$5$&$1575$&$(1400025,\allowbreak1522185,\allowbreak1541440,\allowbreak784)$\\
$A$&$6$&$63$&$(99958,\allowbreak92313)$\\
$\partial$&$3$&$945$&$(0,\allowbreak180,\allowbreak270,\allowbreak73512,\allowbreak83190,\allowbreak72144,\allowbreak29323,\allowbreak16384,\allowbreak46062,\allowbreak54780,\allowbreak83424)$\\
$\partial$&$4$&$1260$&$(48,\allowbreak656,\allowbreak75828,\allowbreak272880,\allowbreak347880,\allowbreak416424,\allowbreak472668,\allowbreak513376,\allowbreak402705)$\\
$\partial$&$5$&$315$&$(33028,\allowbreak71430,\allowbreak141555,\allowbreak244632,\allowbreak284140,\allowbreak304002,\allowbreak303023)$\\
$\partial$&$6$&$63$&$(22325,\allowbreak46414,\allowbreak69780,\allowbreak95500,\allowbreak94985)$\\
$\partial$&$7$&$63$&$(37235,\allowbreak74540,\allowbreak87792)$\\
$\partial$&$8$&$9$&$(3580)$\\
\hline
\end{longtable}

The paired scalar prices are
\[
 J_{r,\delta}(2n)=\Lam(G_{r,\delta})
      -\frac{2n+\delta}{r+1}\Lam(T_\delta).
\]

For a multiplier shape $\nu$ of length $\ell$, with $u=m_1(\nu)$ and
$a_\nu=\prod_t m_t(\nu)!$, the orbit and unit-deletion rules give
\[
 \Lam(m_\nu)=\frac{\ell!}{a_\nu}U_\ell
    +\frac{u(\ell-1)!}{a_\nu}U_{\ell-1},\qquad
 (2n)U_j=(j+1)U_{j+1}+2jU_j.
\]
Write $\Lam(T_\delta)=\sum_jt_jU_j$; its nonzero coefficients are
\[
 (t_r,t_{r+1})=(1,1)\quad(\delta=0),\qquad
 (t_{r-1},t_r)=(r-1,r)\quad(\delta=1).
\]
Using $G_{r,\delta}=\sum_\nu g_\nu m_\nu$, we obtain
\[
 D_{r,\delta}J_{r,\delta}=\sum_j A_jU_j,
 \qquad
 \frac{A_j}{D_{r,\delta}}
 =\sum_{\ell(\nu)=j}\frac{j!g_\nu}{a_\nu}
   +\sum_{\ell(\nu)=j+1}\frac{m_1(\nu)j!g_\nu}{a_\nu}
   -\frac{j t_{j-1}+(2j+\delta)t_j}{r+1}.
\]
In the target term, $U_{j-1}$ contributes $j t_{j-1}U_j$ through
multiplication by $2n$, while $U_j$ contributes $(2j+\delta)t_jU_j$.
This accounts for the last numerator above. Substitution of the five
multiplier tables gives the following integer moment coefficients;
unlisted $A_j$ are zero:
\[
\begin{array}{c|r|rrrrrrr}
(r,\delta)&D_{r,\delta}&A_1&A_2&A_3&A_4&A_5&A_6&A_7\\\hline
(5,0)&4320&115&23598&24828&1080&-600&0&0\\
(6,0)&490&0&6128&7017&2160&110&-60&0\\
(7,0)&23040&0&838322&407328&1030053&-30315&-4780&-42\\
(6,1)&35&0&438&135&270&-30&0&0\\
(8,1)&315&0&12897&48066&33588&22545&575&-280
\end{array}
\]
In the table's row order, choose
\[
 (n_0)=(21,28,36,31,49),\qquad(s)=(4,5,4,4,6).
\]
At each endpoint, normalize by $U_s(2n_0)>0$ and use
\[
 \frac{U_{j+1}(2n_0)}{U_j(2n_0)}
 =\frac{2(n_0-j)}{j+1}.
\]
The products of these ratios avoid expanding the moments themselves.
For example, the $U_3$ and $U_5$ terms in the first row satisfy
\[
 \frac{24828}{4320}\frac{U_3(42)}{U_4(42)}
 =\frac{24828}{4320}\frac4{36}<\frac{16}{25},\qquad
 \frac{600}{4320}\frac{U_5(42)}{U_4(42)}
 =\frac{600}{4320}\frac{34}{5}=\frac{17}{18}.
\]
Bounding each positive summand upwards and each negative magnitude
downwards in the same way gives the following estimates; the terms
are listed in increasing moment order:
\begin{align*}
\frac{J_{5,0}(42)}{U_4(42)}
 &<\frac1{1000}+\frac1{20}+\frac{16}{25}+\frac14-\frac{17}{18}
   =-\frac{31}{9000},\\
\frac{J_{6,0}(56)}{U_5(56)}
 &<\frac1{100}+\frac3{25}+\frac12+\frac14-\frac9{10}
   =-\frac1{50},\\
\frac{J_{7,0}(72)}{U_4(72)}
 &<\frac1{10}+\frac{11}{10}+\frac{179}{4}
       -\frac{84}{5}-\frac{137}{5}-2=-\frac14,\\
\frac{J_{6,1}(62)}{U_4(62)}
 &<\frac1{20}+\frac3{10}+\frac{31}{4}-9=-\frac9{10},\\
\frac{J_{8,1}(98)}{U_6(98)}
 &<\frac1{1000}+\frac3{100}+\frac12+5+2-10
   =-\frac{2469}{1000}.
\end{align*}

In the same row order, the largest nonzero moment indices satisfy
\[
 (m)=(5,6,7,5,7),\qquad
 (n_0-m+1)=(17,23,30,27,43)>0.
\]
Thus $n_0>m-1$, and the table gives $A_j(j-s)\le0$.
Apply Lemma~\ref{v25:moment-comparison} to $D_{r,\delta}J_{r,\delta}$
and divide by $D_{r,\delta}>0$ to obtain
\[
 J_{r,\delta}(2n)
 \le\frac{U_s(2n)}{U_s(2n_0)}J_{r,\delta}(2n_0)<0
 \qquad(n\ge n_0).
\]

\begin{lemma}\label{finite:thm:finite}
There is no tiling satisfying $2n\ge(r+1)(r+2)$ and
$2n\equiv-\delta\pmod{r+1}$ for any of the five pairs
\[
 (r+1,\delta)\in\{(6,0),(7,0),(8,0),(7,1),(9,1)\}.
\]
\end{lemma}
\begin{proof}
For the zero cases the hypotheses imply
$n\ge(r+1)(r+2)/2=21,28,36$, respectively.
For the minus cases, the first even integer at least $(r+1)(r+2)$
which is congruent to $-1$ modulo the odd integer $r+1$ is
$(r+1)(r+3)-1$. Thus $n\ge31,49$, respectively.
All five scalar prices just proved are therefore strictly negative.
The coefficient inequalities proved above and
\eqref{unified:local-master-identity}, applied
with $V=G_{r,\delta}$ and $T=T_\delta$, instead give
$J_{r,\delta}(2n)\ge0$. This is a contradiction.
\end{proof}

We now combine the dimension reduction, the small-radius arguments,
and the two algebraic ranges.

\begin{proof}[Proof of Theorem~\ref{thm:main}]
Suppose a tiling exists for $n\ge3$ and $r\ge2$. For $r=2$,
Lemma~\ref{radius2:thm:radius-two} is a contradiction. For $r=3$,
Lemma~\ref{unified:r3-complete} is a contradiction. For $r=4$,
Lemma~\ref{small:exclusions} is a contradiction. Hence $r\ge5$, and
Lemma~\ref{foundation:thm:main-nonresidue} implies
$2n\ge {(r+1)}(r+2)$ and ${(2n)}\equiv0,-1\pmod {r+1}$.
In the minus-one class, $r+1\mid2n+1$, so $r+1$ is odd.

For ${(r+1)}=6,7,8$, and for the minus-one class at ${(r+1)}=9$, apply
Lemma~\ref{finite:thm:finite}. For the zero class at every
${(r+1)}\ge9$ and the physically possible minus-one class at every
${(r+1)}\ge11$, apply Lemma~\ref{alg:uniform}.
These ranges exhaust all
$r\ge2$. Every argument applies to the arbitrary centre set in
\eqref{eq:tiling}; no lattice or periodicity hypothesis is introduced.
\end{proof}

\section*{Acknowledgments}
The authors used ChatGPT to assist with exploring proof strategies,
checking intermediate calculations, and editing the exposition.
The authors
take full responsibility for the
content of this work.

\end{document}